\documentclass[12pt,a4paper,reqno]{amsart}
\usepackage{amscd}
\usepackage{amssymb}
\usepackage{amsthm}
\usepackage[centering,text={14.5cm,22cm}]{geometry}
\usepackage{graphicx}
\usepackage{color}
\usepackage[all]{xy}
\usepackage{mathrsfs}
\usepackage{marvosym}
\usepackage{stmaryrd}
\usepackage{srcltx}
\usepackage{scalerel,amssymb}
\usepackage{tikz-cd}
\usepackage{adjustbox}
\usepackage{placeins}
\usepackage{floatrow}
\newfloatcommand{capbtabbox}{table}[][\FBwidth]
\definecolor{amaranth}{rgb}{0.9, 0.17, 0.31}
\usepackage[colorlinks=true,citecolor=amaranth,linkcolor=black]{hyperref}%

\definecolor{shadecolor}{rgb}{1,0.9,0.7}

\newcommand\restr[2]{{
  \left.\kern-\nulldelimiterspace 
  #1 
  \vphantom{\big|} 
  \right|_{#2} 
  }}

\let\oldtocsection=\tocsection
 
\let\oldtocsubsection=\tocsubsection
 
\let\oldtocsubsubsection=\tocsubsubsection
 
\renewcommand{\tocsection}[2]{\hspace{0em}\oldtocsection{#1}{#2}}
\renewcommand{\tocsubsection}[2]{\hspace{1em}\oldtocsubsection{#1}{#2}}
\renewcommand{\tocsubsubsection}[2]{\hspace{2em}\oldtocsubsubsection{#1}{#2}}

\newtheorem{theorem}{Theorem}[section]
\newtheorem{lemma}[theorem]{Lemma}
\newtheorem{proposition}[theorem]{Proposition}
\newtheorem{corollary}[theorem]{Corollary}

\newtheoremstyle{defstyle}
  {.6em} 
  {.1em} 
  {} 
  {} 
  {\bfseries} 
  {.} 
  {.5em} 
  {} 
\theoremstyle{defstyle} \newtheorem{definition}[theorem]{Definition}

\newtheorem{example}[theorem]{Example}

\theoremstyle{remark}
\newtheorem{remark}[theorem]{Remark}

\numberwithin{equation}{section}
\numberwithin{figure}{section}

\newcommand{\CC} {\mathbb{C}}
\newcommand{\FF} {\mathbb{F}}
\newcommand{\PP} {\mathbb{P}}
\renewcommand{\AA} {\mathbb{A}}

\newcommand{\LL} {\mathbb{L}}

\newcommand{\cO} {\mathcal{O}}

\newcommand {\shN}  {\mathcal{N}}

\newcommand {\shR}  {\mathcal{R}}

\newcommand {\shP}  {\mathcal{P}}

\newcommand {\shX}  {\mathcal{X}}

\newcommand {\Aut}  {\operatorname{Aut}}

\newcommand {\ev}  {\operatorname{ev}}

\newcommand {\rel} {\operatorname{rel}}

\newcommand{\C} {\mathrm{C}}

\newcommand {\prim} {\operatorname{prim}}

\newcommand {\Spec} {\operatorname{Spec}}

\newcommand {\virt} {\mathrm{virt}}

\newcommand {\Z} {\mathfrak Z}

\newcommand {\V} {\mathscr{V}}

\newcommand {\bP} {\mathbf{P}}

\def\ev{\mathrm {ev}}

\def\virt{\mathrm{vir}}

\def\PP{\mathbb{P}}

\def\cO{\mathcal{O}}
\def\oM{\overline{\mathcal{M}}}

\def\C{{\mathcal{C}}}

\def\Z{\mathbb{Z}}

\def\C{\mathbb{C}}
\def\Q{\mathbb{Q}}

\def\Aut{{\rm Aut}}

\def\E{\mathrm{E}}
\def\n{\mathrm{n}}

\def\V{\mathrm{V}}
\def\H{\mathrm{H}}
\def\g{\mathrm{g}}
\def\G{\mathsf{G}}

\def\ooMM{{\mathfrak{M}}}

\def\v{\mathsf{v}}

\def\mydate{\ifcase\month \or January\or February\or March\or
April\or May\or June\or July\or August\or September\or October\or 
November\or December\fi \space\number\day,\space\number\year}

\newtheoremstyle{cited}%
  {3pt}
  {3pt}
  {\itshape}
  {}
  {\bfseries}
  {.}
  {.3em}
  {\thmname{#1} \thmnumber{#2}\thmnote{\normalfont#3}}

\theoremstyle{cited}
\newtheorem{citedthm}{Theorem}

\theoremstyle{cited}

 \makeatletter
  \def\title@font{\Large\bfseries}
  \let\ltx@maketitle\@maketitle
  \def\@maketitle{\bgroup%
    \let\ltx@title\@title%
    \ltx@maketitle%
  \egroup}
\makeatother

\begin{document}

\title[Gromov--Witten Theory of Complete Intersections]{\resizebox{\textwidth}{!}{Gromov--Witten Theory of Complete Intersections} via Nodal Invariants}

\author[H.\,Arg\"uz]{H\"ulya Arg\"uz}
\address{University of Georgia, Department of Mathematics, Athens, GA 30605}
\email{Hulya.Arguz@uga.edu}

\author[P.\,Bousseau]{Pierrick Bousseau}
\address{University of Georgia, Department of Mathematics, Athens, GA 30605}
\email{Pierrick.Bousseau@uga.edu}

\author[R.\,Pandharipande]{Rahul Pandharipande}
\address{Departement Mathematik, ETH Z\"urich \\ 8092 Z\"urich\\ Switzerland}
\email{rahul@math.ethz.ch}

\author[D.\,Zvonkine]{Dimitri Zvonkine}
\address{Universit\'e de Versailles Saint-Quentin-en-Yvelines, Versailles \\78000, France}
\email{dimitri.zvonkine@uvsq.fr  }

\date{December 2022}

\begin{abstract}
We provide an inductive algorithm computing Gromov--Witten invariants in all genera with arbitrary insertions of all smooth complete intersections in projective space. We also prove that all Gromov--Witten classes of all smooth complete intersections in projective space belong to the tautological ring of the moduli space of stable curves. The main idea is to show that invariants with insertions of primitive cohomology classes are controlled by their monodromy and by invariants defined without primitive insertions but with imposed nodes in the domain curve. To compute these nodal Gromov--Witten invariants, we introduce the new notion of nodal relative Gromov--Witten invariants. We  then prove a nodal degeneration formula and a relative splitting formula. These results for nodal relative Gromov--Witten theory are stated in complete generality and are of independent interest.

\end{abstract}

\maketitle

\tableofcontents
\setcounter{section}{-1}
\section{Introduction}

\subsection{Overview}
Let $X$ be a smooth projective variety over $\CC$. The
\emph{Gromov--Witten invariants} of $X$ are rational numbers defined by intersection theory on the moduli spaces of stable maps to $X$. For all $$g , n \in \Z_{\geq 0}\, , \ \ \beta \in H_2(X,\Z)\, , \ \  
\alpha_1, \dots, \alpha_n \in H^\star(X)\, ,\ \ \text{and}  \ \ k_1, \dots, k_n \in \Z_{\geq 0}\, ,$$ 
there is an associated
Gromov--Witten invariant
\begin{equation} \label{eq_gw}
\left\langle \prod_{i=1}^n \tau_{k_i}(\alpha_i) \right\rangle_{g,n,\beta}^X
:= \int_{[\oM_{g,n,\beta}(X)]^\virt} \prod_{i=1}^n \psi_i^{k_i} \ev_i^{*}(\alpha_i) \in \Q\,,
\end{equation}
where the integration is over the virtual fundamental class of the moduli space of genus $g$, $n$-pointed  
stable maps to $X$ of class $\beta$, $\ev_i$ is the evaluation morphism at the marked point $i$, and 
$\psi_i$ is the first Chern class of the cotangent line bundle at the marked point $i$. We refer the reader
to \cite{CK, FuP, KM} for an introduction to stable maps and Gromov-Witten theory.

If $2g-2+n>0$, \emph{Gromov--Witten classes} of $X$ are 
defined in the cohomology of the moduli space $\oM_{g,n}$ of genus $g$, $n$-pointed stable curves: 
\begin{equation} \label{eq_gw_class}
\left[\prod_{i=1}^n \tau_{k_i}(\alpha_i)\right]^X_{g,n,\beta} 
:= \pi_{*} \left( \prod_{i=1}^n \psi_i^{k_i} \ev_i^{*}(\alpha_i) 
\cap [\oM_{g,n,\beta}(X)]^\virt \right) \in H^{\star}(\oM_{g,n}) \, ,\end{equation}
where the push-forward is along the forgetful morphism 
$$\pi \colon \oM_{g,n,\beta}(X) \rightarrow \oM_{g,n}\, .$$
Taken all together, the Gromov-Witten classes define the \emph{Cohomological Field Theory} associated to $X$.

Though the Gromov--Witten invariants of a variety $X$ are of great interest, providing an effective algorithm to compute them is generally challenging. So far, such algorithms have been
obtained only in a handful of cases. For instance, the Gromov--Witten invariants of a point are determined by Witten's conjecture \cite{W}, proved by Kontsevich \cite{MR1171758}. 
The Gromov--Witten invariants of projective spaces, or more generally of homogeneous varieties, can be computed using the localization formula \cite{GP} and the calculation of Hodge integrals on the moduli spaces of curves \cite{FP}. Gromov--Witten invariants of curves have been computed by Okounkov--Pandharipande using degeneration techniques, monodromy constraints, and Hurwitz theory 
\cite{OP2, OP1, OP}. An algorithm{\footnote{See \cite{CGL,guo2018structure} for more recent progress related to the holomorphic anomaly equation.}} for the Gromov--Witten invariants
of the quintic 3-fold hypersurface  in $\PP^4$ has been given by Maulik--Pandharipande, based on degeneration techniques and a systematic study of 
universal relations in relative Gromov--Witten theory \cite{MP}. The
degeneration scheme was used to prove the GW/DT correspondence for the quintic 3-fold in \cite{PP}.

The main result of this paper is an algorithm computing all the Gromov--Witten invariants of all smooth complete intersections in projective space.
Moreover, this algorithm can be lifted to the level of Gromov--Witten classes, and
we prove, as a corollary, that all the Gromov--Witten classes of all complete intersections in projective space are elements of the tautological ring of 
$\oM_{g,n}$.

The rest of the introduction is organized as follows. In 
\S \ref{sec_intro_primitive}, we review the main obstruction to extending previously known techniques to the case of general complete intersections: the existence of primitive cohomology classes, which are not obtained by restriction of cohomology classes of the ambient projective space. In 
\S \ref{section_intro_monodromy}, we describe a new idea to overcome this obstacle: trading primitive cohomology classes for nodes in the domain curves. 
To combine this idea with existing degeneration techniques, we develop new foundational results 
in \emph{nodal Gromov--Witten theory}: we present in \S \ref{sec_intro_nodal_relative}\,-\,\ref{sec_intro_nodal_relative_splitting} a definition of nodal relative Gromov--Witten invariants, a degeneration formula, and a splitting formula. We summarize in \S \ref{sec_intro_algorithm} how these new tools are used to address the original question of computing Gromov--Witten invariants of complete intersections. Finally, we explain in \S \ref{sec_intro_tautological},
as an the application of the proof, that the Gromov--Witten classes of complete intersections are tautological.

\subsection{Complete intersections and primitive cohomology}
\label{sec_intro_primitive}

Let $X \subset \PP^{m+r}$ be a complex $m$-dimensional smooth complete intersection of $r$ hypersurfaces of degrees $(d_1,\dots,d_r)$ in the complex projective space $\PP^{m+r}$. 
According to the Lefschetz hyperplane theorem, the restriction map in cohomology 
\[ H^i(\PP^{m+r}) \rightarrow H^i(X)\] 
is an isomorphism for $i \neq m$ and $0 \leq i \leq 2m$, and we have a decomposition 
\[ H^m(X)=H^m(\PP^{m+r}) \oplus H^m(X)_{\prim} \]
where $H^m(X)_{\prim}$ is the primitive cohomology of $X$.
A class $\alpha \in H^{*}(X)$ is \emph{simple} if $\alpha$ lies in the image of 
$H^{*}(\PP^{m+r})$ and \emph{primitive} if $\alpha \in H^m(X)_{\prim}$.
Following \cite[\S 0.5.2]{MP}
we say that a Gromov--Witten invariant of $X$
\[ \left\langle \prod_{i=1}^n \tau_{k_i}(\alpha_i) \right\rangle_{g,n,\beta}^X \]
as in \eqref{eq_gw} is \emph{simple} if all the cohomology insertions 
$\alpha_i$ are simple.

A variety of methods have so far been used to study simple Gromov--Witten invariants of $X$. For instance, genus $0$
simple Gromov--Witten invariants can be expressed as twisted Gromov--Witten invariants of 
$\PP^{m+r}$ and then computed by torus localization \cite{givental, kontsevich1995enumeration, LianLiuYau, Z0CI}. This approach has recently been extended to study also higher genus{\footnote{See also \cite{popa, Z1H} for genus $1$.}}
Gromov--Witten invariants \cite{chang2016effective,fan2019quantum,  guo2018structure}.
Another powerful approach is to consider a degeneration of $X$ to a singular subvariety of $\PP^{m+r}$. As simple cohomology classes on $X$ extend to the total space of the degeneration, we can apply the degeneration formula  
\cite{JunLi}.

Unfortunately, these tools cannot be applied to study Gromov--Witten invariants
with primitive insertions. Indeed, by definition, primitive classes do not come from the ambient projective space $\PP^{m+r}$, and so it is unclear how to directly reduce Gromov--Witten invariants of $X$ with primitive insertions to some theory of $\PP^{m+r}$\footnote{As described in \cite[Remark 1.2]{fan2021quantum}, the approach of \cite{fan2019quantum} to quantum Lefschetz by localization on a master space can be applied to primitive insertions if a degree bound is satisfied, but not in general.}. 
We cannot directly use the degeneration formula either: if we degenerate $X$ to a singular subvariety of $\PP^{m+r}$, then primitive cohomology classes often do not extend to the total space of the degeneration as the monodromy action around the special fiber typically acts non-trivially on them. 

Due to the absence of tools, studying Gromov--Witten invariants of 
complete intersections of dimension strictly larger than 1 with primitive insertions has been challenging: the only results so far have been for genus $0$ invariants.
Genus 0 results
for general complete intersections \cite{MR1664668, MR1736987,hu2015computing}
restricted to either $2$  or $4$ primitive insertions
have been
obtained by studying the monodromy constraints and the WDVV equation.
Another approach by
\cite{Gathmann, Manolache} also
computes genus 0 Gromov-Witten
invariants with a restricted number of primitive insertions.
Specific examples such as cubic hypersurfaces and 
complete intersections of two quadrics
have been studied in more detail \cite{hu2015computing,hu2021computing}.
Primitive insertions in genus 0 also play a role in
the Tevelev degree calculations of complete intersections \cite{BuchP,LianP}.

The poor understanding of primitive insertions has also limited our understanding
of Gromov--Witten invariants with simple insertions. 
To apply the degeneration formula,
we consider a degeneration of $X$ into a union of two smooth varieties glued transversally along a smooth divisor $D$. 
Following~\cite{MP}, we select one of the hypersurface equations $f$ defining~$X$ and a product of polynomials $g_1g_2$ of positive degrees such that 
$$\deg(f)= \deg (g_1) + \deg (g_2).$$ 
The family of varieties obtained
by replacing $f$ with $tf + g_1g_2$ is degeneration of $X$ over $t\in \AA^1$ into a union $X_1 \cup_D X_2$ of two complete intersections. However, the total space of this family is, in general, not smooth. To 
desingularize the total space, we blow-up along $X_2$. We then obtain a degeneration with smooth total space, and with general fiber $X$ and special fiber $X_1 \cup_D \widetilde{X}_2$, 
\begin{equation*}
X \rightsquigarrow 
X_1 \cup_D \widetilde{X}_2\, ,
\end{equation*}
where $X_1$ and $D$ are complete intersections and $\widetilde{X}_2$ is a blow-up of a complete intersection along a complete intersection. 

In the degeneration formula, the gluing condition for curves in $X_1$ and $\widetilde{X}_2$ along $D$ involves the insertion of the cohomology class 
\begin{equation} \label{eq:diagonal}
    [\Delta]= \sum_j \delta_j \otimes \delta_j^{\vee} \, \in H^\star(D\times D)\,
\end{equation}
Poincaré dual to the diagonal 
$\Delta\subset D \times D$. In the K\"unneth decomposition \eqref{eq:diagonal}, 
$(\delta_j)_j$ is a basis of $H^{\star}(D)$, and 
$(\delta_j^\vee)_j$ is the dual basis with respect to the Poincaré pairing. 
All  of the cohomology of $D$ (including the primitive cohomology) appears
in \eqref{eq:diagonal}.
Therefore, although we can apply the degeneration formula to compute simple Gromov--Witten invariants of $X$, the result is in general expressed in terms of relative Gromov--Witten invariants of $X_1$
and $\widetilde{X}_2$ with insertions of primitive cohomology classes of $D$ \cite{li2001stable}. In sufficiently  low dimensions,  the primitive cohomology can sometimes be avoided, as shown in 
\cite{MP} for the quintic 3-fold, but in general, it cannot.

Our first step to compute the Gromov--Witten invariants of complete intersections with primitive insertions in all genera is to reduce the problem to computing nodal invariants of complete intersections with only simple insertions.

\subsection{Nodes and monodromy} 
\label{section_intro_monodromy}
A key idea of the paper is to trade primitive insertions against nodes in the domain curves by inverting the relations given by the splitting axiom for invariants with imposed nodes. 

For every graph $\Gamma$ decorated by genus and curve class data on its vertices, there is a moduli space $\oM_\Gamma(X)$ parameterizing stable maps to $X$ together with the data of a 
contraction
of the dual graph of the domain curve to $\Gamma$. In other words, edges of $\Gamma$ correspond to nodes imposed on the domain curve. The \emph{nodal Gromov--Witten invariants} of $X$ are defined by integration over
$\oM_\Gamma(X)$: 
\begin{equation} \label{eq_nodal_gw_intro}
\left\langle \prod_{i=1}^n \tau_{k_i}(\alpha_i)
\prod_h \tau_{k_h}\right\rangle_{\Gamma}^X
:= \int_{[\oM_{\Gamma}(X)]^\virt} \prod_{i=1}^n \psi_i^{k_i} \ev_i^{*}(\alpha_i) \prod_h \psi_h^{k_h} \,,
\end{equation}
where the second product runs over the marked points $h$ on the normalization of the domain curve created by the splitting of the imposed nodes.{\footnote{As before, $\psi_h$ is the first Chern classes of the corresponding cotangent line bundle, and $k_h\in \mathbb{Z}_{\geq 0}$.}} 

The nodal Gromov--Witten invariants can be computed in terms of ordinary Gromov--Witten invariants by the splitting axiom. Imposing a node in the domain curve is equivalent to requiring 
the two marked points obtained by normalizing the node to map to the same point of the target $X$.
The condition can be easily written as the pull-back of the diagonal class
$[\Delta]$ in $X \times X$ under the evaluation map at the two marked points to $X \times X$. 
For example, if $\Gamma$ is the graph with a single vertex decorated by 
$g,n,\beta$ and a single loop, then
\begin{equation} \label{eq_intro_splitting}
\left\langle \left( \prod_{i=1}^{n} \tau_{k_i}(\alpha_i) \right) 
\tau_{k_{h_1}} \tau_{k_{h_2}}\right\rangle_\Gamma^X 
=\sum_j \left\langle \left( \prod_{i=1}^{n} \tau_{k_i}(\alpha_i)\right) \tau_{k_{h_1}}(\delta_j) \tau_{k_{h_2}}(\delta_j^\vee) \right\rangle_{g,n+2,\beta}^X  \,, \end{equation}
where we have used the Künneth decomposition \eqref{eq:diagonal}  of $[\Delta]$.

The basis $(\delta_j)_j$ of $H^{\star}(X)$ may be chosen to contain 
a basis of the primitive cohomology of $X$. So even if we start with \emph{simple nodal Gromov--Witten invariants} of $X$, that is, nodal Gromov--Witten invariants with only simple insertions,
the splitting axiom expresses these invariants in terms of Gromov--Witten invariants with possibly primitive insertions. The first 
non-trivial result of our paper is that these relations given by the splitting axiom can be inverted: Gromov--Witten 
invariants with possibly primitive insertions can be reconstructed from simple nodal Gromov--Witten invariants.

\vspace{8pt}
\begin{citedthm}
\label{thm_intro_1}
Let $X$ be a complete intersection in projective space which is not a cubic surface or an even dimensional complete intersection of two quadrics.
Then, the Gromov--Witten invariants of $X$ can be effectively reconstructed from the simple nodal Gromov--Witten invariants of $X$.
\end{citedthm}

\vspace{8pt}
We give a brief sketch of the proof here (for details see sections \S \ref{sec:gw_ci1} and \S \ref{sec:gw_ci2}):

\vspace{8pt}
\noindent $\bullet$
The first step is to use the constraints given by the monodromy action on the primitive cohomology.
Let $$U \subset \prod_{i=1}^r \PP(H^0(\PP^{m+r},\cO(d_i)))$$ be the open subset parameterizing $m$-dimensional smooth complete intersections of degrees 
$(d_1,\dots,d_r)$ in $\PP^{m+r}$, and let $u \in U$ be the point corresponding to $X$. The fundamental group $\pi_1(U,u)$ of $U$ based at $u$ acts on the primitive cohomology 
$H^\star(X)_{\prim}$ and the (algebraic) \emph{monodromy group} is
by definition the Zariski closure of the image of $\pi_1(U,u)$ in 
the algebraic group $\mathrm{Aut} (H^\star(X)_{\prim})$ of linear automorphisms of $H^\star(X)_{\prim}$.
As the Gromov--Witten invariants are deformation invariant, they are invariant under the action of the monodromy group on the insertions. For all complete intersections (except for the exceptions in the statement of Theorem \ref{thm_intro_1}), the monodromy group is as large as possible \cite{ MR855873, weil2, deligne1973groupes}: it is the group of linear automorphisms of 
$H^\star(X)_{\prim}$ preserving the non-degenerate bilinear form induced by the Poincaré pairing
(an orthogonal group if $m$ is even and a symplectic group if $m$ is odd).
Using the theory of invariants of orthogonal/symplectic groups, we express Gromov--Witten invariants as linear combinations with coefficients to be determined in a basis of invariant tensors.

\vspace{8pt}
\noindent $\bullet$
The second step of the proof of Theorem  \ref{thm_intro_1}  is to show that these coefficients of invariant tensors are solutions of a system of linear equations obtained by applying the splitting axiom to simple nodal Gromov--Witten invariants, as in \eqref{eq_intro_splitting}. We then prove that this system of equations is invertible. Concretely, we have to show that an explicit matrix defined in terms of the combinatorics of ways to form $n$ pairs out of $2n$ objects is invertible. This matrix has appeared previously in mathematics\footnote{The closest appearance is related to the computation of the tautological ring in cohomology of a fixed smooth curve of genus $g$. 
In this context, invertibility implies the Gorenstein property, see \cite{Tavakol}.}, in particular in the theory of zonal symmetric functions \cite[\S VII.2]{MR553598},
and we apply existing results in representation theory to prove invertibility (other
proofs can be found in the literature).

\subsection{Nodal relative Gromov--Witten theory: degeneration}
\label{sec_intro_nodal_relative}
Theorem \ref{thm_intro_1} is clearly progress: it 
allows us to trade Gromov--Witten invariants with possibly primitive insertions, which cannot be used in the degeneration formula, against nodal Gromov--Witten invariants  with only simple insertions.
To compute such nodal Gromov--Witten invariants, we first provide an extension of the degeneration techniques in Gromov--Witten theory to the nodal setting. 

In the usual setting, the degeneration formula of Jun Li \cite{JunLi} is formulated using relative Gromov--Witten invariants, based on the moduli space of relative stable maps \cite{li2001stable} (motivated by earlier work in symplectic geometry by \cite{IonelParker, LiRuan}). Given a smooth projective variety $Y$ over 
$\CC$ and a smooth connected divisor $D\subset Y$, the moduli space of relative stable maps is a (virtual) compactification of the moduli space of stable maps to $Y$ with fixed tangency conditions along $D$. In general, a relative stable map is a map from a curve to an expanded target $Y[l]$
for some $l \in \Z_{\geq 0}$. Here, 
$Y[l]$ is the $l$-step expanded degeneration of $Y$ along $D$ obtained by attaching to $Y$ along $D$ a chain of $l$ copies of the $\PP^1$-bundle 
$\PP(N_{D|Y} \oplus \cO_D)$, where $N_{D|Y}$ is the normal line bundle to $D$ in $Y$. 

To extend the degeneration formula to the nodal setting, for every graph $\Gamma$ decorated by genus, curve classes, and contact orders, we define in 
\S \ref{subsec: Moduli spaces of nodal relative stable maps} a moduli space 
$\shP_\Gamma(Y,D)$ parameterizing relative stable maps to $(Y,D)$, with the data of a 
contraction
of the dual graph of the domain curve to $\Gamma$, and with the extra condition that the nodes of the domain curves imposed by the edges of $\Gamma$ are \emph{not} mapped to the singular locus of the expanded targets $Y[l]$. We prove that $\shP_\Gamma(Y,D)$ is a proper Deligne-Mumford stack which admits a virtual class $[\shP_\Gamma(Y,D)]^\virt$. We define 
the \emph{nodal relative Gromov--Witten invariants} of $(Y,D)$ by integration against the virtual class $[\shP_\Gamma(Y,D)]^\virt$. These are the correct invariants for the degeneration formula for nodal Gromov--Witten invariants.

\vspace{8pt}
\begin{citedthm}
\label{thm_intro_degeneration}
\hspace{-5pt}\footnote{As in the usual degeneration formula of Jun Li, the precise statement is in general more complicated when the monodromy around $0 \in B$ acts non-trivially on $H_2(W_{t\neq 0})$ \cite{liu2004degeneration}: in complete generality, the degeneration formula computes finite sums of Gromov--Witten invariants of the general fiber, and not individual invariants. See Theorem \ref{thm_degeneration_cycle} for details.}
Let $W \rightarrow B$ be a projective family with smooth total space over a smooth connected curve $B$ with
 a distinguished point $0 \in B$ such that 
 \begin{enumerate}
     \item [(i)]
the fibers $W_t$ over $t \in B\setminus\{0\}$
 are smooth varieties,
 \item[(ii)]
 the fiber $W_0$ over $0 \in B$ is the union of two smooth irreducible components
 $Y_1$ and $Y_2$ glued along a smooth connected divisor $D$. 
 \end{enumerate}
 Then, there is a degeneration formula expressing nodal Gromov--Witten invariants of the general fiber $W_{t \neq 0}$ with insertions in the image of the restriction map 
 $$H^\star(W) \rightarrow H^\star(W_{t\neq 0})$$
in terms of the nodal relative Gromov--Witten invariants of $(Y_1,D)$ and $(Y_2,D)$.
\end{citedthm}
\vspace{8pt}

The proof of Theorem \ref{thm_intro_degeneration}
is presented in \S \ref{section_degeneration} and consists in starting from the usual case of Jun Li's degeneration formula, and studying carefully the effect of imposing nodes on both sides of it. 

\subsection{Nodal relative Gromov--Witten theory: splitting}
\label{sec_intro_nodal_relative_splitting}
Finally, in order to make practical use of the nodal degeneration formula given by 
Theorem \ref{thm_intro_degeneration}, an efficient way to compute nodal relative Gromov--Witten invariants is needed. As already reviewed in \S \ref{section_intro_monodromy}, nodal Gromov--Witten invariants can be reduced to ordinary Gromov--Witten invariants using the splitting axiom. We prove a similar result for nodal relative Gromov--Witten invariants.

\vspace{8pt}
\begin{citedthm}
\label{thm_intro_splitting}
Let $Y$ be a smooth projective variety over $\C$ and $D\subset Y$ a smooth divisor.
Then, the nodal relative Gromov--Witten invariants of $(Y,D)$ can be effectively reconstructed from the Gromov--Witten invariants of $Y$, the Gromov--Witten invariants of $D$, and the restriction map $H^\star(Y) \rightarrow H^\star(D)$. 
\end{citedthm}
\vspace{8pt}

The proof of Theorem \ref{thm_intro_splitting} is given in \S \ref{section_splitting}. 
Unlike what happens in the absolute case, splitting a node is not simply equivalent to the insertion of the class of diagonal. Indeed, in the relative case, there is a correction term 
coming from the possibility for the node to fall into the divisor $D$ and then forcing the target to expand. The correction term is expressed in terms of rubber Gromov--Witten invariants of the 
$\PP^1$-bundle $\PP(N_{D|Y} \oplus \cO_D)$. In other words, we obtain a splitting formula expressing nodal relative Gromov--Witten invariants of $(Y,D)$ in terms of relative Gromov--Witten invariants of $(Y,D)$ and rubber Gromov-Witten invariants of $\PP(N_{D|Y} \oplus \cO_D)$, see Theorem \ref{thm_splitting_3}.
Finally, by \cite{MP}, the relative Gromov--Witten invariants of $(Y,D)$ and the rubber  invariants of 
$\PP(N_{D|Y} \oplus \cO_D)$ can be effectively reconstructed from the Gromov--Witten theory of 
$Y$ and the Gromov--Witten theory of $D$.

\subsection{The algorithm for complete intersections}
\label{sec_intro_algorithm}

Our main result is the following.

\begin{citedthm}
\label{thm_intro_main}
Let $X$ be an $m$-dimensional smooth complete intersection in $\PP^{m+r}$ of degrees $(d_1,\dots, d_r)$. 
Then, for every decomposition 
$$d_r=d_{r,1}+d_{r,2}\ \ \ \text{with} \ \ \ d_{r,1}, d_{r,2} \in \Z_{\geq 1}\, ,$$
the Gromov--Witten invariants of $X$ can be effectively 
reconstructed from:
\begin{itemize}
\item[(i)] the Gromov--Witten invariants of an $m$-dimensional smooth complete intersection $X_1 \subset \PP^{m+r}$ of degrees $(d_1,\dots, d_{r-1},d_{r,1})$,
\item[(ii)] the Gromov--Witten invariants of an $m$-dimensional smooth complete intersection $X_2 \subset \PP^{m+r}$ of degrees $(d_1,\dots,d_{r-1}, d_{r,2})$,
\item[(iii)] the Gromov--Witten invariants of an 
$(m-1)$-dimensional smooth complete intersection $D \subset \PP^{m+r}$ of degrees $(d_1,\dots,d_{r-1}, d_{r,1}, d_{r,2})$,
\item[(iv)] the Gromov--Witten invariants of an 
$(m-2)$-dimensional smooth complete intersection 
$Z \subset \PP^{m+r}$ of degrees $(d_1,\dots, d_{r-1},d_r,d_{r,1}, d_{r,2})$.
\end{itemize}
\end{citedthm}
\vspace{8pt}

To prove Theorem \ref{thm_intro_main}, we consider, following 
\cite[\S 0.5.4]{MP}, a degeneration with general fiber $X$ and special fiber $X_1 \cup_D \widetilde{X}_2$
obtained by factoring the degree $d_r$ defining equation of $X$ 
into factors of degrees $d_{r,1}$ and $d_{r,2}$,
\begin{equation}\label{deggg}
X \rightsquigarrow 
X_1 \cup_D \widetilde{X}_2\, .
\end{equation}
Here,
$\widetilde{X}_2$ is the blow-up of $X_2$ along $Z$, and 
$X_1 \cup_D \widetilde{X}_2$ denotes the union of $X_1$ and $\widetilde{X}_2$ transversally glued along a copy of $D$. 

By 
Theorem \ref{thm_intro_1}, when $X$ is not a cubic surface or an even dimensional complete intersection of two quadrics, the Gromov--Witten invariants of $X$ are determined by the simple nodal Gromov--Witten invariants of $X$. 
By the nodal degeneration formula of Theorem \ref{thm_intro_degeneration}
applied to the degeneration \eqref{deggg} 
the simple nodal Gromov--Witten invariants of $X$ can be computed in terms of the nodal relative Gromov--Witten invariants of $(X_1,D)$ and $(\widetilde{X}_2,D)$.
By the splitting formula of Theorem \ref{thm_intro_splitting}, these nodal relative Gromov--Witten invariants can be reconstructed from the Gromov--Witten invariants of $X_1$, $D$, and $\widetilde{X}_2$.
Finally, Gromov--Witten invariants of 
$\widetilde{X}_2$ are determined by the Gromov--Witten invariants of $X_2$ and $Z$
by the blow-up result{\footnote{The more basic
blow-up result of  \cite[Lemma 1]{MP} could also be used (including the Gromov-Witten invariants also
of $D$).}} of \cite[Theorem B]{fan2017chern}.

For the special cases when $X$ is a cubic surface or an even dimensional complete intersection of two quadrics, we show by a direct topological study
that there is actually no monodromy acting on the cohomology in the degeneration \eqref{deggg}
and so the usual degeneration formula can be applied.

As Theorem \ref{thm_intro_main} computes Gromov--Witten invariants of a complete intersection in terms of Gromov--Witten invariants of complete intersections of either smaller degree or smaller dimension, it can be used 
recursively to compute Gromov--Witten invariants of all complete intersections in terms of Gromov--Witten invariants of projective spaces, which are known by localization \cite{GP}.


\subsection{Tautological classes} \label{sec_intro_tautological}
For every $g, n \in \Z_{\geq 0}$ such that $2g-2+n>0$, the moduli space $\oM_{g,n}$ of $n$-pointed genus $g$ connected stable curves is a smooth proper Deligne--Mumford stack of dimension $3g-3+n$. In particular, one can use Poincaré duality to identify $H_\star(\oM_{g,n})$
with $H^{2(3g-3+n)-\star}(\oM_{g,n})$ and define push-forward maps in cohomology. 

The (cohomological) \emph{tautological rings} $RH^\star(\oM_{g,n})$ are most compactly defined 
\cite{FP1} as the smallest system of subrings (with unit) of the cohomology rings 
$H^{\star}(\oM_{g,n})$ stable under push-forward and pull-back by the maps 
\begin{itemize}
    \item[(i)] $\oM_{g,n+1} \rightarrow \oM_{g,n}$ forgetting one of the markings,
    \item[(ii)] $\oM_{g_1,n_1+1} \times \oM_{g_2,n_2+1} \rightarrow \oM_{g_1+g_2,n_1+n_2}$ gluing two curves at a point, 
    \item[(iii)] $\oM_{g-1,n+2} \rightarrow \oM_{g,n}$ gluing together two markings of a curve. 
\end{itemize}
Elements of the tautological rings are referred to as \emph{tautological classes}.
For example, $\psi$, $\kappa$, and $\lambda$ classes are tautological \cite{FP, Mumford}.
We refer to \cite{FP-tautandnontautcoh,Pcalculus} for a review of the great amount of recent activity on the structure of the tautological rings. 

For every smooth projective variety $X$, the Gromov--Witten classes 
\eqref{eq_gw_class} are elements of $H^{\star}(\oM_{g,n})$, and so it is natural to ask whether these classes are tautological \cite{FP1}. There are currently no known counterexamples to the bold conjecture that Gromov--Witten classes are always tautological. As the tautological rings are generated by classes of algebraic cycles, the conjecture implies, in particular, that Gromov--Witten classes in odd cohomological degree or, more generally, of Hodge type $(p,q)$ with $p \neq q$ should vanish. Prior to the present work, there were only two main families of varieties whose
Gromov--Witten classes were known to be tautological
(and nontrivial): 
\begin{itemize}
    \item[(i)] Toric and homogeneous varieties $G/P$, for which the result follows from the localization formula \cite{GP} and the fact that $\lambda$ classes are tautological \cite{Mumford}.
    \item[(ii)] Curves, for which the result has been shown by Janda \cite{Janda} using the Okounkov--Pandharipande study of Gromov--Witten theory of curves \cite{OP2, OP1, OP}.
\end{itemize}
The Gromov-Witten classes of other families of varieties 
can be seen to be tautological using (i) and (ii):
\begin{itemize}
\item[(iii)] Calabi-Yau varieties of all dimensions (the only nontrivial case is for elliptic curves covered by (ii)). 
    \item[(iv)] Rational surfaces and birationally ruled surfaces (using deformation invariance and (i) for rational surfaces and  localization and (ii) for the ruled case).
    \item[(v)] Products $X=X_1\times X_2$, where both $X_1$ and $X_2$ have tautological Gromov-Witten classes, also have tautological Gromov-Witten classes by \cite{Beh-prod}.
\end{itemize}
Finally, various sporadic examples of tautological
Gromov-Witten theories are known. For example,
the Gromov-Witten theory of Enriques surfaces 
is tautological by 
the degeneration method of \cite{MaulikP}, 
the reduced theory of $K3$ surfaces in 
primitive curves classes is tautological by
\cite{MaulikPThomas}, and the 
Gromov--Witten theory of even dimensional complete intersections of two quadrics is tautological by the combination of the recent result of Hu \cite{hu2021computing}
showing the quantum cohomology is generically semisimple with the Givental-Teleman classification of semisimple cohomological field theories
\cite{Tel}.{\footnote{Work in progress by D. Maulik and D.Ranganathan using logarithmic degenerations shows the push-forward to the moduli of curves of the fundamental class of the moduli space of stable maps is tautological for a large class of surfaces.}}

By lifting the algorithm described in \S \ref{sec_intro_algorithm} to the level of Gromov--Witten classes,
we are able to add a fundamentally 
new family to this list.

\vspace{5pt}
\begin{citedthm}
\label{thm_intro_tautological}
Let $X$ be a smooth complete intersection in projective space. Then, the Gromov--Witten classes of $X$ are tautological.
\end{citedthm}
\vspace{5pt}

We also obtain a slightly different proof using nodal Gromov--Witten theory of the result of Janda \cite{Janda} for curves: 
by degeneration, the study of a genus $g$ curve can be reduced to the study of a genus $1$ curve, which is a cubic curve in $\PP^2$ and for which Theorem \ref{thm_intro_tautological} applies (see Remark \ref{rem_tautological_curves}).

More generally, the methods of this paper can be applied to the study of the
Gromov--Witten theory of varieties $X$ which vary in families with good degenerations and large monodromy. A natural question is to consider complete
intersections in toric varieties and homogeneous spaces. Further results along these
lines will be presented in \cite{ABPZ2}.

\subsection{Related work} Building on the techniques developed here, Oberdieck 
\cite{oberdieck} 
has recently proven the Gromov--Witten/Donaldson--Thomas correspondence conjectured in \cite{MNOP1,MNOP2} for complete intersection Fano threefolds (including primitive classes).

\subsection{Acknowledgments}
We thank D.~Abramovich, A. Beauville, M.~Bousquet-M\'elou, J. Bryan, 
C.~Faber, T.~Graber, F.~Janda, Y.-P.~Lee,
D.~Maulik, A.~Okounkov, A.~Pixton, D.~Ranganathan, and Y.~Ruan for
conversations over the years related to monodromy and degeneration in
Gromov--Witten theory. The last steps of the argument were completed
at the
{\em Helvetic Algebraic Geometry Seminar}
hosted by A. Szenes at the University of Geneva in August 2021.

H.A. was supported by Fondation
Math\'ematique Jacques Hadamard.
R.P. was supported by SNF-200020-182181,  ERC-2017-AdG-786580-MACI, and SwissMAP. 
D.Z. was supported by ANR-18-CE40-0009 ENUMGEOM.
This project has received funding
from the European Research Council (ERC) under the European Union Horizon 2020 research and innovation program (grant agreement No 786580).


\section{Nodal Gromov--Witten theory}
\label{NGWTh}
Throughout the paper, unless explicitly specified, homology and cohomology groups are taken with $\Q$ coefficients. We systematically use intersection theory as developed in \cite{Fulton}, extended to Deligne--Mumford stacks in \cite{vistoli}, and 
to Artin stacks in \cite{Kresch}. 

We start by reviewing the notation for graphs which we will use for studying nodal Gromov--Witten invariants in what follows. 

\subsection{Graphs}
A \emph{graph} $\Gamma$ consists of the following data: 
\begin{itemize}
\item[(i)]    $\V_{\Gamma}$ : a finite set of vertices. 
\item[(ii)]    $\H_{\Gamma}$ : 
a finite set of half-edges equipped with a vertex assignment
$$\v :  \H_{\Gamma} \to  \V_{\Gamma}$$
and an involution $\iota$. 
The set
$\E_\Gamma$ of 2-cycles of $\iota$
is the set of edges of $\Gamma$, 
and the set $L_\Gamma$ of fixed points of 
$\iota$ is the set of legs of $\Gamma$. We denote by $n_\Gamma$ the cardinality of $L_{\Gamma}$. 
\item[(iii)] $\ell_\Gamma$ : a bijection between the set of legs
$L_\Gamma$  and an ordered set of markings.
\end{itemize}
The graph $\Gamma$ is not required to be connected.

For every variety $X$ over $\C$, we denote by $H_2^+(X) \subset H_2(X,\Z)$ the monoid{\footnote{In particular, we consider $0\in H_2(X,\mathbb{Z})$ to be an effective curve class.}}  generated by effective curve classes in $X$. 
We recall the definition of $X$-valued stable graphs{\footnote{Such graphs also appear in \cite{BehrendManin}, where for an $X$-valued stable graph, the terminology stable $A$-graph is used, where $A=H_2^+(X)$).}} 
below, following \cite[\S 0.3]{janda2020double}.

\begin{definition}
\label{Def: Xvalued}
An {\em $X$-valued stable graph} is a graph $\Gamma$ equipped with  
\begin{itemize}
    \item[(i)] a genus function $\g:\V_{\Gamma} \to \Z_{\geq 0}$,  
    \item[(ii)] a map $\beta: \V_{\Gamma} \to H_2^+(X)$ assigning a curve class to each vertex of $\Gamma$,
\end{itemize}
satisfying the stability condition: if $\beta(v)=0$,
then
\begin{equation}
2\g(v)-2+ \n(v) >0\,,
\end{equation}
where $\n(v) = |\v^{-1}(v)|$ is the valence of $\Gamma$ at $v$, that is, the number of half-edges adjacent to $v$, which by definition are formed either as part of an edge or a leg. 
\hspace*{\fill} 
$\Diamond$
\end{definition} 

The following definition will be used systematically in the upcoming sections for comparing general graphs with both edges and legs to much simpler graphs without any edges, but just with legs.

\begin{definition} \label{def_contraction}
Given an $X$-valued stable graph $\Gamma$, we denote by 
$\overline{\Gamma}$ the $X$-valued stable graph without edges obtained from
$\Gamma$
by contracting all edges. Explicitly, for every connected component 
$\gamma$ of $\Gamma$, with set of vertices $V_\gamma$, we have a vertex 
$v_\gamma$ in $\overline{\Gamma}$, with genus
\[ \g(v_\gamma):= \sum_{v \in V_\gamma} \g(v) + h^1(\gamma)\,,\] 
where $h^1(\gamma)$ is the first Betti number of $\gamma$, and class
\[ \text{\hspace{150pt}}
\beta(v_\gamma):= \sum_{v \in V_\gamma} \beta(v)\,.
\text{\hspace{150pt} $\Diamond$}
\]
\end{definition}

\begin{figure}
\resizebox{.9\linewidth}{!}{\input{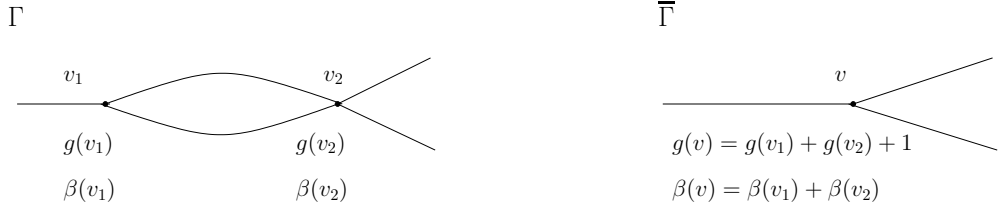}}
\caption{An $X$-valued stable graph $\Gamma$ and the contraction $\overline{\Gamma}$}
\label{Fig: gammabar}
\end{figure}


\subsection{Nodal curves}
\label{sec:moduli_curves}
An {\em $H_2^+(X)$-stable curve} of class $\beta \in H_2^+(X)$ 
is a prestable curve $C$ with a de\-co\-ration $\beta(C') \in H_2^+(X)$ for each irreducible component $C'$ of~$C$, such that 
\[ \sum_{C'} \beta(C') =\beta \,,\]
and which further satisfies the \emph{$H_2^+(X)$-stability condition}:
\begin{itemize}
    \item[(i)] 
if a genus $0$ component is decorated by the curve class $0$, then it contains at least 3 special points, and
    \item[(ii)] 
if a genus $1$ component is decorated by the curve class $0$, then it contains at least $1$ special point, 
\end{itemize}
where a special point is either a node or a marked point.
In the situation when we consider a family of $H_2^+(X)$-stable curves, if an irreducible component $C'$ degenerates into a union of several irreducible components $C'_1, \dots, C'_k$, the decorations are required to satisfy $\beta(C') = \sum_{i=1}^k \beta(C'_i)$.

For every $g,n \in  \Z_{\geq 0}$ and $\beta \in H_2^+(X)$, we denote by $\ooMM_{g,n,\beta}$ the moduli stack of genus $g$, $n$-pointed connected $H_2^+(X)$-stable curves of class $\beta$.
It is shown in \cite[Proposition 2.0.2]{costello} that the stack $\ooMM_{g,n,\beta}$ is smooth and the natural morphism 
forgetting the curve class decorations
\begin{equation}
    \label{Eq:forget}
\ooMM_{g,n,\beta}\longrightarrow \ooMM_{g,n}    
\end{equation}
to the moduli stack of $n$-pointed genus 
$g$ prestable curves is étale. 

\begin{definition} \label{def:gamma_curve}
Let $\Gamma$ be an $X$-valued stable graph. A \emph{$\Gamma$-curve} $(C_v \to S)_{v \in V_\Gamma}$ over a scheme $S$ is the data, for each vertex~$v$ of~$\Gamma$, of an $S$-point $C_v \to S$ of $\ooMM_{\g(v),\n(v),\beta(v)}$ and a one-to-one correspondence between its $\n(v)$ sections of marked points and the half-edges of~$\Gamma$ adjacent to~$v$.
 \hspace*{\fill} $\Diamond$
\end{definition} 
\vspace{8pt}

\begin{definition}
\label{Def: glued}
Let  $(C_v \to S)_{v \in V_\Gamma}$ be a $\Gamma$-curve over a scheme~$S$. The prestable curve $C \to S$ \emph{formed by  $(C_v \to S)_{v \in V_\Gamma}$} is the prestable curve obtained from $\bigcup_{v \in V_\Gamma} C_v$ by gluing for each edge $e$ of $\Gamma$, the sections of marked points corresponding to the half-edges composing~$e$. \hspace*{\fill} $\Diamond$
\end{definition}
\vspace{8pt}

\begin{definition} \label{def:gamma_marking}
A \emph{$\Gamma$-marking} of a prestable curve $C \to S$ is the data of a $\Gamma$-curve $(C_v \to S)_{v \in V_\Gamma}$, such that $C \to S$ is formed by  $(C_v \to S)_{v \in V_\Gamma}$. A prestable curve $C \to S$ endowed with a $\Gamma$-marking is called {\em $\Gamma$-marked}. For every edge $e$ of $\Gamma$, the family of {\em nodes $\Gamma$-marked by $e$} is the section of $C \to S$ obtained by the identification of the two sections corresponding to the half-edges of~$e$. 
\hspace*{\fill} $\Diamond$
\end{definition}  
\vspace{8pt}

Figure~\ref{Fig: g2g3} represents an $X$-valued stable graph $\Gamma$ and a $\Gamma$-marked prestable curve.
\begin{figure}
\resizebox{.9\linewidth}{!}{\input{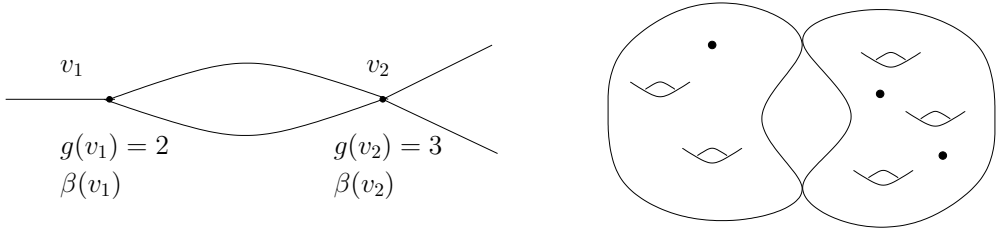}}
\caption{An $X$-valued stable graph $\Gamma$ and a $\Gamma$-marked prestable curve}
\label{Fig: g2g3}
\end{figure}
A $\Gamma$-marking automatically endows a prestable curve with the structure of an $H_2^+(X)$-stable curve.

Let $\ooMM_\Gamma$ be the moduli stack of 
$\Gamma$-marked prestable curves. By Definition \ref{def:gamma_marking}, we have 
a natural isomorphism
$$\ooMM_\Gamma \simeq \prod_{v \in V_{\Gamma}}
\ooMM_{\g(v),\n(v),\beta(v)}\,.$$ 
As the stacks 
$\ooMM_{\g(v),\n(v),\beta(v)}$ are smooth, the stack $\ooMM_\Gamma$ is also smooth.

Note that a $\Gamma$-marked prestable curve has nodes imposed by the edges of $\Gamma$. Let $\overline{\Gamma}$ be the $X$-valued stable graph without edges, obtained from $\Gamma$ by contraction of all edges, as in Definition~\ref{def_contraction}.
Then, there is a canonical morphism
\begin{equation}\label{eq_iota_gamma}
\iota_\Gamma : \ooMM_\Gamma \rightarrow \ooMM_{\overline{\Gamma}}\,,
\end{equation}
as a $\Gamma$-marked prestable curve is naturally $\overline{\Gamma}$-marked. The morphism 
$\iota_\Gamma$ is finite, unramified, local complete intersection (lci) of codimension $|E_\Gamma|$, with normal bundle 
\begin{equation}\label{eq_normal}
N_{\iota_\Gamma} := \iota_\Gamma^{*}T_{\ooMM_{\overline{\Gamma}}}/T_{\ooMM_\Gamma}
= \bigoplus_{e=\{h,h'\}\in E_\Gamma} \LL_h^\vee \otimes \LL_{h'}^\vee\,,\end{equation}
where $T_{\ooMM_{\overline{\Gamma}}}$ and $T_{\ooMM_\Gamma}$ denote the tangent bundles to $\ooMM_{\overline{\Gamma}}$ and $\ooMM_\Gamma$ respectively, and for every half-edge $h$ of $\Gamma$, we denote by
$\LL_h$ the line bundle on $\ooMM_\Gamma$ given by the cotangent line at the marked point corresponding to $h$. Indeed, the summand in \eqref{eq_normal} indexed by the edge $e$ of 
$\Gamma$ is the contribution of the smoothing of the node which is $\Gamma$-marked by $e$ to the normal bundle $N_{\iota_\Gamma}$. 

As $\iota_\Gamma$ is local complete intersection of codimension $|E_\Gamma|$, for any stacks $M_{\overline{\Gamma}}$ and $M_\Gamma$ fitting into a fiber diagram 
\[\begin{tikzcd}
M_\Gamma  
\arrow[r]
\arrow[ d]
&
M_{\overline{\Gamma}}
\arrow[d]\\
\ooMM_\Gamma
\arrow[r,"\iota_\Gamma"]& \ooMM_{\overline{\Gamma}}\,,
\end{tikzcd}\]
there exists a well-defined Gysin pull-back \cite[Chapter 6]{Fulton}
$$ \iota_\Gamma^! \colon H_{*}(M_{\overline{\Gamma}}) \longrightarrow H_{*-2|E_\Gamma|}(M_\Gamma) \,.$$

\subsection{Nodal Gromov--Witten theory}
\label{sec:moduli_nodal}
Let $X$ be a smooth projective variety over $\C$.
For every $g,n \in \Z_{\geq 0}$ and $\beta \in H_2^+(X)$,
let $\oM_{g,n,\beta}(X)$ be the moduli stack of genus $g$, $n$-pointed, connected stable maps to $X$ of class $\beta$
\cite{FuP,kontsevich1995enumeration}, and let
$[\oM_{g,n,\beta}(X)]^\virt$ be its virtual class given by Gromov--Witten theory \cite{behrend97gw, BF}.
The natural morphism 
$$\oM_{g,n,\beta}(X) \longrightarrow \ooMM_{g,n}$$ 
remembering the domain curve 
factors through the morphism in \eqref{Eq:forget} via a morphism 
\begin{equation}\label{eq_epsilon}
\epsilon \colon \oM_{g,n,\beta}(X) \longrightarrow \ooMM_{g,n,\beta}\end{equation}
remembering the domain curve and the curve class of  each irreducible component.

\begin{definition}
For every $X$-valued stable graph $G$ without edges, we define 
a moduli stack of (possibly disconnected) stable maps
\begin{equation} \label{eq_moduli_no_nodes} 
\oM_G(X) := \prod_{v \in V_G} \oM_{\g(v),\n(v),\beta(v)}(X)\,, 
\end{equation}
with a virtual class
\begin{equation}\label{eq_G_virtual}
[\oM_G(X)]^\virt := \prod_{v \in V_G} [\oM_{\g(v),\n(v),\beta(v)}(X)]^\virt\,. \end{equation}
We denote by 
\begin{equation} \label{eq_epsilon_G}
\epsilon_G \colon \oM_G(X) \rightarrow \ooMM_G\end{equation}
the morphism induced by the morphisms \eqref{eq_epsilon}. \hspace*{\fill} $\Diamond$
\end{definition}  
\vspace{8pt}

\begin{definition}
For every $X$-valued stable graph $\Gamma$,
we define the moduli stack $\oM_\Gamma(X)$ of
\emph{$\Gamma$-marked stable maps} to~$X$ by the fiber diagram
\[\begin{tikzcd}
\oM_\Gamma(X)  
\arrow[r]
\arrow["\epsilon_\Gamma"', d]
&
\oM_{\overline{\Gamma}}(X)
\arrow[d,"\epsilon_{\overline{\Gamma}}"]\\
\ooMM_\Gamma
\arrow[r,"\iota_\Gamma"]& \ooMM_{\overline{\Gamma}}\,,
\end{tikzcd}\]
where 
$\overline{\Gamma}$ is the $X$-valued stable graph without edges obtained from $\Gamma$ by contraction of all edges, as in Definition \ref{def_contraction}, and 
$\oM_{\overline{\Gamma}}(X)$ and $\epsilon_{\overline{\Gamma}}$ are defined by \eqref{eq_moduli_no_nodes} and \eqref{eq_epsilon_G} applied to $G=\overline{\Gamma}$.
In other words, a $\Gamma$-marked stable 
map 
is a stable map with the data of a
$\Gamma$-marking of its domain curve.

We define a virtual class on $\oM_\Gamma(X)$ by 
\begin{equation} \label{eq_def_absolute}
[\oM_\Gamma(X)]^\virt := \iota_\Gamma^! [\oM_{\overline{\Gamma}}(X)]^\virt \,,\end{equation}
where $[\oM_{\overline{\Gamma}}(X)]^\virt$ is defined by 
\eqref{eq_G_virtual} applied to $G=\overline{\Gamma}$. \hspace*{\fill} $\Diamond$
\end{definition} 

\begin{definition}
\label{def_nodal_gw}
Let $\Gamma$ be an $X$-valued stable graph. Nodal Gromov--Witten invariants of $X$ of type 
$\Gamma$ are 
\begin{equation} \label{eq_nodal_gw}
   \left\langle \prod_{i=1}^{n_\Gamma} \tau_{k_i}(\alpha_i) 
   \prod_{h\in H_\Gamma \setminus L_\Gamma} \tau_{k_h}
   \right\rangle_\Gamma^X :=
    \int_{[\oM_\Gamma(X)]^\virt} \prod_{i=1}^{n_\Gamma} \psi_i^{k_i}\,\ev_i^{*}(\alpha_i) \prod_{h \in H_\Gamma \setminus L_\Gamma} \psi_h^{k_h} \,,
\end{equation}
where:
\begin{itemize}
\item[(i)] for every $1 \leq i \leq n_\Gamma$, $\ev_i$ is the evaluation morphism at the $i$-th leg of $\Gamma$, 
$k_i$ is a nonnegative integer, $\alpha_i \in H^{\star}(X)$,
and $$\psi_i :=c_1(\LL_i)$$ where $\LL_i$ is the line bundle over $\oM_\Gamma(X)$
formed by the cotangent lines at the $i$-th marked point of the domain curves, 
\item[(ii)] for every half-edge $h$ of $\Gamma$ which is not a leg, adjacent to a vertex $v$, $k_h$ is a nonnegative integer and
$$\psi_h :=c_1(\LL_h)$$ where $\LL_h$ is the line bundle over 
$\oM_\Gamma(X)$ formed by the cotangent lines at the marked point of the curve $C_v$ corresponding to $h$.
\end{itemize}

A \emph{nodal Gromov--Witten invariant} of $X$ is by definition a nodal Gromov--Witten invariant of $X$ of type $\Gamma$ for some $X$-valued stable graph $\Gamma$. When $\Gamma$ has a single vertex, nodal Gromov--Witten invariants of type 
$\Gamma$ are just ordinary (connected) Gromov--Witten invariants. \hspace*{\fill} $\Diamond$
\end{definition}  
\vspace{8pt}

As \eqref{eq_nodal_gw} is a linear function of each $\alpha_i$, we may assume without loss of generality 
that $\alpha_i$'s in \eqref{eq_nodal_gw}
are elements of a fixed basis of $H^\star(X)$.
More general nodal invariants can be defined by allowing evaluation classes at the markings
${h\in H_\Gamma \setminus L_\Gamma}$. However, for our study, only the more restrictive
invariants of Definition \ref{def_nodal_gw} are needed.

\subsection{Review of relative Gromov--Witten theory}
\label{subsec: Moduli spaces of relative stable maps}

\subsubsection{Graphs for the relative theory} Let $X$ be a smooth projective variety over $\C$, and
let $D \subset X$ be a smooth divisor with connected components $(D_j)_{j\in J}$. 
\begin{definition} \label{Def: XDvalued}
An \emph{$(X,D)$-valued stable graph} $\Gamma$ is an $X$-valued stable graph as in Definition \ref{Def: Xvalued}
with the additional data of:
\begin{enumerate}
    \item[(i)] a partition 
$L_\Gamma=L_{\Gamma,I} \sqcup \bigsqcup_{j \in J} L_{\Gamma,D_j} $ of the set of legs into a set $L_{\Gamma,I}$ of interior legs, and sets $L_{\Gamma, D_j}$ of relative legs associated to the divisors $D_j$.
\item[(ii)] a relative multiplicity function $\mu_{\Gamma} \colon \bigsqcup_{j \in J} L_{\Gamma, D_j} \rightarrow \Z_{>0}$.
\end{enumerate}
Given an $(X,D)$-valued stable graph $\Gamma$,
we denote the number of interior and
relative legs by $$n_{\Gamma,I}:=|L_{\Gamma,I}|\ \ \ \text{and}\ \ \  n_{\Gamma,D}:= \sum_{j \in J}|L_{\Gamma,D_j}|$$
respectively. The total number of legs is $n_\Gamma= n_{\Gamma,I}
+n_{\Gamma,D}$. \hspace*{\fill} $\Diamond$
\end{definition}  
\vspace{8pt}

Given an $(X,D)$-valued stable graph 
$\Gamma$, we use the notation $\overline{\Gamma}$ to denote the 
$(X,D)$-valued stable graph without edges obtained from 
$\Gamma$ by contraction of all edges, analogously as in Definition 
\ref{def_contraction}. 

For every $(X,D)$-valued stable graph $G$ without edges, there is a moduli stack
\begin{equation} \label{eq_relative_moduli}
\oM_G(X,D)\end{equation}
defined by Jun Li \cite{li2001stable, JunLi}
of $G$-marked relative stable maps to $(X,D)$. 
The connected components $C_v$ of the domain curve of a $G$-marked relative stable map are indexed by 
the vertices $v \in V_G$. The relative stable map restricted to $C_v$ is
of genus $\g(v)$ with
$\n(v)$ marked points and class $\beta(v)$. In addition, the relative multiplicities along $D$ are prescribed by $\mu_G$. We denote by $[\oM_G(X,D)]^\virt$ the virtual class defined by relative Gromov--Witten theory \cite{li2001stable, JunLi}.

\subsubsection{Relative stable maps}
We briefly recall here the basics of the theory of relative stable maps \cite{li2001stable, JunLi}, 
in particular the notion of expanded degeneration and the predeformability condition, which will play a crucial role here.
To simplify the exposition, we assume that $D$ is connected. To treat the general case of a possibly disconnected divisor $D$, the only change is that  different orders of expansions along each connected component of $D$
are permitted.

Let $\bP$ be the $\PP^1$-bundle over $D$ given by 
$$\bP:= \PP(N_{D|X} \oplus \cO_D)\, ,$$ where 
$N_{D|X}$ is the normal line bundle to $D$ in $X$.
The bundle $\bP$ has two natural disjoint sections; one with normal bundle 
$N_{D|X}^\vee$, called the zero section, and the other with normal bundle 
$N_{D|X}$, called the infinity section.
For every $l \in \Z_{\geq 0}$, let $\bP_l$ be the variety obtained by gluing together $l$ copies of $\bP$, where the infinity section of the $i$th
component is glued to the zero section of the $(i + 1)$st for all 
$1 \leq i\leq l-1$.
We denote by $D_0 \subset \bP_l$ the zero section of the first copy of $\bP$ and by $D[l] \subset \bP_l$ the infinity section of the last copy of $\bP$.
The \emph{$l$-step expanded degeneration} of the pair 
$(X,D)$ is the pair $(X[l],D[l])$, where $X[l]$
is the variety obtained by gluing $X$ along $D$ to $\bP_l$ along $D_0$. The subvariety
$\bP_l \subset X[l]$ is called the \emph{expansion}, and each component 
$\bP \subset \PP_l$ is called a \emph{bubble}. We illustrate the $2$-step expanded degeneration in Figure \ref{Fig: bubbles}.

\begin{figure}
\resizebox{.5\linewidth}{!}{\input{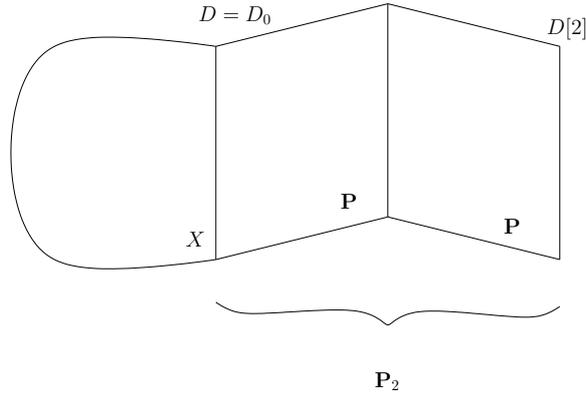}}
\caption{The $2$-step expanded degeneration $X[2]$ of $(X,D)$}
\label{Fig: bubbles}
\end{figure}

In \cite{li2001stable}, Jun Li defines through explicit constructions of versal deformations of the spaces $X[l]$, a notion of a family of expanded degenerations and constructs a moduli stack
$\mathcal{T}$ of expanded degenerations, along with a universal family 
$\mathcal{X} \rightarrow \mathcal{T}$. This universal family has the property that for every scheme $S$ and morphism 
$S \rightarrow \mathcal{T}$, the pull-back family $$\mathcal{X}_S := \mathcal{X} \times_{\mathcal{T}} S \rightarrow S$$ is a flat and proper morphism with every geometric fiber  isomorphic to $X[l]$ for some $l$.

\begin{definition} \label{def_relative_map}
A \emph{relative stable map} to $(X,D)$ with relative multiplicities $\mu$ over a scheme $S$ 
is a commutative diagram 
\[\begin{tikzcd}
C 
\arrow[r]
\arrow[d]
&
\mathcal{X}
\arrow[d]\\
S
\arrow[r]& \mathcal{T}\,,
\end{tikzcd}\]
where $C \rightarrow S$ is a prestable curve, 
$\mathcal{X} \rightarrow \mathcal{T}$ is the universal family of expanded degenerations of $(X,D)$, and such that the following conditions are satisfied:
\begin{itemize}
    \item[(i)] Let $f \colon C \rightarrow \mathcal{X}_S:= \mathcal{X}\times_{\mathcal{T}} S$ be the induced $S$-morphism. 
    Then, for every geometric fiber, $$f_s \colon C_s \rightarrow \mathcal{X}_{S,s}\simeq X[l]\, ,$$ over $s \in S$, no irreducible component of $C_s$ is entirely mapped by $f_s$ into the singular locus of $X[l]$ or the divisor $D[l] \subset X[l]$. In addition, the relative multiplicities along $D[l]$ are fixed to be given by $\mu$.
    \item[(ii)] Every geometric fiber $f_s \colon C_s \rightarrow \mathcal{X}_{S,s}$ over $s \in S$ is \emph{stable} in the sense that there are finitely many pairs $(r_1,r_2)$, where $r_1$ is an automorphism of $C_s$, $r_2$ is an automorphism of 
    $\mathcal{X}_{S,s}$ fixing $X$, and $f_s \circ r_1=r_2 \circ f_s$.
    \item[(iii)] For each point $s \in S$ and $p \in C_s$ such that $f_s(p)$ is contained in the singular locus of $\mathcal{X}_{S,s}\simeq X[l]$, $f$ is 
    \emph{predeformable} at $p$, that is, $p$ is a node of $C_s$, and étale-locally on $C$, and smooth-locally on $\mathcal{X}_S$, the morphism $f$ admits the following
form: \[ \Spec A[x,y]/(xy-t) \rightarrow \Spec A[u,v]/(uv-w) \]
over $\Spec A$, for some algebra $A$ and $t,w \in A$, where $w = t^n$, $u \mapsto x^n$, and 
$v \mapsto y^n$ for some $n \in \Z_{\geq 1}$.
\end{itemize}
The nodes of the domain curve of a relative stable map which 
map to the singular locus of the expanded target (as described in (iii))   are called \emph{distinguished} nodes. 
An isomorphism between two relative stable maps is an isomorphism between the corresponding diagrams, which is the identity on $S$ and $\mathcal{T}$, and an automorphism fixing $X$ on $\mathcal{X}$. \hspace*{\fill} $\Diamond$
\end{definition}  
\vspace{8pt}

Definition \ref{def_relative_map} in particular implies that two
relative stable maps with target $X[l]$ which only differ by the action of $\mathbb{G}_m^l$ on $X[l]$ by scaling of the fibers of the bubbles are isomorphic.

\subsection{Nodal relative Gromov--Witten theory}
\label{subsec: Moduli spaces of nodal relative stable maps}
Throughout this section $X$ denotes a smooth projective variety over $\C$,  and $D\subset X$ a smooth divisor. Forgetting the relative information (the data (i) and (ii) of Definition \ref{Def: XDvalued}), 
we view every $(X,D)$-valued stable graph $\Gamma$ as an
$X$-valued stable graph. Recall that, we denote by $\overline{\Gamma}$ the graph obtained by contracting all edges of $\Gamma$, and we have a natural morphism $\iota_\Gamma \colon \ooMM_\Gamma \rightarrow \ooMM_{\overline{\Gamma}}$
defined in \eqref{eq_iota_gamma}, where $\ooMM_\Gamma$ and $\ooMM_{\overline{\Gamma}}$ denote the moduli stacks of $\Gamma$-marked and $\overline{\Gamma}$-marked prestable curves respectively. We denote by 
\begin{equation}\label{eq_epsilon_G_relative}
\epsilon_{\overline{\Gamma}} \colon \oM_{\overline{\Gamma}}(X,D) \rightarrow \ooMM_{\overline{\Gamma}}\,,\end{equation}
the morphism defined as in \eqref{eq_epsilon_G}, applied to $G=\overline{\Gamma}$. 

\begin{definition} \label{def_M_Gamma}
For every $(X,D)$-valued stable graph $\Gamma$, 
we define the moduli stack $\oM_\Gamma(X,D)$ of
\emph{$\Gamma$-marked relative stable maps} by the fiber diagram
\[\begin{tikzcd}
\oM_\Gamma(X,D)  
\arrow[r]
\arrow["\epsilon_\Gamma"', d]
&
\oM_{\overline{\Gamma}}(X, D)
\arrow[d,"\epsilon_{\overline{\Gamma}}"]\\
\ooMM_\Gamma
\arrow[r,"\iota_\Gamma"]& \ooMM_{\overline{\Gamma}}\,,
\end{tikzcd}\]
where $\overline{\Gamma}$ is the $(X,D)$-valued stable graph obtained from $\Gamma$ by contraction of all edges as in Definition \ref{def_contraction},
$\oM_{\overline{\Gamma}}(X,D)$
and $\epsilon_{\overline{\Gamma}}$
are defined by \eqref{eq_relative_moduli}
and \eqref{eq_epsilon_G_relative}.
We define a virtual class on the moduli space $\oM_\Gamma(X,D)$ by 
\begin{equation}
    \label{eq_def_relative}
    [\oM_\Gamma(X,D)]^{\virt} := \iota_\Gamma^! [\oM_{\overline{\Gamma}}(X,D)]^\virt \,.
\end{equation}
A \emph{$\Gamma$-marked relative stable map} is a relative stable map to $(X,D)$ with the data of a
$\Gamma$-marking on its domain curve. \hspace*{\fill} $\Diamond$

\end{definition}  
\vspace{8pt}

Integration over $[\oM_\Gamma(X,D)]^\virt$, however, does not give a good notion of nodal relative Gromov--Witten invariants, which for example should appear in a degeneration formula for nodal absolute Gromov--Witten invariants. This will become clear in our discussion of the degeneration formula in \S \ref{section_degeneration}. To obtain a good notion of nodal relative Gromov--Witten invariants, we add the extra condition that the nodes of the domain curve imposed by $\Gamma$ remain away from the singular locus of the expanded targets. We show below that the extra condition on nodes defines a moduli stack $\shP_\Gamma(X,D)$ which is a union of connected components of $\oM_\Gamma(X,D)$. Nodal relative Gromov--Witten invariants are then defined in what follows by integration over the virtual class 
$[\shP_\Gamma(X,D)]^\virt$ obtained by restriction of $[\oM_\Gamma(X,D)]^\virt$ to $\shP_\Gamma(X,D)$.

Let $f \colon C \rightarrow \mathcal{X}_S$ be a $\Gamma$-marked relative 
stable map over a scheme $S$, where $\mathcal{X}_S$ is as in Definition \ref{def_relative_map}.
By Definition \ref{def:gamma_marking} of the $\Gamma$-marking, for every edge $e$ of $\Gamma$, we have a section $\sigma_e$ of $C \rightarrow S$ with image in the singular locus of $C$. In other words, for every point $s \in S$, the point $(\sigma_e)_s$ is a node of the curve $C_s$, and we can then ask if this node is distinguished or not in the sense of Definition \ref{def_relative_map}, see Figure \ref{Fig: distinguish} for an illustration of distinguished and  non-distinguished nodes. The following result is essential for our paper.

\begin{lemma} \label{lem_key}
Let $S$ be a connected scheme, and let $f \colon C \rightarrow \mathcal{X}_S$ be a 
$\Gamma$-marked relative stable map over $S$. Let $e$ be an edge of $\Gamma$ and $\sigma_e$ the corresponding section of $C \rightarrow S$ given by the $\Gamma$-marking. Then, we have the following alternative:
\begin{itemize}
    \item[(i)] either the node $(\sigma_e)_s$ of $C_s$ is distinguished for every point $s \in S$, or
    \item[(ii)] the node $(\sigma_e)_s$ of $C_s$ is distinguished for no point $s \in S$.
\end{itemize}
\end{lemma}

\begin{proof}
We prove that the subset $Z \subset S$ of points $s \in S$ such that $(\sigma_e)_s$ is a distinguished node of $C_s$ is closed and open in the complex analytic topology.

First of all, $Z$ is closed. Indeed, its complement $S \setminus Z$ is open: if 
$s \in S \setminus Z$
$f((\sigma_e)_s)$ is a smooth point of $\mathcal{X}_{S,s}$, then, as the morphism $\mathcal{X}_S \rightarrow S$
is flat and locally of finite presentation, it is smooth in restriction to an open subset of 
$\mathcal{X}_S$ containing $(\sigma_e)_s$, 
see for example
\cite[\href{https://stacks.math.columbia.edu/tag/01V9}{Tag 01V9}]{stacks-project}.
So $\sigma_e^{-1}(U)$ is an open subset of $S$ (by continuity of $\sigma_e$), containing $s$ and contained in $S\setminus Z$.

It remains to show that $Z$ is open. Let $s \in Z$, so that $(\sigma_e)_s$ is a distinguished node of $C_s$ so $f((\sigma_e)_s)$ is a singular point of $\mathcal{X}_{S,s}$. By the predeformability condition in Definition 
\ref{def_relative_map}(iii), 
locally on $C$ for the complex analytic topology, and smooth-locally on $\mathcal{X}_S$, the morphism $f$ admits the following
form: \[ \Spec A[x,y]/(xy-t) \rightarrow \Spec A[u,v]/(uv-w) \]
over $\Spec A$, for some $t,w \in A$, in which, $w = t^n$, $u \mapsto x^n$, and 
$v \mapsto y^n$ for some $n \in \Z_{\geq 1}$. As the image of $\sigma_e$ is contained in the singular locus of $C$ (so $\sigma_e$ is a section of nodes), we deduce that $t=0$ and 
$x((\sigma_e)_{s'})=y((\sigma_e)_{s'})=0$ for every $s' \in \Spec A$. From the form of $f$, we deduce that $w=0$ and $u(f((\sigma_e)_{s'})= v(f((\sigma_e)_{s'})=0$ for every $s' \in \Spec A$, and so that $f((\sigma_e)_{s'})$ is contained in the singular locus of $\mathcal{X}_{S,s'}$
for every $s' \in \Spec A$. We conclude that $Z$ is open.
\end{proof}

\begin{figure}
\resizebox{1.0\linewidth}{!}{\input{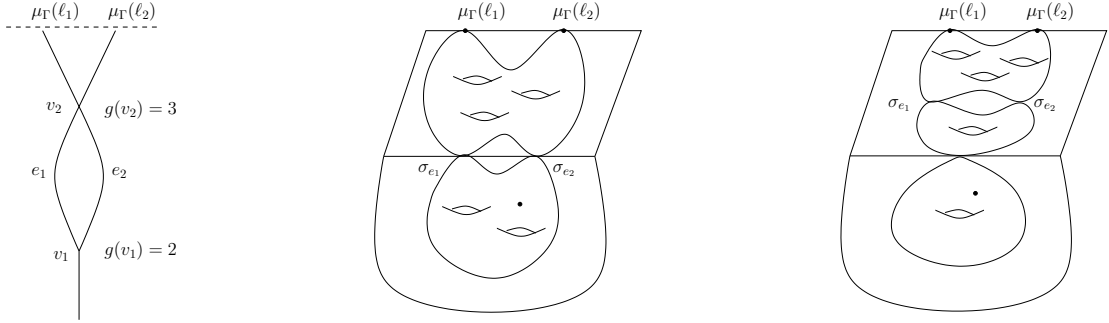}}
\caption{The left figure is an $(X,D)$-valued stable graph $\Gamma$ with two relative legs $\ell_1,\ell_2$ of multiplicities $\mu_\Gamma(\ell_1)$ and $\mu_\Gamma(\ell_2)$ respectively. In the middle and on the right, there are two $\Gamma$-marked relative stable maps. While in the middle the two nodes $\sigma_{e_1},\sigma_{e_2}$ corresponding to the edges $e_1$ and $e_2$  are distinguished, on the right they are not distinguished.}
\label{Fig: distinguish}
\end{figure}

\begin{definition}
\label{def_relative}
For every $(X,D)$-valued stable graph $\Gamma$, we define
$\shP_\Gamma(X,D)$ as the substack of $\oM_\Gamma(X,D)$ whose $S$-points
for every scheme $S$ are $\Gamma$-marked relative stable maps 
$f \colon C \rightarrow \shX_S$ such that for every edge $e$ of $\Gamma$ and every point $s \in S$, 
the node $(\sigma_e)_s$ of $C_s$ marked by $e$ is \emph{not} distinguished. \hspace*{\fill} $\Diamond$
\end{definition}  
\vspace{8pt}

By Lemma \ref{lem_key}, 
$\shP_\Gamma(X,D)$ is a union of connected components of 
$\oM_\Gamma(X,D)$ and is therefore well-defined and proper. 
We denote by 
$\shN_\Gamma(X,D)$ the complement of $\shP_\Gamma(X,D)$ in $\oM_\Gamma(X,D)$, which is also a union of connected components of $\oM_\Gamma(X,D)$ and also proper. In other words, we have a disjoint union decomposition 
\begin{equation} \label{eq_decomposition}
\oM_\Gamma(X,D) = \shP_\Gamma(X,D) \sqcup \shN_\Gamma(X,D) \,.\end{equation}
An explicit and non-trivial example of the decomposition \eqref{eq_decomposition}
is given at the end of Appendix \ref{sec_appendix}.

\begin{definition}
\label{Def: virtual classes}
We define the virtual classes 
\begin{equation}
\nonumber
    [\shP_\Gamma(X,D)]^\virt \in H_*(\shP_\Gamma(X,D)) \,\ \mathrm{and} \,\ [\shN_\Gamma(X,D)]^\virt \in H_*(\shN_\Gamma(X,D))
\end{equation}
as the classes obtained by restriction of the virtual class 
$[\oM_\Gamma(X,D)]^\virt$ to $\shP_\Gamma(X,D)$ and $\shN_\Gamma(X,D)$ respectively.
\hspace*{\fill} $\Diamond$
\end{definition}
\vspace{8pt}

Nodal relative Gromov--Witten invariants are then defined by integration over the virtual class
$[\shP_\Gamma(X,D)]^\virt$.

\begin{definition} \label{def_nodal_relative_gw}
Given an $(X,D)$-valued stable graph $\Gamma$,
the nodal relative Gromov--Witten invariants of $(X,D)$ of type 
$\Gamma$ are 
\begin{multline} \label{eq_nodal_relative_gw}
   \left\langle \prod_{i=1}^{n_{\Gamma,I}} \tau_{k_i}(\alpha_i) \prod_{h\in H_\Gamma \setminus L_\Gamma} \tau_{k_h}
   \bigg| \prod_{j=1}^{n_{\Gamma,D}} \delta_j \right\rangle_\Gamma^{(X,D)} :=\\
    \int_{[\shP_\Gamma(X,D)]^\virt} \prod_{i=1}^{n_{\Gamma,I}} \psi_i^{k_i}\,\ev_i^{*}(\alpha_i) \prod_{h\in H_\Gamma \setminus L_\Gamma} \psi_h^{k_h}
    \prod_{j=1}^{n_{\Gamma,D}} \ev_j^{*}(\delta_j)\,,
\end{multline}
where:
\begin{itemize}
\item[(i)]
for every $1 \leq i \leq n_{\Gamma,I}$, $\ev_i$ is the evaluation morphism at the $i$-th interior leg of $\Gamma$, 
$k_i$ is a nonnegative integer, $\alpha_i \in H^{\star}(X)$, 
and $$\psi_i :=c_1(\LL_i)\, ,$$ 
where $\LL_i$ is the cotangent line bundle over $\oM_\Gamma(X)$
associated to the $i$-th marked point of the domain curve,
\item[(ii)]for every half-edge $h$ of $\Gamma$ which is not a leg, adjacent to a vertex $v$, $k_h$ is a nonnegative integer, and 
$$\psi_h :=c_1(\LL_h)\, , $$ where $\LL_h$ is the cotangent line bundle over 
$\oM_\Gamma(X)$ associated to marked point of the curve $C_v$ corresponding to $h$, and;
\item[(iii)] for every $1 \leq j \leq n_{\Gamma,D}$, 
$\ev_j$ is the evaluation morphism at the $j$-th relative leg of $\Gamma$, and
$\delta_j \in H^{\star}(D)$. \hspace*{\fill} $\Diamond$
\end{itemize}
\end{definition}  
\vspace{8pt}

As \eqref{eq_nodal_relative_gw} is a linear function of each $\alpha_i$ and $\delta_j$, we may assume without loss of generality that the $\alpha_i$'s
are elements of a fixed basis of $H^\star(X)$ and the $\delta_j$'s are elements of a fixed basis of $H^{*}(D)$.

When $\Gamma$ has no edges, nodal relative  Gromov--Witten invariants of type 
$\Gamma$ are just ordinary (possibly disconnected) relative Gromov--Witten invariants.
A \emph{nodal relative Gromov--Witten invariant of $(X,D)$} is a nodal relative Gromov--Witten invariant of $(X,D)$ of type $\Gamma$ for some $(X,D)$-valued stable graph $\Gamma$.

In Definition \ref{def_nodal_relative_gw},
 we do not include insertions of $\psi$-classes at the relative markings
 for nodal relative Gromov--Witten invariants.
The reason is that we restrict our study here  to the relative invariants which appear in the degeneration formula for absolute invariants.

\section{Degeneration formula for nodal Gromov--Witten theory}
\label{section_degeneration}

After reviewing the standard degeneration formula in relative
Gromov-Witten theory in \S \ref{section_degeneration_review}, we state a new nodal degeneration formula in \S \ref{section_degeneration_statement}. The proof of the nodal degeneration formula is given in \S \ref{section_degeneration_technical}\,-\,\ref{section_degeneration_proof}.
The nodal degeneration formula follows from the standard degeneration formula
after intersection with the nodal locus in the Artin stack of curves.
The main subtlety in the argument concerns the location of
the node (and requires the vanishing of certain contributions to
the intersection product).

\subsection{Jun Li's degeneration formula}
\label{section_degeneration_review}
Let $W \rightarrow B$ be a flat projective morphism from a smooth variety $W$ to a smooth connected curve $B$ with  a distinguished point $0 \in B$
such that
\begin{enumerate}
\item[(i)] the fibers $W_t$ over $t \in B \setminus \{0\}$ are smooth varieties, 
\item[(ii)] the fiber
$W_0$ over $0\in B$ is the union of two smooth irreducible components $Y_1$ and $Y_2$ glued along
 a smooth connected divisor $D$.
\end{enumerate}
We write the degeneration as
$$W_t \rightsquigarrow W_0=Y_1\cup_D Y_2\,. $$
The inclusions 
$W_t, Y_1, Y_2 \subset W$ induce maps $$p_{W_t} \colon H_2^+(W_t)\rightarrow H_2^+(W)\, ,\ p_{Y_1}: H_2^+(Y_1) \rightarrow H_2^+(W)\, ,\   
p_{Y_2}: H_2^+(Y_2) \rightarrow H_2^+(W)\, .$$

Jun Li's degeneration formula \cite{JunLi}
(motivated by earlier work in symplectic geometry by \cite{IonelParker, LiRuan}) relates the Gromov--Witten theory of the general fibers
$W_t$ for $t \neq 0$ to the relative Gromov--Witten theories of the 
pairs $(Y_1,D)$ and $(Y_2,D)$ given by the components of the special fiber relative to their common intersection. We review below the degeneration formula, first at the level of virtual cycles, and then at the level of numerical invariants.

We first recall the notion of an {\em expanded degeneration} of the special fiber $W_0$, which is parallel to the notion of expanded degeneration of a pair $(X,D)$ reviewed in \S \ref{subsec: Moduli spaces of relative stable maps}.
Let $\bP$ be the $\PP^1$-bundle over $D$ defined by 
$$\bP:= \PP(N_{D|Y_1} \oplus \cO_D)\, $$ where 
$N_{D|Y_1}$ is the normal line bundle to $D$ in $Y_1$.
The normal line bundle to $D$ in $Y_2$
satisfies
$$N_{D|Y_2}\stackrel{\sim}{=}N_{D|Y_1}^\vee\, .$$
The bundle $\bP$ has two natural disjoint sections: one with normal bundle 
$N_{D|Y_1}^\vee$, called the zero section, and another with normal bundle 
$N_{D|Y_1}$, called the infinity section.
For every $l \in \Z_{\geq 0}$, let $\bP_l$ be the variety obtained by gluing together $l$ copies of $\bP$, where the infinity section of the $i$th
component is glued to the zero section of the $(i + 1)$st for all 
$1 \leq i\leq l-1$. The \emph{$l$-step expanded degeneration} of the special fiber $W_0$ along $D$ is the variety $W_0[l]$ obtained by gluing $Y_1$ along $D$ to $\bP_l$ along the zero section of the first copy of $\bP$, and gluing $Y_2$ along $D$ to $\bP_l$ along the infinity section of the last copy of 
$\bP$.

In \cite{li2001stable}, Jun Li constructs a moduli stack 
$\mathfrak{T} \rightarrow B$ with a universal family $\mathcal{W} \rightarrow \mathfrak{T}$ such that for every section $\sigma \colon B \rightarrow \mathfrak{T}$, the pull-back family
$\mathcal{W}_\sigma \rightarrow B$ is isomorphic to $W \rightarrow B$ away from $0 \in B$, and has a special fiber over $0 \in B$ isomorphic to $W_0[l]$ for some $l$.

The definition
of a stable map to a degeneration \cite{li2001stable,JunLi} is parallel
to Definition \ref{def_relative_map}
of a relative stable map to a pair $(X,D)$.

\begin{definition} \label{def_deg}
Let $S \rightarrow B$ be a scheme over $B$.
A \emph{stable map over $S$ to the degeneration 
$W \rightarrow B$}  is   
a commutative diagram 
\[\begin{tikzcd}
C 
\arrow[r]
\arrow[d]
&
\mathcal{W}
\arrow[d]\\
S
\arrow[r] \arrow[dr]& \mathfrak{T} \arrow[d] \\
& B
\,,
\end{tikzcd}\]
where $C \rightarrow S$ is a prestable curve,
$\mathcal{W} \rightarrow \mathfrak{T}$ is the universal family of expanded degenerations of 
$W \rightarrow B$, and such that the following conditions are satisfied:
\begin{itemize}
    \item[(i)] Let $f \colon C \rightarrow \mathcal{W}_S:= \mathcal{W}\times_{\mathfrak{T}} S$ be the induced $S$-morphism. 
    Then, for every geometric fiber $$f_s \colon C_s \rightarrow \mathcal{W}_{S,s}$$ over a point $s \in S$, no irreducible component of $C_s$ is entirely mapped by $f_s$ into the singular locus of $\mathcal{W}_{S,s}$.
    \item[(ii)] Every geometric fiber $f_s \colon C_s \rightarrow \mathcal{W}_{S,s}$ over $s \in S$ is \emph{stable} in the sense that there are finitely many pairs $(r_1,r_2)$, such that $r_1$ is an automorphism of $C_s$, $r_2$ is an automorphism of 
    $\mathcal{W}_{S,s}$ fixing $W_t$, where $t\in B$  is the image of $s \in S$ by $S \rightarrow B$, and $f_s \circ r_1=r_2 \circ f_s$.
    \item[(iii)] For every $s \in S$ and $p \in C_s$ such that $f_s(p)$ is contained in the singular locus of $\mathcal{X}_{S,s}$, $f$ is 
    \emph{predeformable} at $p$ as in Definition \ref{def_relative_map}(iii).
\end{itemize}

The nodes of the domain curve of a relative stable map which 
map to the singular locus of the expanded target (as described in (iii)) are called \emph{distinguished} nodes. \hspace*{\fill} $\Diamond$
\end{definition}  
\vspace{8pt}

For every $W$-valued stable graph $G$ without edges, let 
\[ \oM_G(W/B) \rightarrow B\] 
be the moduli stack of $G$-marked stable maps to the degeneration $W \rightarrow B$.
Connected components $C_v$ of the domain curve of a $G$-marked stable map are indexed by 
the vertices $v$ of $G$, and a $G$-marked stable map restricted to $C_v$ is $\n(v)$-pointed, of genus $\g(v)$ and class $\beta(v)$.
We assume further that $G$ is \emph{vertical}, in the sense that for every vertex $v$ of $G$, the push-forward to $B$ of the curve class $\beta(v)$ is zero. 

For every $t \in B$, denote by 
$\iota_{t}$ the inclusion of the point $t$ in the curve $B$, and by 
$\oM_G(W_t)$ the fiber of $\oM_G(W)$ over $t$.
Given a $W_t$-valued stable graph $G_t$, we denote by $p_{W_t,*}G_t$ the $W$-valued stable graph obtained from $G_t$ by replacing all the curve classes 
$\beta(v) \in H_2^+(W_t)$ by $p_{W_t,*}\beta(v) \in H_2^+(W)$
(if $\beta(v) \neq 0$, then $\beta(v)$ has a non-zero degree with respect to a relative polarization of $W \rightarrow B$, and so $p_{W_t,*}\beta(v) \neq 0$ and the graph $p_{W_t,*}G_t$ is indeed stable).
For $t \in B\setminus \{0\}$, stable maps to the degeneration $W \rightarrow B$
with image contained in $W_t$ are ordinary stable maps and so
\begin{equation} \label{eq_decomp_moduli}
\oM_G(W_t)=\bigsqcup_{\substack{G_t\\ p_{W_t,*}G_t=G}} \oM_{G_t}(W_t)\,,\end{equation}
where the disjoint union is over the $W_t$-valued stable graphs 
$G_t$ satisfying the condition $p_{W_t,*}G_t=G$, and $\oM_{G_t}(W_t)$ is the moduli stack of
$G_t$-marked
stable maps to $W_t$. There are finitely many such graphs $G_t$ as the degrees of the curve classes with respect to a relative polarization of 
$W \rightarrow B$ are fixed by $G$.

Jun Li constructs in \cite{li2001stable, JunLi} a virtual class $[\oM_G(W/B)]^{\virt}$
on $\oM_G(W/B)$ such that, for every $t \in B \setminus \{0\}$, 
\begin{equation} \label{eq_decomp_classes}
\iota_{t}^! [\oM_G(W/B)]^{\virt} =\sum_{\substack{G_t\\ 
p_{W_t,*}G_t=G}} \iota_{G_t,*}[\oM_{G_t}(W_t)]^{\virt} \,,\end{equation}
where $\iota_{G_t}$ is the inclusion of $\oM_{G_t}(W_t)$ in $\oM_G(W_t)$
given by \eqref{eq_decomp_moduli}.
The degeneration formula expresses the virtual class 
\[ \iota_0^! [\oM_G(W/B)]^\virt \]
on $\oM_G(W_0)$ in terms of the virtual classes in relative Gromov--Witten theory of 
$(Y_1,D)$ and $(Y_2,D)$. To state the formula, we first introduce some terminology about splittings.

\begin{definition} \label{def_splitting_gamma}
Let $\Gamma$ be a $W$-valued stable graph.
A \emph{splitting} $\sigma$ of $\Gamma$ is an ordered pair $(\gamma_1, \gamma_2)$ where 
$\gamma_1$ is a $(Y_1,D)$-valued stable graph and $\gamma_2$ is a $(Y_2,D)$-valued stable graph. Moreover, we require $(\gamma_1, \gamma_2)$ to satisfy the following conditions:
\begin{itemize}
    \item[(i)] $\gamma_1$ and $\gamma_2$ have the same number $\ell(\sigma)$ of relative legs,
    \item[(ii)]For each relative
    leg  $1 \leq i \leq \ell(\sigma)$, the relative multiplicity $\mu_{\gamma_1}(i)$ attached to the $i$-th leg of $\gamma_1$ is equal to the relative multiplicity 
    $\mu_{\gamma_2}(i)$ attached to the $i$-th leg of $\gamma_2$.
    \item[(iii)] The labelling of the interior legs of $\gamma_1$ and $\gamma_2$ forms a partition of 
    the labelling of the legs of $\Gamma$.
\end{itemize}
Furthermore, a splitting carries the extra data of an isomorphism between $\Gamma$ and the graph obtained by first gluing for all $1 \leq i \leq \ell(\sigma)$ the 
    $i$-th relative leg of $p_{Y_1,*}\gamma_1$ with the $i$-th relative leg of $p_{Y_2,*}\gamma_2$, and then contracting the 
    $\ell(\sigma)$ newly created edges. Here 
    $p_{Y_i,*}\gamma_i$ is the graph obtained from $\gamma_i$ by replacing all the curve classes 
    $\beta(v) \in H_2^+(Y_i)$ by $p_{Y_i,*}\beta(v) \in H_2^+(W)$.
    
Two splittings $(\gamma_1,\gamma_2)$ and $(\gamma_1',\gamma_2')$ are isomorphic if there exist
isomorphisms $\gamma_1 \simeq \gamma_1'$ and $\gamma_2 \simeq \gamma_2'$ compatible 
with the data of the isomorphisms between $\Gamma$ and the glued contracted graphs.
We denote by $\Omega_\Gamma$ the set of isomorphism classes of splittings of $\Gamma$. 
We say that two splittings $\sigma_1$ and $\sigma_2$ are equivalent if they differ by a permutation of the labelling of the relative legs. We denote by $\overline{\Omega}_{\Gamma}$ the set of equivalence classes of splittings of $\Gamma$. For every $\sigma \in \overline{\Omega}_{\Gamma}$, we denote by $|\Aut(\sigma)|$ the order of the group of permutations of the labelling of relative legs fixing a splitting representative of the class $\sigma$, and by
\begin{equation}
    \label{Eq: multiplicity factor}
m(\sigma):=\prod_{i=1}^{\ell(\sigma)} \mu_{\gamma_1}(i)
=\prod_{i=1}^{\ell(\sigma)} \mu_{\gamma_2}(i)    
\end{equation}
the product of the 
$\ell(\sigma)$  relative multiplicities. We refer to $m(\sigma)$ as the \emph{multiplicity factor}.
\hspace*{\fill} $\Diamond$
\end{definition}  
\vspace{8pt}

Let $\eta=(G_1,G_2) \in \overline{\Omega}_G$ be an equivalence class of splittings of $G$, as in Definition \ref{def_splitting_gamma}. As $G$ does not have any edges, the same is true for $G_1$ and $G_2$.
The $(Y_1,D)$-valued and $(Y_2,D)$-valued stable graphs $G_1$ and $G_2$ define moduli stacks 
$\oM_{G_1}(Y_1,D)$ and $\oM_{G_2}(Y_2,D)$ of relative stable maps to $(Y_1,D)$ and $(Y_2,D)$ respectively, as
\eqref{eq_relative_moduli}. 
The evaluation at the relative marked points corresponding to the relative legs of $G_1$ and $G_2$ defines a morphism
\[ \oM_{G_1}(Y_1,D) \times \oM_{G_2}(Y_2,D) \rightarrow (D \times D)^{\ell(\eta)} \,.\]
We denote by $\Delta_\eta$ the diagonal morphism 
\[ \Delta_\eta \colon  D^{\ell(\eta)} \rightarrow (D \times D)^{\ell(\eta)}\,,\]
and we form the fiber product
\[\begin{tikzcd}
\oM_{G_1}(Y_1,D) \times_{D^{\ell(\eta)}} \oM_{G_2}(Y_2,D)  
\arrow[r]
\arrow[d]
&
\oM_{G_1}(Y_1,D) \times \oM_{G_2}(Y_2,D) \arrow[d]\\
D^{\ell(\eta)}
\arrow[r,"\Delta_\eta"]& (D \times D)^{\ell(\eta)}\,.
\end{tikzcd}\]
For every scheme $S$, an $S$-point of the above fiber product is the data of a stable map over $S$ to an expansion 
$Y_1[l_1]$ of $Y_1$ along $D$ and a stable map over $S$ to an expansion $Y_2[l_2]$ of $Y_2$
along $D$, such that the images of the relative marked points on $$D[l_1]=D[l_2] \simeq D$$
match. 
As the contact orders also match, one can glue these two stable maps together to obtain a stable map over $S$ to the expansion $W_0[l_1+l_2]$ of $W_0$. In other words, we have a gluing morphism
\begin{equation} \label{eq_gluing_morphism_1}
\Phi_\eta \colon 
\oM_{G_1}(Y_1,D) \times_{D^{\ell(\eta)}} \oM_{G_2}(Y_2,D)  
\rightarrow \oM_G(W_0) \,.\end{equation}
We can finally state Jun Li's degeneration formula at the cycle-level \cite{JunLi}:
\begin{equation}  \label{eq_li_deg_cycle}
\iota_0^! [\oM_G(W/B)]^\virt 
= \sum_{\eta=(G_1,G_2)\in  
\overline{\Omega}_G} 
\frac{m(\eta)}{|\Aut(\eta)|} \Phi_{\eta,*} \Delta_\eta^! 
([\oM_{G_1}(Y_1,D)]^\virt \times [\oM_{G_2}(Y_2,D)]^\virt)\,.\end{equation}
Combining \eqref{eq_decomp_classes} with \eqref{eq_li_deg_cycle} and using the Künneth decomposition of the class of the diagonal $\Delta_\eta$, we obtain the numerical version of the degeneration formula: for every 
$k_i\in \Z_{\geq 0}$ and cohomology classes $\alpha_i \in H^{*}(W)$ indexed by the legs $i \in L_G$ of $G$, we have for every $t \in B \setminus \{0\}$:
\begin{align}
     \label{eq_li_deg_num}
     \nonumber
    \sum_{\substack{G_t\\p_{W_t,*}G_t=G}} \left\langle \prod_{i \in L_G} \tau_{k_i}(\alpha_i) \right\rangle_{G_t}^{W_t}= \,\  \,\ \,\ \,\ &  \nonumber
  \\  \sum_{\eta=(G_1,G_2)\in  
\overline{\Omega}_G} 
\frac{m(\eta)}{|\Aut(\eta)|}
\sum_{j_1,\dots,j_{\ell(\eta)}} 
(-1)^\epsilon  & 
\left\langle \prod_{i \in L_{G_1,I}}\tau_{k_i}(\alpha_i)\big|\prod_{i=1}^{\ell(\eta)} 
\delta_{j_i} \right\rangle_{G_1}^{(Y_1,D)} \nonumber
\\ & \,\  \,\ \,\ \,\ \,\  \,\ \,\ \,\ \boldsymbol{\cdot} \left\langle \prod_{i=1}^{\ell(\eta)} 
\delta_{j_i}^\vee \, \Big| \prod_{i \in L_{G_2,I}}\tau_{k_i}(\alpha_i)\right\rangle_{G_2}^{(Y_2,D)}\,,
\nonumber
\end{align}
where by abuse of notation we still denote by $\alpha_i$ the restrictions of 
$\alpha_i$ to $W_t$, $Y_1$, and $Y_2$ respectively,
$(\delta_j)_j$ is a basis of $H^{\star}(D)$, 
$(\delta_j^\vee)_j$ is the Poincaré-dual basis, and 
$(-1)^\epsilon$ is the sign determined formally by the equality 
\[ \prod_{i\in L_\G} \alpha_i =(-1)^\epsilon 
\prod_{i\in L_{G_1,I}} \alpha_i \prod_{i\in L_{G_2,I}} \alpha_i \,.\]

\subsection{Statement of the nodal degeneration formula}
\label{section_degeneration_statement}

Let $\Gamma$ be a $W$-valued stable graph and 
$\overline{\Gamma}$ the $W$-valued stable graph obtained from $\Gamma$ by contracting all edges, as in 
Definition \ref{def_contraction}.
We assume further that $\Gamma$ is \emph{vertical}, in the sense that for every vertex $v$ of $\Gamma$, the push-forward to $B$ of the curve class $\beta(v)$ is zero. 
We define the moduli stack $\oM_\Gamma(W/B)$ of
\emph{$\Gamma$-marked stable maps to the degeneration $W \rightarrow B$} by the fiber diagram
\[\begin{tikzcd}
\oM_\Gamma(W/B)  
\arrow[r]
\arrow[d]
&
\oM_{\overline{\Gamma}}(W/B)
\arrow[d]\\
\ooMM_\Gamma \times B
\arrow[r,"\iota_\Gamma"]& \ooMM_{\overline{\Gamma}} \times B\,.
\end{tikzcd}\]
In other words, a $\Gamma$-marked stable curve is a stable map with the data of a
$\Gamma$-marking on its domain curve. Using the predeformability condition 
(iii) in the Definition \ref{def_deg}
of a stable map to the degeneration $W \rightarrow B$, we obtain a version of Lemma \ref{lem_key}, ensuring that we have a disjoint union decomposition
\[ \oM_\Gamma(W/B) = \shP_\Gamma(W/B) \sqcup \shN_\Gamma(W/B) \,,\]
where
$\shP_\Gamma(W/B) \rightarrow B$ is the substack of $\oM_\Gamma(W/B) \rightarrow B$ whose $S$-points are $\Gamma$-marked stable maps to the degeneration 
$f \colon C \rightarrow \mathcal{W}_S$ such that for every edge $e$ of $\Gamma$ and every point $s \in S$, 
the node $(\sigma_e)_s$ of $C_s$ marked by $e$ is \emph{not} distinguished
(the image lies away from the singular locus of the target $\mathcal{W}_{S,s}$).
The complement substack 
$\shN_\Gamma(W/B) \rightarrow B$ parameterizes stable maps having at least one 
$\Gamma$-marked node mapped to the singular locus of the target.
As the target remains unexpanded over $B \setminus \{0\}$, the stack
$\shN_\Gamma(W/B)$ is entirely over $0 \in B$. 

We define a virtual class on $\oM_\Gamma(W/B)$ by 
\begin{equation} \label{eq_virt_class_W_B}
[\oM_\Gamma(W/B)]^\virt := \iota_\Gamma^! [\oM_{\overline{\Gamma}}(W/B)]^{\virt}\,,
\end{equation}
and virtual classes $[\shP_\Gamma(W/B)]^{\virt}$ and $[\shN_\Gamma(W/B)]^{\virt}$
by restriction to $\shP_\Gamma(W/B)$ and $\shN_\Gamma(W/B)$ respectively.

For every $t \in B \setminus \{0\}$, the fiber $\oM_\Gamma(W_t)$  of $\oM_\Gamma(W/B) \rightarrow B$ over $t$ is given as in 
\eqref{eq_decomp_moduli} by
\begin{equation} \label{eq_decomp_moduli_nodal}
    \oM_\Gamma(W_t)=\bigsqcup_{\substack{\Gamma_t\\ p_{W_t,*}\Gamma_t=\Gamma}} \oM_{\Gamma_t}(W_t)\,,\end{equation}
where the disjoint union is over the $W_t$-valued stable graphs 
$\Gamma_t$ satisfying the condition $p_{W_t,*}\Gamma_t=\Gamma$, and $\oM_{\Gamma_t}(W_t)$ are moduli stacks of
nodal stable maps to $W_t$.

By \cite[Theorem 6.4]{Fulton} applied to the fiber diagram
\begin{equation}\label{eq_com_1}
\begin{tikzcd}[column sep=1em]
\oM_\Gamma(W_t) \arrow[r] \arrow[d] &
\oM_\Gamma(W/B)
 \arrow[r]\arrow[d] &
\ooMM_\Gamma \arrow[d, "\iota_\Gamma"]
 \\
\oM_{\overline{\Gamma}}(W_t)
\arrow[r] \arrow[d] &
 \oM_{\overline{\Gamma}}(W/B)
\arrow[r] \arrow[d] & 
\ooMM_{\overline{\Gamma}}\\
t  \arrow[r,"\iota_{t}"]& B\,,
\end{tikzcd}
\end{equation}
the Gysin pullbacks $\iota_\Gamma^!$ and $\iota_t^!$ commute. Then, by
using \eqref{eq_decomp_classes}, we have
\begin{equation} \label{eq_decomp_classes_nodal}
    \iota_t^{!}[\oM_\Gamma(W/B)]^\virt = \sum_{\substack{\Gamma_t\\ 
p_{\Gamma_t,*}\Gamma_t=\Gamma}} \iota_{\Gamma_t,*}[\oM_{\Gamma_t}(W_t)]^\virt\,,
\end{equation}
where $\iota_{\Gamma_t}$ is the inclusion of $\oM_{\Gamma_t}(W_t)$ in $\oM_\Gamma(W_t)$
given by \eqref{eq_decomp_moduli}.

The degeneration formula for nodal Gromov--Witten theory expresses the virtual class 
\[ \iota_0^! [\oM_\Gamma(W/B)]^\virt \]
in terms of the virtual classes in relative Gromov--Witten theory of $(Y_1,D)$ and $(Y_2,D)$. 
As the component $\shN_\Gamma(W/B)$ is entirely over $0 \in B$ and the normal bundle to $0$ in $B$ is trivial, we have 
\begin{equation}\label{vanN}
\iota_0^! [\shN_\Gamma(W/B)]^{\virt}=0
\end{equation}
by the excess intersection formula \cite[Theorem 6.3]{Fulton},
and so 
\[ \iota_0^! [\oM_\Gamma(W/B)]^{\virt} = \iota_0^! [\shP_\Gamma(W/B)]^{\virt} \]
is entirely supported on the fiber $\shP_\Gamma(W_0)$ of $\shP_\Gamma(W/B)$ over $0 \in B$.
The vanishing \eqref{vanN}  plays an essential role in the proof of the nodal degeneration formula given in 
\S \ref{section_degeneration_proof}.

We introduced the set $\overline{\Omega}_{\Gamma}$ of equivalence classes of splittings of $\Gamma$ in Definition \ref{def_splitting_gamma}.
Let $\sigma=(\gamma_1,\gamma_2) \in \overline{\Omega}_{\Gamma}$. The $(Y_1,D)$-valued and $(Y_2,D)$-valued stable graphs $\gamma_1$ and $\gamma_2$ define
as in Definition \ref{def_relative}
the moduli stacks 
$\shP_{\gamma_1}(Y_1,D)$ and $\shP_{\gamma_2}(Y_2,D)$ of $\gamma_1$-marked and $\gamma_2$-marked relative stable maps to $(Y_1,D)$ and $(Y_2,D)$  with the condition that no $\gamma_i$-marked node is mapped to the singular locus of the expanded target. 
The evaluation at the relative marked points corresponding to the relative legs of $\gamma_1$ and $\gamma_2$
defines a morphism
\[ \shP_{\gamma_1}(Y_1,D) \times \shP_{\gamma_2}(Y_2,D) \rightarrow (D \times D)^{\ell(\sigma)} \,.\]
We denote by $\Delta_\sigma$ the diagonal morphism 
\[ \Delta_\sigma \colon  D^{\ell(\sigma)} \rightarrow (D \times D)^{\ell(\sigma)}\,,\]
and we form the fiber product
\[\begin{tikzcd}
\shP_{\gamma_1}(Y_1,D) \times_{D^{\ell(\sigma)}} \shP_{\gamma_2}(Y_2,D)  
\arrow[r]
\arrow[d]
&
\shP_{\gamma_1}(Y_1,D) \times \shP_{\gamma_2}(Y_2,D) \arrow[d]\\
D^{\ell(\sigma)}
\arrow[r,"\Delta_\sigma"]& (D \times D)^{\ell(\sigma)}\,.
\end{tikzcd}\]
For every scheme $S$, an $S$-point of this fiber product is the data of a stable map over $S$ to an expansion 
$Y_1[l_1]$ of $Y_1$ along $D$, and of a stable map over $S$ to an expansion $Y_2[l_2]$ of $Y_2$
along $D$, such that the positions of the relative marked points on $$D[l_1]=D[l_2] \simeq D$$ match. 
As the contact orders also match, we can glue these two stable maps together to obtain a stable map over $S$ to the expansion $W_0[l_1+l_2]$ of $W_0$. This stable map is naturally 
$\Gamma$-marked, and no $\Gamma$-marked node is mapped to the singular locus of 
$W_0[l_1+l_2]$. Hence, we obtain a gluing morphism
\begin{equation} \label{eq_gluing_morphism_2}
\Phi_\sigma \colon 
\shP_{\gamma_1}(Y_1,D) \times_{D^{\ell(\eta)}} \shP_{\gamma_2}(Y_2,D)  
\rightarrow \shP_{\Gamma}(W_0) \,.\end{equation}
We can finally state the degeneration formula for nodal Gromov--Witten theory:

\begin{theorem}[\textbf{Theorem \ref{thm_intro_degeneration}}]
\label{thm_degeneration_cycle}
Let $W \rightarrow B$ be a flat projective morphism from a smooth variety $W$ to a smooth connected curve $B$ with a distinguished point $0 \in B$ 
such that 
 \begin{enumerate}
     \item [(i)]
the fibers $W_t$ over $t \in B\setminus\{0\}$
 are smooth varieties, 
 \item[(ii)]
 the fiber $W_0$ over $0 \in B$ is the union of two smooth irreducible components
 $Y_1$ and $Y_2$ glued along a smooth connected divisor $D$. 
 \end{enumerate}
 Then, for every vertical $W$-valued stable graph $\Gamma$, 
\begin{equation} \label{eq_degeneration_cycle}  
\iota_0^! [\oM_\Gamma(W/B)]^\virt 
= \sum_{\sigma=(\gamma_1,\gamma_2)\in  
\overline{\Omega}_{\Gamma}} 
\frac{m(\sigma)}{|\Aut(\sigma)|} \Phi_{\sigma,*} \Delta_\sigma^! 
([\shP_{\gamma_1}(Y_1,D)]^\virt \times [\shP_{\gamma_2}(Y_2,D)]^\virt)\,,
\end{equation}
where $m(\sigma)$ is the multiplicity factor of the splitting $\sigma$ as in \eqref{Eq: multiplicity factor}, and the virtual classes $[\shP_{\gamma_1}(Y_1,D)]^\virt, [\shP_{\gamma_2}(Y_2,D)]^\virt$ are given by Definition \ref{Def: virtual classes}.

At the numerical level, for all
$k_i\in \Z_{\geq 0}$ and cohomology classes 
 $\alpha_i \in H^{*}(W)$ indexed by the legs $i \in L_\Gamma$ of $\Gamma$, 
and for all $k_h \in \Z_{\geq 0}$ indexed by the half-edges 
$h \in H_\Gamma \setminus L_\Gamma$ of $\Gamma$ which are not legs, we have for every $t \in B \setminus \{0\}$:
\begin{align} \label{eq_deg_num}   \sum_{\substack{\Gamma_t\\p_{W_t,*}\Gamma_t=\Gamma}}  \bigg\langle \prod_{i \in L_\Gamma} \tau_{k_i}(\alpha_i) \prod_{h \in H_{\Gamma}\setminus L_{\Gamma}}
\tau_{k_h} & 
\bigg\rangle_{\Gamma_t}^{W_t}= 
 \nonumber\\
  \sum_{\sigma=(\gamma_1,\gamma_2)\in  
\overline{\Omega}_\Gamma} 
\frac{m(\sigma)}{|\Aut(\sigma)|}
\sum_{j_1,\dots,j_{\ell(\sigma)}}
(-1)^\epsilon &
\left\langle \prod_{i \in L_{\gamma_1,I}}\tau_{k_i}(\alpha_i) \prod_{h \in H_{\gamma_1}\setminus L_{\gamma_1}}
\tau_{k_h}\, \Bigg|\ \prod_{i=1}^{\ell(\sigma)} 
\delta_{j_i} \right\rangle_{\gamma_1}^{(Y_1,D)} \nonumber
\\
\boldsymbol{\cdot}  & 
\left\langle \prod_{i=1}^{\ell(\sigma)} 
\delta_{j_i}^\vee \ \Bigg|\,  \prod_{i \in L_{\gamma_2,I}}\tau_{k_i}(\alpha_i)
\prod_{h \in H_{\gamma_2}\setminus L_{\gamma_1}}
\tau_{k_h}
\right\rangle_{\gamma_2}^{(Y_2,D)}\,,
\end{align}
where the nodal relative Gromov--Witten invariants on the right-hand side are as in Definition \ref{def_nodal_relative_gw}, $\alpha_i$ denotes the restriction of 
$\alpha_i$ to $W_t$, $Y_1$, and $Y_2$ respectively,
$(\delta_j)_j$ is a basis of $H^{\star}(D)$,
$(\delta_j^\vee)_j$ is the Poincaré-dual basis,
and 
$(-1)^\epsilon$ is the sign determined formally by the equality 
\[ \prod_{i\in L_\Gamma} \alpha_i =(-1)^\epsilon 
\prod_{i\in L_{\gamma_1,I}} \alpha_i \prod_{i\in L_{\gamma_2,I}} \alpha_i \,.\]
\end{theorem}

The proof of \eqref{eq_degeneration_cycle} in Theorem \ref{thm_degeneration_cycle} is given in 
\S \ref{section_degeneration_technical}\,-\,\ref{section_degeneration_proof} below.
The numerical version \eqref{eq_deg_num} follows immediately from
\eqref{eq_decomp_classes_nodal}, \eqref{eq_degeneration_cycle}, and the Künneth decomposition of the class of the diagonal
$\Delta_\sigma$.

\subsection{Preliminary results}
\label{section_degeneration_technical}
We prove here a number of technical results which will be used in 
\S \ref{section_degeneration_proof} to prove Theorem \ref{thm_degeneration_cycle}.

For every splitting $\eta=(G_1,G_2) \in
\overline{\Omega}_{\overline{\Gamma}}$, we define 
\begin{eqnarray*}
\oM_{G_1,G_2}^\Delta(W_0) & := & \oM_{G_1}(Y_1,D) \times_{D^{\ell(\eta)}} \oM_{G_2}(Y_2,D) \,, \\
\oM_{G_1,G_2}(W_0) & := & \oM_{G_1}(Y_1,D) \times \oM_{G_2}(Y_2,D) \,, 
\end{eqnarray*}
We define the moduli stacks $\oM_{G_1,G_2,\Gamma}(W_0)$ and 
$\ooMM_{G_1,G_2,\Gamma}$ by the fiber diagram
\begin{equation}\label{eq_diagr}
\begin{tikzcd}
\oM_{G_1,G_2,\Gamma}(W_0)
\arrow[r]
\arrow[d]
&
\oM_{G_1,G_2}(W_0)
\arrow[d]\\
\ooMM_{G_1,G_2,\Gamma} \arrow[r,"\iota_{G_1,G_2,\Gamma}"] \arrow[d]
&
\ooMM_{G_1} \times \ooMM_{G_2} \arrow[d]
\\
\ooMM_\Gamma 
\arrow[r,"\iota_\Gamma"]& \ooMM_{\overline{\Gamma}}\,,
\end{tikzcd}
\end{equation}
where the morphism 
$\ooMM_{G_1} \times \ooMM_{G_2} \rightarrow \ooMM_{\overline{\Gamma}}$ 
is defined by gluing together a $G_1$-marked stable curve and a $G_2$-marked stable curves along their relative marked points. 

We have a disjoint union decomposition 
\[ \ooMM_{G_1,G_2,\Gamma} = \mathfrak{P}_{G_1,G_2,\Gamma} \sqcup \mathfrak{N}_{G_1,G_2,\Gamma}\,,\]
where $\mathfrak{P}_{G_1,G_2,\Gamma}$ is defined by the condition that the 
$\Gamma$-marked nodes are not nodes created by the gluing of the 
$G_1$-marked and $G_2$-marked stable curves.
Using the predeformability condition, we obtain a version of Lemma \ref{lem_key}, ensuring that we have a disjoint union decomposition
\[ \oM_{G_1,G_2,\Gamma}(W_0) = \shP_{G_1,G_2,\Gamma}(W_0) \sqcup \shN_{G_1,G_2,\Gamma}(W_0) \,.\]
Here,
$\shP_{G_1,G_2,\Gamma}(W_0)$ is defined by the condition that the 
$\Gamma$-marked nodes are not distinguished: the $\Gamma$-marked nodes are distinct from the nodes mapped to the singular locus of the expansions of $Y_1$ or $Y_2$ and distinct from the nodes 
newly created by the gluing of the relative marked points. 
In particular, the morphism $\oM_{G_1,G_2,\Gamma}(W_0) \rightarrow \ooMM_{G_1,G_2,\Gamma}$ restricts to a morphism 
$\shP_{G_1,G_2,\Gamma}(W_0) \rightarrow \mathfrak{P}_{G_1,G_2,\Gamma}$.

We define similarly 
\[ \oM_{G_1,G_2,\Gamma}^\Delta(W_0) = \shP_{G_1,G_2,\Gamma}^\Delta(W_0) \sqcup \shN_{G_1,G_2,\Gamma}^\Delta(W_0) \]
by replacing $\oM_{G_1,G_2}(W_0)$ by $\oM_{G_1,G_2}^\Delta(W_0)$ in \eqref{eq_diagr}.
The definition of the gluing morphism $\Phi_\eta \colon \oM_{G_1,G_2}^\Delta(W_0) \rightarrow \oM_{\overline{\Gamma}}(W_0)$
extends to define a gluing morphism 
\begin{equation}\label{eq_gluing_morphism_3}
\Phi_{\eta,\Gamma} \colon \oM_{G_1,G_2,\Gamma}^\Delta(W_0) \rightarrow \oM_\Gamma(W_0)\,.
\end{equation}

\begin{lemma} \label{lem_alpha}
For every cycle $\alpha$ on 
$\oM_{G_1,G_2}^\Delta(W_0)$, we have
\[ \iota_\Gamma^! \Phi_{\eta,*}\alpha = \Phi_{\eta,\Gamma,*} \iota_\Gamma^! \alpha \]
on $\oM_{G_1,G_2,\Gamma}^\Delta(W_0)$.
\end{lemma}

\begin{proof}
The result follows from the compatibility of the Gysin pull-back and proper push-forward
\cite[Theorem 6.2 (a)]{Fulton} 
applied to the fiber diagram 
\begin{equation}\begin{tikzcd}
\oM_{G_1,G_2,\Gamma}^\Delta(W_0)  \arrow[r]
 \arrow["\Phi_{\eta,\Gamma}"', d]&
 \oM_{G_1,G_2}^\Delta(W_0) \arrow[d,"\Phi_\eta"]\\
\oM_\Gamma(W_0) \arrow[r] 
\arrow[d]
& \oM_{\overline{\Gamma}}(W_0)
\arrow[d]\\
\ooMM_{\Gamma}
\arrow[r,"\iota_{\Gamma}"] 
&\ooMM_{\overline{\Gamma}}\,.
\end{tikzcd}\end{equation}
\end{proof}

\begin{lemma} \label{lem_beta}
For every cycle $\beta$ on $\oM_{G_1,G_2}(W_0)$, we have 
 \[ \iota_{\Gamma}^! \Delta_\eta^! \beta = \Delta_\eta^! \iota_{\Gamma}^! \beta\]
on $\oM_{G_1,G_2,\Gamma}^\Delta(W_0)$. 
\end{lemma}

\begin{proof}
The Gysin pullbacks $\iota_\Gamma^!$ and $\Delta_\eta^!$ commute by
\cite[Theorem 6.4]{Fulton} applied to the fiber diagram 
 \begin{equation} 
\begin{tikzcd}[column sep=1em]
\oM_{G_1,G_2,\Gamma}^\Delta(W_0)  \arrow[r] \arrow[d] &
\oM_{G_1,G_2}^\Delta(W_0)
 \arrow[r]\arrow[d] &
 D^{\ell(\eta)}\arrow[d, "\Delta_\eta"]
 \\
\oM_{G_1,G_2,\Gamma}(W_0)
\arrow[r] \arrow[d] &
\oM_{G_1,G_2}(W_0)
\arrow[r] \arrow[d] & 
(D \times D)^{\ell(\eta)}\\
\ooMM_\Gamma  \arrow[r,"\iota_{\Gamma}"]& \ooMM_{\overline{\Gamma}}\,.
\end{tikzcd}
\end{equation}
\end{proof}

\begin{lemma} \label{lem_decomposition}
There are disjoint union decompositions
\begin{equation} \label{eq_dec_1}
\mathfrak{P}_{G_1,G_2,\Gamma}= \bigsqcup_{\substack{\sigma=(\gamma_1,\gamma_2) \in \overline{\Omega}_\Gamma\\ \overline{\gamma}_1=G_1, \overline{\gamma_2}=G_2}} 
\ooMM_{\gamma_1}\times \ooMM_{\gamma_2}\,,
\end{equation}
\begin{equation} \label{eq_dec_2}
\shP_{G_1,G_2,\Gamma}(W_0) = \bigsqcup_{\substack{\sigma=(\gamma_1,\gamma_2) \in \overline{\Omega}_\Gamma\\ \overline{\gamma}_1=G_1, \overline{\gamma_2}=G_2}} 
\shP_{\gamma_1}(Y_1,D) \times \shP_{\gamma_2}(Y_2,D) \,,\end{equation}
where we denote by $\overline{\gamma}_1$
(resp.\ $\overline{\gamma}_2$) the graph without edges obtained from 
$\gamma_1$ (resp.\ $\gamma_2$) by contraction of all edges, as in Definition \ref{def_contraction}.
\end{lemma}

\begin{proof}
We explain how to prove \eqref{eq_dec_2}. The proof of \eqref{eq_dec_1}
is similar and in fact simpler.
For every splitting $\sigma=(\gamma_1,\gamma_2) \in \overline{\Omega}_\Gamma$ such that 
$\overline{\gamma}_1=G_1$ and $\overline{\gamma}_2=G_2$, there is a natural morphism 
\[ \shP_{\gamma_1}(Y_1,D) \times \shP_{\gamma_2}(Y_2,D)\rightarrow \shP_{G_1,G_2,\Gamma}(W_0) \,.\]
Indeed, a curve with $\gamma_i$-marking has in particular a $\Gamma_i$-marking, and a curve obtained by gluing two curves with $\gamma_1$ and $\gamma_2$-markings along their relative marked points has a natural $\Gamma$-marking by Definition \ref{def_splitting_gamma} of a splitting of $\Gamma$.

Conversely, a point of $\shP_{G_1,G_2}(W_0)$ consists
of the data of relative stable maps  
$$C_1 \rightarrow Y_1[l_1]\,, \ \ \ C_2 \rightarrow Y_2[l_2]\, $$ together with  a $\Gamma$-structure on the curve $C$ obtained by gluing $C_1$ and $C_2$ along their relative marked points.
The data must satisfy the following condition:
every $\Gamma$-marked node of $C$ is distinct from the nodes of $C_1$ or $C_2$ mapping to the singular locus of $Y_1[l_1]$ or $Y_2[l_2]$ and distinct from the nodes in $C_1 \cap C_2$ created by the gluing. 
In particular, no $\Gamma$-marked node of $C$ is a node created by the gluing. Hence, splitting $C$ into $C_1$ and $C_2$ induces a splitting $\sigma=(\gamma_1,\gamma_2)$ of $\Gamma$ and 
$\gamma_i$-structures on $C_i$.
\end{proof}

\begin{lemma}\label{lem_decomposition_virtual}
Under the identification \eqref{eq_dec_2}, the following equality of cycles holds on $\shP_{G_1,G_2,\Gamma}(W_0)$: 
\begin{align*}
   \iota_\Gamma^!([\oM_{G_1}(Y_1,D)]^\virt \times [\oM_{G_2}(Y_2,D)]^\virt)|_{\shP_{G_1,G_2,\Gamma}(W_0)} & \\ 
   = \sum_{\substack{\sigma=(\gamma_1,\gamma_2) \in  \overline{\Omega}_\Gamma\\ \overline{\gamma}_1=  G_1, \overline{\gamma}_2=G_2}} &
[\shP_{\gamma_1}(Y_1,D)]^{\virt} \times [\shP_{\gamma_2}(Y_2,D)]^\virt\,.
\end{align*}
\end{lemma}

\begin{proof}
As \eqref{eq_diagr} is a fiber diagram, we have $\iota_\Gamma^! = \iota_{G_1,G_2,\Gamma}^!$ by \cite[Theorem 6.2 c)]{Fulton}.
Under the identification \eqref{eq_dec_1}, the restriction of 
$\iota_{G_1,G_2,\Gamma}$ to $\ooMM_{\gamma_1} \times \ooMM_{\gamma_2}$
is $(\iota_{\gamma_1}, \iota_{\gamma_2})$. Then, using the identification \eqref{eq_dec_2}, the restriction of  $$\iota_\Gamma^!\left([\oM_{G_1}(Y_1,D)]^\virt \times [\oM_{G_2}(Y_2,D)]^\virt\right)$$ 
to $\shP_{\gamma_1}(Y_1,D) \times \shP_{\gamma_2}(Y_2,D)$ is equal to the restriction of $$\iota_{\gamma_1}^! [\oM_{G_1}(Y_1,D)]^\virt \times
\iota_{\gamma_2}^! [\oM_{G_2}(Y_2,D)]^\virt$$ to 
$\shP_{\gamma_1}(Y_1,D) \times \shP_{\gamma_2}(Y_2,D)$, and so is equal to 
$[\shP_{\gamma_1}(Y_1,D)]^{\virt} \times [\shP_{\gamma_2}(Y_2,D)]^\virt$
by Definition \ref{Def: virtual classes} and \eqref{eq_def_relative}.
\end{proof}

\subsection{Proof of the nodal degeneration formula}
\label{section_degeneration_proof}

We will prove \eqref{eq_degeneration_cycle} in Theorem \ref{thm_degeneration_cycle} by applying 
$\iota_\Gamma^!$ to both sides of Jun Li's degeneration formula given by 
\begin{equation}  \label{li_degeneration}
\iota_0^! [\oM_{\overline{\Gamma}}(W/B)]^\virt 
= \sum_{\eta=(G_1,G_2)\in  
\overline{\Omega}_{\overline{\Gamma}}} 
\frac{m(\eta)}{|\Aut(\eta)|} \Phi_{\eta,*} \Delta_\eta^! 
([\oM_{G_1}(Y_1,D)]^\virt \times [\oM_{G_2}(Y_2,D)]^\virt)\,,
\end{equation}
as reviewed in 
\S \ref{section_degeneration_review}. By \cite[Theorem 6.4]{Fulton} applied to the fiber diagram \eqref{eq_com_1} for $t=0$, 
the Gysin pullbacks $\iota_\Gamma^!$ and $\iota_0^!$ commute. Then, via using \eqref{eq_virt_class_W_B}, we obtain
\begin{equation}\label{eq_proof_1}
\iota_\Gamma^! \iota_0^! [\oM_{\overline{\Gamma}}(W/B)]^\virt= \iota_0^! [\oM_\Gamma(W/B)]^\virt\,.\end{equation}
In other words, \ $\iota_\Gamma^!$ applied to the left side of \eqref{li_degeneration}
is the left side of \eqref{eq_degeneration_cycle}. It remains to match
the right sides of the formulas.

We observed in \eqref{vanN} that $\iota_0^! [\oM_\Gamma(W/B)]^\virt$ is supported on the component $\shP_\Gamma(W_0)$ of $\oM_\Gamma(W_0)$. Hence, by
\eqref{eq_proof_1}, the result of $\iota_\Gamma^!$ 
applied to the right side of \eqref{li_degeneration} is also supported on 
$\shP_\Gamma(W_0)$. To prove \eqref{eq_degeneration_cycle}, it is 
therefore enough to compute
the restriction 
\[ \left(\iota_{\Gamma}^! \Phi_{\eta,*} \Delta_\eta^!([\oM_{G_1}(Y_1,D)]^\virt \times [\oM_{G_2}(Y_2,D)]^\virt) \right)|_{\shP_\Gamma(W_0)} \]
to $\shP_\Gamma(W_0)$
for every $\eta =(G_1, G_2) \in \overline{\Omega}_{\overline{\Gamma}}$.

Using Lemmas \ref{lem_alpha} and \ref{lem_beta}, we have 
\begin{align*}
\iota_{\Gamma}^! \Phi_{\eta_{*}} \Delta_\eta^!([\oM_{G_1}(Y_1,D)]^\virt \times [\oM_{G_2}(Y_2,D)]^\virt) &
\\
= 
\Phi_{\eta, \Gamma,*} \Delta_\eta^! 
\iota_{\Gamma}^! &  ([\oM_{G_1}(Y_1,D)]^\virt \times [\oM_{G_2}(Y_2,D)]^\virt)\,.    
\end{align*}

A simple observation is the following: {\em for every cycle $\alpha$ on $\oM_{G_1,G_2,\Gamma}(W_0)$ supported on 
$\shN_{G_1,G_2,\Gamma}(W_0)$, the cycle
$\Delta_\eta^! \alpha$ is supported on 
$\shN_{G_1,G_2,\Gamma}^\Delta(W_0)$}.
By applying the observation to the part of 
\[\iota_{\Gamma}^! ([\oM_{G_1}(Y_1,D)]^\virt \times [\oM_{G_2}(Y_2,D)]^\virt)\,\] 
supported on $\shN_{G_1,G_2,\Gamma}(W_0)$,
we obtain
\begin{multline*}
\big( 
\Phi_{\eta, \Gamma,*} \Delta_\eta^!
\iota_{\Gamma}^! ([\oM_{G_1}(Y_1,D)]^\virt \times [\oM_{G_2}(Y_2,D)]^\virt)\big)|_{\shP_\Gamma(W_0)}  \\
= 
 \Phi_{\eta, \Gamma,*} \Delta_\eta^! 
 \Big(\iota_\Gamma^! ([\oM_{G_1}(Y_1,D)]^\virt  \times [\oM_{G_2}(Y_2,D)]^{\virt} ) 
|_{\shP_{G_1,G_2,\Gamma}(W_0)} \Big)\,.
\end{multline*}
Finally, by Lemma \ref{lem_decomposition_virtual}, we have 
\begin{align*}
   \iota_\Gamma^!([\oM_{G_1}(Y_1,D)]^\virt \times [\oM_{G_2}(Y_2,D)]^\virt)|_{\shP_{G_1,G_2,\Gamma}(W_0)} & \\ 
   = \sum_{\substack{\sigma=(\gamma_1,\gamma_2) \in  \overline{\Omega}_\Gamma\\ \overline{\gamma}_1=  G_1, \overline{\gamma}_2=G_2}} &
[\shP_{\gamma_1}(Y_1,D)]^{\virt} \times [\shP_{\gamma_2}(Y_2,D)]^\virt\,, 
\end{align*}
which completes the proof that $\iota_\Gamma^!$ applied to the right side of 
\eqref{li_degeneration} is equal to the right side of \eqref{eq_degeneration_cycle}. $\hfill \blacklozenge$

\section{Splitting formula for nodal relative Gromov--Witten theory}
\label{section_splitting}
We prove here a new splitting formula for nodal relative Gromov--Witten theory
which differs from the standard splitting formula for nodal absolute
Gromov-Witten theory by certain rubber terms.

We review 
the construction of the virtual class in relative Gromov--Witten theory in \S \ref{section_virtual_class} and we describe the virtual class in nodal 
relative Gromov--Witten theory in \S \ref{section_virtual_class_nodal}.
After a review of the 
statement  of the splitting formula for nodal absolute Gromov--Witten invariants  in 
\S \ref{sec: splitting for nodal}, 
the splitting formula in the
nodal relative case is proven in 
\S \ref{section_splitting_nodal} and recast in 
a more explicit form in \S \ref{section_splitting_explicit}.
In \S \ref{section_nodal_relative_absolute}, we combine the splitting formula with the
nodal rubber calculus of
\S \ref{section_nodal_rubber}
to prove that nodal relative Gromov--Witten invariants can be effectively reconstructed from absolute Gromov--Witten invariants.

\subsection{Relative Gromov--Witten theory}
\label{section_virtual_class}
We review here the construction{\footnote{There are by now many approaches to the foundations of
relative Gromov-Witten theory. A unified presentation with all of the comparison results (which show
equivalences of the virtual fundamental classes) can be found in the upcoming
paper \cite{HMPW}.}}
of the virtual 
class in relative Gromov--Witten theory based on working relatively to the moduli space of maps to a universal target. For details we refer to \cite[\S 5]{ACW} and \cite[\S 3.2]{AMW}. 

In absolute Gromov--Witten theory, we use the perfect obstruction theory relative to 
the forgetful morphism 
\begin{equation*}  
\epsilon \colon \oM_{g,n,\beta}(X) \rightarrow \ooMM_{g,n,\beta} 
\end{equation*}
defined in \eqref{eq_epsilon}, given by 
\begin{equation}
\nonumber
(R\pi_{*}f^{*}T_X)^\vee\,,
\end{equation}
where $\pi \colon C \rightarrow \oM_{g,n,\beta}(X)$
is the universal curve, $f \colon C \rightarrow X$ is the universal stable map, and 
$T_X$ is the tangent bundle of $X$, see 
\cite[Proposition 6.2]{BF} and \cite[Proposition 5]{behrend97gw}.
The virtual class is defined by the corresponding virtual pull-back \cite{manolache2008virtual}
of the 
ordinary fundamental class on the equidimensional stack $\ooMM_{g,n,\beta}$:
\begin{equation} 
\label{eq_virtual_absolute}
\nonumber
[\oM_{g,n,\beta}(X)]^\virt := \epsilon^! [\ooMM_{g,n,\beta}] \,.
\end{equation}

Let $X$ be a smooth projective variety over $\CC$, 
$$D \subset X$$ 
a smooth divisor, and $G$ an 
$(X,D)$-valued stable graph without edges.
In relative Gromov--Witten theory, Jun Li \cite{JunLi} defined the virtual class 
$[\oM_G(X,D)]^\virt$
using a 
perfect obstruction theory which can be viewed as being relative to the morphism
\[ \epsilon_G \colon \oM_G(X,D) \rightarrow \ooMM_G \,,\]
defined in \eqref{eq_epsilon_G}. The definition of this perfect obstruction theory is quite complicated due to the subtle nature of the predeformability condition (see Definition \ref{def_relative_map} (iii)). Generally, perfect obstruction theories become easier to describe when taken relative to the largest available smooth geometry. 
For example, Graber--Vakil gave a more compact description of the virtual class by defining a perfect obstruction theory relative to $\ooMM_G \times \mathcal{T}$, where 
$\mathcal{T}$ is the stack of expanded degenerations \cite[\S 2.8]{GraberVakil}. More recently,  Abramovich-Cadman-Wise have provided an alternative description obtained by working relatively to the space of maps to a universal target
\cite[\S 5]{ACW}, see also
\cite[\S 3.2]{AMW}. 
We follow here the Abramovich-Cadman-Wise approach which is technically simpler.

When working with pairs $(X,D)$, where $D$ is a smooth divisor, the \emph{universal target} is 
the pair of stacks $(\mathscr{A},\mathscr{D})$, where $$\mathscr{A}:= [\AA^1/\mathbb{G}_m]$$ is the classifying stack of line bundles with sections 
and $\mathscr{D}:=[0/\mathbb{G}_m]$. For every pair $(X,D)$, we have a canonical morphism 
of pairs 
\begin{equation} \label{eq_can_morphism}
(X,D) \rightarrow (\mathscr{A},\mathscr{D})\,,
\end{equation}
induced by the line bundle $\cO_X(D)$ with its section vanishing on $D$.

Let 
$\ooMM_G(\mathscr{A},\mathscr{D})$ be the moduli stack of relative prestable maps to 
$(\mathscr{A},\mathscr{D})$, defined as in 
Definition \ref{def_relative_map}
using expanded degenerations and the predeformability condition, but without the stability condition (Definition \ref{def_relative_map}(ii)). 
The crucial point is that $\ooMM_G(\mathscr{A},\mathscr{D})$ 
is equidimensional by \cite[Lemma 4.1.2]{ACW} and thus admits an ordinary fundamental class
$[\ooMM_G(\mathscr{A},\mathscr{D})]$. The natural morphism 
\eqref{eq_can_morphism} induces a morphism 
\begin{equation}\label{eq_epsilon_relative}
\nonumber
\eta_G \colon \oM_G(X,D) \rightarrow \ooMM_G(\mathscr{A},\mathscr{D})\,.
\end{equation}
By \cite[\S 5]{ACW}, the virtual
class in relative Gromov--Witten theory can be defined using the perfect obstruction theory relative to 
$\eta_G$ given by 
\begin{equation}\label{eq_pot_relative}
(R\pi_{*}f^{*}T_{(X,D)})^\vee\,,
\end{equation}
where $\pi \colon C \rightarrow \oM_G(X,D)$ is the universal curve, 
$f \colon C \rightarrow \mathcal{X} \rightarrow X$ is the composition of the universal relative stable map with the contraction of the expanded target on $X$, and $T_{(X,D)}$ is the log tangent
bundle of the pair $(X,D)$. The virtual class is then defined by
\begin{equation} 
\label{eq_virtual_relative}
[\oM_G(X,D)]^\virt := \eta_G^! [\ooMM_G(\mathscr{A},\mathscr{D})] \,,
\end{equation}
where $\eta_G^!$ is the corresponding virtual pull-back
\cite{manolache2008virtual}. 

\subsection{Nodal relative Gromov--Witten theory (revisited)}
\label{section_virtual_class_nodal}
The virtual class in nodal relative Gromov--Witten theory was defined in
Definition \ref{Def: virtual classes}.
A second approach to the nodal relative theory is provided by the geometric perspective of 
Abramovich-Cadman-Wise. We show here that the two definitions agree.

In Definition \ref{def_relative},
for every $(X,D)$-valued stable graph $\Gamma$,  we introduced the stack $\shP_\Gamma(X,D)$ of $\Gamma$-marked relative stable maps whose $\Gamma$-marked nodes are not distinguished (the $\Gamma$-marked nodes are not mapped to the singular locus of the expanded target). Let $\ooMM_\Gamma(\mathscr{A},\mathscr{D})$ denote the moduli stack
 of $\Gamma$-marked prestable maps to $(\mathscr{A},\mathscr{D})$.  Let $\mathfrak{P}_\Gamma(\mathscr{A},\mathscr{D})$ 
 be the substack of $\Gamma$-marked prestable maps to $(\mathscr{A},\mathscr{D})$ whose $\Gamma$-marked nodes are not distinguished. 
The latter requirement on nodes ensures that the proof of \cite[Lemma 4.1.2]{ACW} applies to $\mathfrak{P}_\Gamma(\mathscr{A},\mathscr{D})$. Hence,  
$\mathfrak{P}_\Gamma(\mathscr{A},\mathscr{D})$ is equidimensional and has an ordinary fundamental class
$[\mathfrak{P}_\Gamma(\mathscr{A},\mathscr{D})]$.
Moreover, the forgetful morphism 
$$\mathfrak{P}_\Gamma(\mathscr{A},\mathscr{D}) \rightarrow \ooMM_{\overline{\Gamma}}(\mathscr{A},\mathscr{D})$$ is 
a local complete intersection 
of codimension $|E_\Gamma|$,  so 
\begin{equation}\label{eq_pullback}
(\iota_\Gamma^![\ooMM_{\overline{\Gamma}}(\mathscr{A},\mathscr{D})])|_{\mathfrak{P}_\Gamma(\mathscr{A},\mathscr{D})}
=[\mathfrak{P}_\Gamma(\mathscr{A},\mathscr{D})]\, .\end{equation}

On the other hand, the natural morphism 
\eqref{eq_can_morphism} induces a morphism 
\begin{equation} 
\label{eq_epsilon_nodal}
\nonumber
\eta_\Gamma \colon \shP_\Gamma(X,D) \rightarrow \mathfrak{P}_\Gamma(\mathscr{A},\mathscr{D})\,,
\end{equation} 
and we obtain a perfect obstruction theory relative to $\eta_\Gamma$ by pulling-back  \eqref{eq_pot_relative}.
By Definition \ref{def_M_Gamma}, we have a  fiber diagram
\begin{equation}
\nonumber
\begin{tikzcd}
\oM_\Gamma(X,D)  \arrow[r]
 \arrow["\eta_\Gamma", d]&
\oM_{\overline{\Gamma}}(X,D) \arrow[d,"\eta_{\overline{\Gamma}}"]\\
\ooMM_\Gamma(\mathscr{A},\mathscr{D}) \arrow[r] 
\arrow[d]
& \mathfrak{M}_{\overline{\Gamma}}(\mathscr{A},\mathscr{D})
\arrow[d]\\
\ooMM_{\Gamma}
\arrow[r,"\iota_{\Gamma}"] 
&\ooMM_{\overline{\Gamma}}\,,
\end{tikzcd}\end{equation}
and by Definition \ref{def_relative}, $\shP_\Gamma(X,D)$ and $\mathfrak{P}_\Gamma(\mathscr{A},\mathscr{D})$
are unions of connected components of $\oM_\Gamma(X,D)$ and $\ooMM_\Gamma(\mathscr{A},\mathscr{D})$
respectively.
By \eqref{eq_virtual_relative}, 
property \eqref{eq_def_relative}, and the pull-back relation \eqref{eq_pullback},
we conclude the following result.
\begin{lemma} We have
$\label{eq_virtual_relative_nodal}
[\shP_\Gamma(X,D)]^\virt= \eta_\Gamma^! [\mathfrak{P}_\Gamma(\mathscr{A},\mathscr{D})]$ .
\end{lemma}

\subsection{Splitting formula for the nodal absolute theory}
\label{sec: splitting for nodal}
Let $X$ be a smooth projective variety over $\C$, and let
$\Gamma$ be an $X$-valued stable graph. 
Let $e=\{h,h'\}$ be an edge of $\Gamma$ connecting the
half-edges $h$ and $h'$.
Cutting $\Gamma$ along $e$ defines a new $X$-valued stable graph $\Gamma \setminus e$ in which the half-edges $h$ and $h'$ become legs. Cutting a $\Gamma$-marked stable map along the node marked by $e$ defines a 
\emph{splitting morphism} 
\[ s_e \colon \oM_\Gamma(X) \rightarrow \oM_{\Gamma \setminus e}(X) \,.\]
The evaluations at the two legs of 
$\Gamma \setminus e$ defined by $h$ and $h'$ define a morphism 
\[ \ev_e \colon \oM_{\Gamma \setminus e}(X) \rightarrow X \times X\,.\]
Denote the diagonal morphism by
$$\Delta \colon X \rightarrow X \times X\, .$$ 
As a node in the domain curve of a stable map can be created by gluing two marked points mapped to the same point, we have a fiber diagram 
\[\begin{tikzcd}
\oM_\Gamma(X)
\arrow[r, "s_e"]
\arrow[d]
&
\oM_{\Gamma \setminus e}(X)
\arrow[d, "\ev_e"]\\
X
\arrow[r,"\Delta"]& X \times X\,,
\end{tikzcd}\]
The 
\emph{splitting formula} of Gromov-Witten theory is 
\begin{equation} \label{eq_splitting}
 [\oM_\Gamma(X)]^\virt = \Delta^! [\oM_{\Gamma \setminus e}(X)]^{\virt} \,.\end{equation}
Splitting is a basic property of the virtual 
class of the moduli space of stable maps, see 
\cite[Axiom III]{behrend97gw} for the proof.

\subsection{Splitting formula for the nodal relative theory}
\label{section_splitting_nodal}

Let $X$ be a smooth projective variety over $\C$ with a smooth divisor  $D\subset X$.
Let $\Gamma$ be an $(X,D)$-valued stable graph, and 
let $e=\{h,h'\}$ be an edge of $\Gamma$, connecting the half-edges $h$ and $h'$.
Cutting $\Gamma$ along $e$ defines a new $(X,D)$-valued stable graph $\Gamma \setminus e$ in which the half-edges $h$ and $h'$ become legs. 
Cutting a $\Gamma$-marked relative stable map along the node marked by $e$ defines a
\emph{splitting morphism} 
\[ s_e \colon \shP_\Gamma(X,D) \rightarrow \shP_{\Gamma \setminus e}(X,D) \,.\]

Let $(X,D)^2$ be the moduli space parameterizing ordered pairs of points in the pair $(X,D)$:
$X$ expands along $D$ when the points approach $D$. The configurations of points in the bubbles are considered up to the 
scaling action of $\mathbb{G}_m$ on the $\PP^1$-fibers of the bubble. 
As a variety, 
$(X,D)^2$ is the blow-up of $X \times X$ along $D \times D$.
The space $(\PP^1,0)^2$ is illustrated as an example in Figure \ref{Fig: mod}.

\begin{figure}
\resizebox{.9\linewidth}{!}{\input{mod.pspdftex}}
\caption{The moduli space 
$(\PP^1,0)^2=\mathrm{Bl}_{(0,0)}(\PP^1 \times \PP^1)$
of ordered pairs of points in 
$(\PP^1,0)$}
\label{Fig: mod}
\end{figure}

The space $(X,D)^2$ and the natural $n$-pointed generalization $(X,D)^n$ have previously appeared 
in relative Gromov--Witten theory in the formulation of the Gromov--Witten/Pairs correspondence for relative theories \cite[\S 1.2]{PP}.

Let $\Delta_{\rel} \colon X \rightarrow (X,D)^2$ be the strict transform in $(X,D)^2$ of the diagonal 
$$\Delta \colon X \rightarrow X \times X\, .$$ 
Geometrically, $\Delta_{\rel}$ is the locus in $(X,D)^2$ parameterizing pairs of coincident points in $(X,D)$.
The evaluations at the two legs of 
$\Gamma \setminus e$ defined by $h$ and $h'$ define a natural morphism 
\[ \ev_e \colon \shP_{\Gamma \setminus e}(X,D) \rightarrow (X,D)^2\,,\]
where bubbles of the expanded target are contracted if they do not contain any of the images of the marked points corresponding to $h$ and $h'$.

As $\Gamma \setminus e$-marked points of nodal relative stable maps parameterized by 
$\shP_{\Gamma \setminus e}(X,D)$ and $\Gamma$-marked nodes of nodal relative
stable maps parameterized by $\shP_\Gamma(X,D)$ are not mapped in the singular locus of the expanded targets, we have a fiber diagram
\[\begin{tikzcd}
\shP_\Gamma(X,D)
\arrow[r, "s_e"]
\arrow[d]
&
\shP_{\Gamma \setminus e}(X,D)
\arrow[d, "\ev_e"]\\
X
\arrow[r,"\Delta_{\rel}"]& (X,D)^2\,.
\end{tikzcd}\]

Our first version of the splitting formula in nodal relative Gromov--Witten theory is 
formally similar to the splitting formula \eqref{eq_splitting} in absolute Gromov--Witten theory.

\begin{theorem} \label{thm_splitting_1}
For every $(X,D)$-valued stable graph $\Gamma$ and every edge $e$ of $\Gamma$, 
\[ [\shP_\Gamma(X,D)]^\virt =  \Delta_{\rel}^! [\shP_{\Gamma \setminus e}(X,D)]^\virt \,.\]
\end{theorem}

\begin{proof}
We will work relatively to the universal target
$$(\mathscr{A},\mathscr{D})
=([\AA^1/\mathbb{G}_m], [0/\mathbb{G}_m])\,$$
as reviewed
in \S \ref{section_virtual_class}.
The first step is to prove that the splitting morphism induces an isomorphism
\begin{equation}
\label{eq_isom}
\mathfrak{P}_\Gamma(\mathscr{A},\mathscr{D})\simeq \mathfrak{P}_{\Gamma \setminus e}(\mathscr{A},\mathscr{D})\,.\end{equation}
When the target does not expand, a $\Gamma \setminus e$-marked relative prestable map 
to $(\mathscr{A},\mathscr{D})$ consists of a $\Gamma \setminus e$-marked prestable curve $C$ and 
a morphism, $$C \rightarrow \mathscr{A}=[\AA^1/\mathbb{G}_m]\,, $$
given by a line bundle $L$ on $C$ with a section $\sigma$.
The marked points of $C$ are away from the divisor $\mathscr{D}$ and hence are not zeros of $\sigma$. 
In particular, $\sigma$ does not vanish at the two marked points $x_h$ and $x_{h'}$ defined by the legs $h$ and $h'$ of $\Gamma \setminus e$. By gluing $x_h$ and $x_{h'}$ together, we obtain a prestable curve $\overline{C}$. We construct a line bundle $\overline{L}$ on $\overline{C}$ by gluing the fibers 
$L|_{x_h}$ and $L|_{x_{h'}}$ by the unique element of 
$\mathbb{G}_m$ sending $\sigma(x_h)$ to $\sigma(x_{h'})$.
By construction, $\sigma$ extends to a section 
$\overline{\sigma}$ of $\overline{L}$. The data of 
$\overline{C}$, $\overline{L}$, $\overline{s}$, defines a 
$\Gamma$-marked relative prestable map to  $(\mathscr{A},\mathscr{D})$ for which  the
splitting of the node marked by $e$ recovers $C$, $L$, and $s$.

When the target does expand, the proof is  similar using the description of maps to expansions of  $(\mathscr{A},\mathscr{D})$ as collections of line bundle with sections, as sketched in 
\cite[Proof of Lemma A(ii)]{AMW} using 
\cite[\S 8.2]{ACFW} and reviewed in more detail below. 

We recall the explicit description of the universal deformation 
$$(\widetilde{\AA^1[l]}, \widetilde{D[l]}) \rightarrow \AA^l$$ 
of the $l$-step expanded degeneration $\AA^1[l]$ of
$(\AA^1,0)$. For $l=0$, we have $$(\widetilde{\AA^1[l]}, \widetilde{D[l]})=(\AA^1,0)\,.$$
For $l>0$, $\widetilde{\AA^1[l]}$ is the blow-up of $\widetilde{\AA^1[l-1]} \times \AA^1$ along $\widetilde{D[l-1]} \times \{0\}$, and $\widetilde{D[l]}$ is the  strict transform of 
$\widetilde{D[l-1]} \times \AA^1$. In other words, $\widetilde{\AA^1[l]}$ is the subscheme of $\AA^1 \times (\PP^1)^l \times \AA^l$ given by the equations
\[ w_0 z_1=t_1 w_1\,, w_1 z_2=t_2 z_1 w_2\,,\dots, w_{l-1}z_l=t_l z_{l-1}w_l\,,\]
where $w_0$ is the coordinate on $\AA^1$, $w_i$ and $z_i$ are homogeneous coordinates on the $i$-th copy of $\PP^1$, and $t_1,\dots, t_l$ are coordinates on $\AA^l$.
The divisor $\widetilde{D[l]}$ is defined by the equation $w_l=0$.
We deduce that the universal deformation of the $l$-step expanded degeneration
of $(\mathscr{A},\mathscr{D})$ is a substack of $$[\AA^1/\mathbb{G}_m] \times [\PP^1/\mathbb{G}_m]^l \times [\AA^1/\mathbb{G}_m]^l\,.$$ 

The stack of relative prestable maps $\pi \colon C \rightarrow S$ to the universal deformation of the $l$-step expanded degeneration of $(\mathscr{A},\mathscr{D})$
is an open substack of the stack of prestable curves $\pi \colon C \rightarrow S$ endowed with the data
\begin{itemize}
\item[(i)]$l$ line bundles with sections $(T_1,s_{T_1}), \dots, (T_l, s_{T_l})$ on $S$, which are pull-backs of the tautological line bundles with sections on $[\AA^1/\mathbb{G}_m]^l$,
\item[(ii)] one line bundle with section $(L_0, s_{L_0})$ on $C$, which is the pull-back of the tautological line bundle with section on $[\AA^1/\mathbb{G}_m]$, 
\item[(iii)] $2l$ line bundles with sections $(L_{i,0}, s_{L_{i,0}}), (L_{i,\infty}, s_{L_{i,\infty}})$ for $1 \leq i\leq l$ on $C$, which are pull-backs of the tautological line bundles with sections on the $i$-th copy of 
$[\PP^1/\mathbb{G}_m]$ induced by the equivariant line bundles with sections 
$(\cO(0), z_i)$ and $(\cO(\infty), w_i)$ on $\PP^1$,
\end{itemize}
satisfying the conditions
\begin{align}
\label{eq_relations}
   (L_0, s_{L_0})\otimes(L_{1,0}, s_{L_{1,0}}) & \simeq \pi^{*} (T_1, s_{T_1}) \\
    \nonumber
   (L_{1,\infty}, s_{L_{1,\infty}})\otimes(L_{2,0}, s_{L_{2,0}}) & \simeq \pi^{*} (T_2, s_{T_2}) \\
     \nonumber
   & \vdots \\ 
     \nonumber
(L_{l-1,\infty}, s_{L_{l-1,\infty}})\otimes(L_{l,0}, s_{L_{l,0}}) & \simeq \pi^{*} (T_l, s_{T_l}) \,.  
  \nonumber
\end{align}
By generalizing the gluing argument given in the case without expansion to all these line bundles with sections, we obtain a proof of \eqref{eq_isom}.
We write 
$$\mathfrak{P} := \mathfrak{P}_\Gamma(\mathscr{A},\mathscr{D})= \mathfrak{P}_{\Gamma \setminus e}(\mathscr{A},\mathscr{D})\,.$$

The second step is to study the obstruction theories.
Consider the fiber diagram
\[\begin{tikzcd}
\shP_\Gamma(X,D)
\arrow[r, "s_e"]
\arrow[d, "g"]
&
\shP_{\Gamma \setminus e}(X,D)
\arrow[d, "\ev_e"]\\
\mathfrak{P} \times X
\arrow[r,"\Delta_{\rel}"]&
\mathfrak{P} \times (X,D)^2
\end{tikzcd}\]
of stacks over $\mathfrak{P}$. By \eqref{eq_virtual_relative_nodal} and \eqref{eq_pot_relative}, it is enough to show that 
the perfect obstruction theories 
$(R \pi_{\Gamma,*}f_\Gamma^* T_{(X,D)})^\vee$
and
$(R \pi_{\Gamma\setminus e,*}f_{\Gamma \setminus e}^* T_{(X,D)})^\vee$ 
are compatible over $\Delta_{\rel}$. Here,
$\pi_\Gamma, \pi_{\Gamma\setminus e}$ are the universal domain curves, and $f_\Gamma$, $f_{\Gamma\setminus e}$ are the universal relative stable maps composed with the contraction onto $X$ of the expanded targets.

Following the proof of the splitting formula in absolute Gromov--Witten theory \cite[Axiom III, Eq. (1)]{behrend97gw}, we obtain a distinguished triangle 
\[ (R \pi_{\Gamma\setminus e,*}f_{\Gamma \setminus e}^* T_{(X,D)})^\vee \rightarrow (R \pi_{\Gamma,*}f_\Gamma^* T_{(X,D)})^\vee \rightarrow \sigma_e^* f^* T_{(X,D)}^\vee[1] 
\xrightarrow{[1]} \ \,,\]
where $\sigma_e$ is the section of $\pi_{\Gamma}$
defined by the node marked by $e$.
The result is then obtained from the compatibility of the above triangle with 
the cotangent complex of $\Delta_{\rel}$,
$$L_{\Delta_{\rel}}=T_{(X,D)}^\vee[1]\, , \ \ \ 
g^{*}L_{\Delta_{\rel}}=\sigma_e^* f^* T_{(X,D)}^\vee[1] \,,$$
as argued in  \cite{behrend97gw}.
\end{proof}

\begin{example}
We illustrate  the line bundles with sections appearing in the proof of Theorem \ref{thm_splitting_1}
in a straightforward example.
Let \[\begin{tikzcd}
C
\arrow[r, "f"]
\arrow[d, "\pi"']
&
Z
\arrow[d,"\nu"]\\
\AA^1
\arrow[r,"\simeq"]&\ \AA^1\,,
\end{tikzcd}\]
be a relative prestable map to a 1-parameter family $\nu \colon Z \rightarrow \AA^1$ of expansions of $(X,D)$.
Assume that the general fiber $Z_{t \neq 0}$ is not expanded
and that the special fiber $Z_0$ is the $l$-step expansion of $(X,D)$,
$$Z_0 =X_0 \cup \bP_1 \cup \dots \cup \bP_l\,,$$
where we denote by $X_0$ the copy of $X$ in $Z_0$. 
Then, the (pull-backs from the universal case $(X,D)=(\mathscr{A},\mathscr{D})$ of the) line bundles with sections appearing in the proof of Theorem \ref{thm_splitting_1} are:
\begin{itemize}
\item[(i)] for all $1\leq i\leq l$, $(T_i,s_{T_i})=(\cO_{\AA^1}, s_0)$ where $s_0$ is the section of $\cO_{\AA^1}$ vanishing at $0 \in \AA^1$,
\item[(ii)] $(L_0, s_{L_0})=f^*(\cO_Z(\bP_1+\dots+\bP_l), s_{\bP_1+\dots+\bP_l})$, where 
$s_{\bP_1+\dots+\bP_l}$ is the section vanishing on $\bP_1 \cup \dots \cup \bP_l$, and
\item[(iii)] for all $1 \leq i \leq l$, $(L_{i,0}, s_{L_{i,0}})=f^*(\cO_Z(X_0+\bP_1+\dots+\bP_{i-1}), s_{X_0+\bP_1+\dots+\bP_{i-1}})$, where
$s_{X_0+\bP_1+\dots+\bP_{i-1}}$ is the section vanishing on $X_0 \cup \bP_1 \cup \dots \cup \bP_{i-1}$, and $(L_{i,\infty}, s_{L_{i,\infty}})=f^* (\cO_Z(\bP_{i+1}+\dots+\bP_{l}), s_{\bP_{i+1}+\dots+\bP_{l}})$, where
$s_{\bP_{i+1}+\dots+\bP_{l}}$ is the section vanishing on $\bP_{i+1} \cup \dots \cup \bP_{l}$.
\end{itemize}
For every $1 \leq i \leq l$,
we have 
$$ (\cO_Z(\bP_{i}+\dots+\bP_{l}), s_{\bP_{i}+\dots+\bP_{l}})
\otimes (\cO_Z(X_0+\bP_1+\dots+\bP_{i-1}), s_{X_0+\bP_1+\dots+\bP_{i-1}}) $$
$$=(\cO_Z(Z_0),s_{Z_0})\,,$$
where $s_{Z_0}$ is the section vanishing on $Z_0$. As $(\cO_Z(Z_0), s_{Z_0})
=\nu^{*}(\cO_{\AA^1}, s_0)$, the relations \eqref{eq_relations} hold.
\end{example}

\subsection{Explicit form of the splitting formula}
\label{section_splitting_explicit}
\subsubsection{The strict transform}
To make practical use of the splitting formula given by 
Theorem \ref{thm_splitting_1}, we have to calculate $\Delta_{\rel}^!$. The crucial point is
the following: $\Delta^! \neq \Delta_{\rel}^!$
in the geometry  
\[\begin{tikzcd}
\shP_\Gamma(X,D)
\arrow[r, "s_e"]
\arrow[d]
&
\shP_{\Gamma \setminus e}(X,D)
\arrow[d, "\ev_e"]\\
X
\arrow[r,"\Delta_{\rel}"]
\arrow[rd, "\Delta"']& (X,D)^2 \arrow[d,"p"] \\
& X \times X
\,.
\end{tikzcd}\]
Indeed, $\Delta_{\rel}(X)\subset (X,D)^2$ is the strict transform of the diagonal $\Delta(X) \subset X \times X$
under the blow-up map 
$$p \colon (X,D)^2 \rightarrow X \times X\, , $$
and not the total transform. Whereas we can use the Künneth decomposition of the class of the diagonal $\Delta(X)$
to concretely compute $\Delta^!$, we do not immediately have
such a decomposition for $\Delta_{\rel}(X)$.
To calculate $\Delta_{\rel}^!$, we will study how it differs from  $\Delta^!$. 

Let $R \subset (X,D)^2$ be the locus in $(X,D)^2$ of pairs of points in $(X,D)$ such that $X$ expands along $D$, 
both points are contained in the expansion, and both are in the same fiber of the projection of the expansion on $D$. By the definition of $(X,D)^2$, we have 
\begin{equation} \label{eq_R}
p^{-1}(\Delta(X))=\Delta_{\rel}(X)\cup R \,.
\end{equation}
In other words, $R$ is exactly the excess component responsible for the difference between $\Delta^!$ and $\Delta_{\rel}^!$.

\begin{lemma} \label{lem_excess_component}
The natural projection $R \rightarrow D$ is a trivial $\PP^1$-bundle.
Moreover, the intersection $\Delta_{\rel}(X) \cap R =D$ is a section of this 
$\PP^1$-bundle.
\end{lemma}

\begin{proof}
There is an open subset $U$ of  $R$ which parameterizes ordered pairs of points $(p_1,p_2)$
contained in a bubble $$\bP=\PP(N_{D|X} \oplus \cO_D)\, ,$$
away from the sections $D_0$, 
$D_\infty$ of $\bP$, belonging to the same $\PP^1$-fiber of the projection 
$\bP \rightarrow D$, and considered up to the $\mathbb{G}_m$-action scaling the fibers of $\bP \rightarrow D$. For every such pair $(p_1,p_2)$, there is a unique element $g$ in 
$\mathbb{G}_m$ such that $p_2=g\cdot p_1$. It follows that $U \rightarrow D$ is a trivial $\mathbb{G}_m$-torsor, with section $s_\Delta$ determined by the locus of pairs $(p_1,p_2)$
with $p_1= p_2$. The complement 
$R \setminus U$ is the union of the two sections of 
$R \rightarrow D$ given by having two bubbles, a point in each bubble, and either $p_1$ or $p_2$ in the bubble attached to $X$. After adding these two sections to the trivial $\mathbb{G}_m$-torsor $U \rightarrow D$, we obtain
 a trivial $\PP^1$-bundle.
Finally, the intersection $\Delta_{\rel}(X) \cap R$ is the section $s_\Delta$ of 
$R \rightarrow D$.
\end{proof}

By Lemma \ref{lem_excess_component}, the inclusion $\Delta_R \colon R \rightarrow (X,D)^2$ is a regular embedding of codimension $\dim X$. Hence, 
\begin{equation} \label{eq_excess}
\Delta_{\rel}^! = \Delta^! - \Delta_R^! \,.\end{equation}
Next, we will  compute $\Delta_R^! [\shP_{\Gamma \setminus e}(X,D)]^\virt$.
The answer, given in Theorem \ref{thm_splitting_2} below, is phrased in terms of nodal rubber Gromov--Witten invariants.

\subsubsection{Nodal rubber Gromov-Witten theory}
The definition of a rubber stable map \cite{li2001stable, JunLi, MP}
is similar to 
 Definition \ref{def_relative_map}
 of a relative stable map. Over a geometric point, a \emph{rubber stable map with target $(D,N_{D|X})$} is a map 
 $$f \colon C \rightarrow \bP_l$$ 
 from a prestable curve to a chain $\bP_l$ of $l$-copies of the $\PP^1$-bundle 
\[ \bP = \PP(N_{D|X} \oplus \cO_D) \rightarrow D\,\]
satisfying the following properties:
\begin{itemize}
    \item[(i)] No irreducible component of $C$ is entirely mapped by $f$ into the singular locus of $\bP_l$ or the divisors
    $$D_0, D_\infty \subset \bP_l\, ,$$ where $D_0$ is the zero section of the first copy of $\bP$ and $D_\infty$ is the infinity section of the last copy of 
$\bP$.
    In addition, the relative multiplicities along $D_0$ and $D_\infty$ are fixed.
    \item[(ii)] The map $f$ is \emph{stable} in the sense that there are finitely many pairs $(r_1,r_2)$, where $r_1$ is an automorphism of $C$, $r_2\in \mathbb{G}_m^l$ is an automorphism of $\bP_l$ scaling the fibers of the bubbles, 
    and $f \circ r_1=r_2 \circ f$.
    \item[(iii)] For each point $p \in C_s$ such that $f(p)$ is contained in the singular locus of $\bP_l$, $f$ is 
    \emph{predeformable} at $p$.
\end{itemize}
To such maps which only differ by the action of 
$\mathbb{G}_m^l$ on $\bP_l$ are considered to be isomorphic.

\begin{definition} \label{def_rubber_graph}
A \emph{$(D, N_{D|X})$-valued 
stable rubber graph} is a
$(\bP, D_0 \cup D_\infty)$-valued stable graph as in 
Definition \ref{Def: XDvalued} in all ways except two: 
\begin{enumerate}
\item[(i)] the curve classes $\beta(v)$ attached to the vertices $v$
of $\Gamma$ are elements of $H_2^+(D)$, {\em not} of 
$H_2^+(\bP)$. 
\item[(ii)] vertices $v$ with $g(v)=0$, $n(v)=2$,
and $\beta(v)=0$ are stable if there exists
at least one vertex $v'$ satisfying
$H_2^+(D)$-stability.
\end{enumerate}
The second condition will permit  multiple
covers of the fibers over $\bP\rightarrow D$ ramified over $D_0$ and $D_\infty$ in
the presence of other components.
\hspace*{\fill} $\Diamond$
\end{definition} 
\vspace{8pt}

For every $(D,N_{D|X})$-valued stable rubber graph $G$ without edges, we denote by
\begin{equation} \label{eq_rubber_moduli}
\oM_G^\sim(D,N_{D|X})\end{equation}
the moduli stack
of $G$-marked rubber stable maps to $(D,N_{D|X})$.
Connected components $C_v$ of the domain curve of a $G$-marked rubber stable map are 
indexed by 
the vertices $v$ of $G$. A $G$-marked rubber stable map restricted to $C_v$ is of genus $\g(v)$ with  $\n(v)$ marked point
and of class $\beta(v)$.
The relative multiplicities along $D_0$ and $D_\infty$ are prescribed by $\mu_G$. 

Let $[\oM_G^\sim(D,N_{D|X})]^\virt$ be the virtual 
class given by rubber Gromov--Witten theory \cite{li2001stable, JunLi}.
As in \eqref{eq_epsilon_G_relative}, we have a forgetful morphism 
\begin{equation}\label{eq_epsilon_G_rubber}
\epsilon_G \colon \oM_G^\sim(D,N_{D|X}) \rightarrow \ooMM_G
\end{equation}
remembering the domain curve.

\begin{definition} \label{def_rubber}
For every $(D,N_{D|X})$-valued stable rubber graph $\gamma$, 
we define the moduli stack $\oM_\gamma(D,N_{D|X})$ of
\emph{$\gamma$-marked rubber stable maps} by the fiber diagram
\[\begin{tikzcd}
\oM_\gamma^\sim(D,N_{D|X})
\arrow[r]
\arrow["\epsilon_\gamma"', d]
&
\oM_{\overline{\gamma}}^\sim(D,N_{D|X})
\arrow[d,"\epsilon_{\overline{\gamma}}"]\\
\ooMM_\gamma
\arrow[r,"\iota_\gamma"]& \ooMM_{\overline{\gamma}}\,,
\end{tikzcd}\]
where $\overline{\gamma}$ is the $(D,N_{D|X})$-valued stable rubber graph without edges obtained from $\gamma$ by contraction of all edges, and
$\oM_{\overline{\gamma}}^\sim(D,N_{D|X})$
and $\epsilon_{\overline{\gamma}}$
are defined by \eqref{eq_rubber_moduli}
and \eqref{eq_epsilon_G_rubber}
applied to $G=\overline{\Gamma}$.
In other words, a $\gamma$-marked relative stable curve is a stable map with the data of a
$\gamma$-marking on its domain curve.

We define a virtual class on $\oM_{\overline{\gamma}}^\sim(D,N_{D|X})$ by 
\begin{equation}
    \label{eq_def_rubber}
    [\oM_\gamma^\sim (D,N_{D|X})]^{\virt} := \iota_\gamma^! [\oM_{\overline{\gamma}}^\sim(D,N_{D|X})]^\virt \,.
\end{equation}
By the predeformability condition, we have a version of Lemma \ref{lem_key} for rubber stable maps and hence
a corresponding moduli stack 
$$\shR_\gamma(D,N_{D|X}) \subset \oM_\gamma (D,N_{D|X})$$
of 
$\gamma$-marked rubber maps (with no $\gamma$-marked node mapped in the singular locus of the expanded target) which is a union of connected components of $\oM_\gamma (D,N_{D|X})$.
We define a virtual class $[\shR_\gamma(D,N_{D|X})]^\virt$
by restriction of $ [\oM_\gamma^\sim (D,N_{D|X})]^{\virt}$
to $\shR_\gamma(D,N_{D|X})$.\hspace*{\fill} $\Diamond$
\end{definition}
\vspace{8pt}

\begin{definition}
Nodal rubber Gromov--Witten invariants of $(D,N_{D|X})$ of type 
$\gamma$ are defined by integration over  $[\shR_\gamma(D,N_{D|X})]^\virt$, as in 
\eqref{eq_nodal_relative_gw}, with arbitrarily psi-classes insertions at interior marked points but no psi-classes insertions at the relative marked points along $D_0$ and $D_\infty$.
A \emph{nodal rubber Gromov--Witten invariant of $(D,N_{D|X})$} is a nodal rubber Gromov--Witten invariant of $(D,N_{D|X})$ of type $\gamma$ for some $(D,N_{D|X})$-valued stable rubber graph $\gamma$. \hspace*{\fill} $\Diamond$
\end{definition}  

\subsubsection{Splitting}
Additional notation is necessary to state the result of the 
computation of 
$\Delta_R^! [\shP_{\Gamma \setminus e}(X,D)]^\virt$ 
given in Theorem \ref{thm_splitting_2} below.

\begin{definition} \label{def_boundary_splitting_gamma}
A \emph{boundary splitting} $\sigma$ of $\Gamma \setminus e$ is an ordered pair $(\gamma_1, \gamma_2)$, where 
$\gamma_1$ is an $(X,D)$-stable graph and $\gamma_2$ is a $(D, N_{D|X})$-valued stable rubber graph,  satisfying 
\begin{itemize}
    \item[(i)] $\gamma_1$ and $\gamma_2$ have the same number $\ell(\sigma)$ of relative legs
    and relative legs to $D_0$ respectively, 
    \item[(ii)] for all $1 \leq i \leq \ell(\sigma)$, the relative multiplicity $\mu_{\gamma_1}(i)$ attached to the $i$-th leg of $\gamma_1$ is equal to the relative multiplicity 
    $\mu_{\gamma_2}(i)$ attached to the $i$-th leg of $\gamma_2$ relative to $D_0$,
    \item[(iii)] the labelling of legs of $\gamma_1$ and $\gamma_2$ form a partition of 
    the labelling of legs of $\Gamma \setminus e$, and the legs $h$ and $h'$ of $\Gamma \setminus e$ are legs of $\gamma_2$,
\end{itemize}
together 
with the extra data of an isomorphism between $\Gamma \setminus e$ and the graph obtained by first gluing for all $1 \leq i \leq \ell(\sigma)$ the 
    $i$-th relative leg of $\gamma_1$ with the $i$-th leg of $\gamma_2$ relative to $D_0$, then contracting the 
    $\ell(\sigma)$ newly created edges, and finally viewing the curve classes $\beta(v) \in H_2^+(D)$ attached to vertices of $\gamma_2$ as elements of 
    $H_2^+(X)$ using the natural map
    $H_2^+(D) \rightarrow H_2^+(X)$ induced by the inclusion $D \subset X$. In particular, we have a bijection preserving the relative multiplicities between the set of relative legs of $\Gamma$ and the set of legs of $\gamma_2$ relative to $D_\infty$.
    
Two boundary splittings $(\gamma_1,\gamma_2)$ and $(\gamma_1',\gamma_2')$ are isomorphic if there exist
isomorphisms $\gamma_1 \simeq \gamma_1'$ and $\gamma_2 \simeq \gamma_2'$ compatible 
with the data of the isomorphisms between $\Gamma$ and the glued contracted graphs.

We denote by $\partial \Omega_{\Gamma \setminus e}$ the set of isomorphism classes of splittings of $\Gamma \setminus e$. 
Two boundary splittings $\sigma_1$ and $\sigma_2$ are equivalent if they differ by a permutation of the labelling of relative legs. We denote by $\partial \overline{\Omega}_{\Gamma \setminus e}$ the set of equivalence classes of splittings of $\Gamma \setminus e$. For every $\sigma \in \partial \overline{\Omega}_{\Gamma \setminus e}$, we denote by $|\Aut(\sigma)|$ the order of the group of permutations of the labelling of relative legs fixing one splitting representative of the class $\sigma$, and by $$m(\sigma)=\prod_{i=1}^{\ell(\sigma)} \mu_{\gamma_1}(i)
=\prod_{i=1}^{\ell(\sigma)} \mu_{\gamma_2}(i)$$ the product of the 
$\ell(\sigma)$ relative multiplicities.
\hspace*{\fill} $\Diamond$
\end{definition}
\vspace{8pt}

Fix a boundary splitting $\sigma=(\gamma_1,\gamma_2) \in \partial \overline{\Omega}_{\Gamma \setminus e}$ of $\Gamma \setminus e$ as in Definition \ref{def_boundary_splitting_gamma}.
As in Definitions \ref{def_relative} and 
\ref{def_rubber}, let
$\shP_{\gamma_1}(X,D)$ and $\shR_{\gamma_2}(D,N_{D|X})$ 
be the corresponding moduli stacks of $\gamma_1$-marked
relative stable maps to $(X,D)$ and $\gamma_2$-marked
rubber stable maps to $(D,N_{D|X})$ respectively, with the condition that no $\gamma_i$-marked node is mapped to the singular locus of the expanded target. Evaluation at the relative marked points corresponding to the relative legs of
$\gamma_1$ and the legs of $\gamma_2$ relative to $D_0$
defines a morphism
\[ \shP_{\gamma_1}(X,D) \times \shR_{\gamma_2}(D,N_{D|X}) \rightarrow (D \times D)^{\ell(\sigma)} \,.\]
We denote by $\Delta_\sigma$ the diagonal morphism 
\[ \Delta_\sigma \colon  D^{\ell(\sigma)} \rightarrow (D \times D)^{\ell(\sigma)}\,,\]
and we form the fiber product
\[\begin{tikzcd}
\shP_{\gamma_1}(X,D) \times_{D^{\ell(\sigma)}} \shR_{\gamma_2}(D,N_{D|X})
\arrow[r]
\arrow[d]
&
\shP_{\gamma_1}(X,D) \times \shR_{\gamma_2}(D,N_{D|X}) \arrow[d]\\
D^{\ell(\sigma)}
\arrow[r,"\Delta_\sigma"]& (D \times D)^{\ell(\sigma)}\,.
\end{tikzcd}\]
As in \eqref{eq_gluing_morphism_1}, \eqref{eq_gluing_morphism_2}, and \eqref{eq_gluing_morphism_3}, we have a natural gluing morphism
\[ \Psi_\sigma \colon 
\shP_{\gamma_1}(X,D) \times_{D^{\ell(\sigma)}} \shR_{\gamma_2}(D,N_{D|X}) 
\rightarrow \shP_{\Gamma\setminus e}(X,D) \,.\]

Evaluation at the two marked points 
associated to the legs $h$ and $h'$ of $\gamma_2$ induces a morphism
\[ \ev_e \colon \shR_{\gamma_2}(D,N_{D|X}) 
\rightarrow D \times D \,.\]
For rubber stable maps, we do not have an evaluation morphism valued in $\bP$ because of the quotient by the scaling 
$\mathbb{G}_m$-action, but the composition with the contraction of 
$\bP$ on $D$ is well-defined.
Finally, we form the fiber diagram 
\[\begin{tikzcd}
\shR_{\gamma_2}^{\Delta_D}(D,N_{D|X})
\arrow[r]
\arrow[d]
&
\shR_{\gamma_2}(D,N_{D|X}) \arrow[d]\\
D
\arrow[r,"\Delta_D"]& D \times D\,,
\end{tikzcd}\]
where $\Delta_D \colon D \rightarrow D \times D$ is the diagonal morphism.

\begin{theorem} \label{thm_splitting_2}
For every $(X,D)$-valued stable graph $\Gamma$ and every edge $e$ of $\Gamma$, 
\begin{multline*} 
\Delta_R^! [\shP_{\Gamma \setminus e}(X,D)]^\virt
=\\
\sum_{\sigma=(\gamma_1,\gamma_2)\in \partial \overline{\Omega}_{\Gamma \setminus e}}
\frac{m(\sigma)}{|\Aut(\sigma)|}
\Psi_{\sigma, *} \Delta_\sigma^!([\shP_{\gamma_1}(X,D)]^\virt
\times \Delta_D^! [\shR_{\gamma_2}(D,N_{D|X})]^\virt)\,.
\end{multline*}
\end{theorem}

\begin{proof}
Let $E$ be the exceptional divisor of the blow-up 
$(X,D)^2 \rightarrow X \times X$, the locus in 
$(X,D)^2$ where $X$ expands along $D$ and both points are in the expansion. The regular embedding 
$$\Delta_R \colon R \rightarrow (X,D)^2$$ of codimension 
$\dim X$ is the composition of the regular embeddings 
$$\iota_{R,E} \colon R \rightarrow  E\ \  \ \text{and} \ \ \  \iota_E \colon E \rightarrow (X,D)^2$$
of codimensions $\dim X -1$ and $1$ respectively.
Remembering the projections to $D$ of the two points contained in the expansion defines a morphism $E \rightarrow D \times D$, and we have a fiber diagram 
\begin{equation}\label{eq_delta_D}
\begin{tikzcd}
R
\arrow[r,"\iota_{R,E}"]
\arrow[d]
&
E
\arrow[d]\\
D
\arrow[r,"\Delta_D"]& D \times D \,.
\end{tikzcd}\end{equation}


To prove Theorem \ref{thm_splitting_2}, we first show 
\begin{multline}
\label{eq_splitting_proof}
\iota_E^! [\shP_{\Gamma \setminus e}(X,D)]^\virt
=\\
\sum_{\sigma=(\gamma_1,\gamma_2)\in \partial \overline{\Omega}_{\Gamma \setminus e}}
\frac{m(\sigma)}{|\Aut(\sigma)|}
\Psi_{\sigma, *} \Delta_\sigma^!([\shP_{\gamma_1}(X,D)]^\virt
\times [\shR_{\gamma_2}(D,N_{D|X})]^\virt)\,.
\end{multline}
The proof of \eqref{eq_splitting_proof} is parallel to the proof of the nodal degeneration formula
given in detail in \S \ref{section_degeneration_technical} - \ref{section_degeneration_proof}.
Let $\overline{\Gamma \setminus e}$ be the graph without edges
obtained by contracting all edges of $\Gamma \setminus e$,
as in Definition \ref{def_contraction}.
The analogue of \eqref{eq_splitting_proof} 
with $\Gamma \setminus e$ replaced by $\overline{\Gamma \setminus e}$
holds by the usual gluing formula in relative Gromov--Witten theory describing the virtual divisor of the moduli stack of relative stable maps where the target expands in terms of moduli stacks of rubber stable maps \cite{JunLi}. 
To obtain \eqref{eq_splitting_proof}, we apply 
$\iota_{\Gamma \setminus e}^!$ to the latter formula and restrict to the locus where the $\Gamma \setminus e$-marked nodes are not mapped to the singular locus of the expanded targets.

Theorem \ref{thm_splitting_2} follows by applying 
$\iota_{R,E}^!$ to both sides of \eqref{eq_splitting_proof}.
On the left side, we use  
$$\iota_{R,E}^! \iota_E^! =\Delta_R^!$$ by functoriality of the Gysin pull-back \cite[Theorem 6.5]{Fulton}. On the right side, we use  $\iota_{R,E}^!=\Delta_D^!$ by 
\eqref{eq_delta_D} and the fact that $\Delta_D^!$
commutes with $\Psi_{\sigma,*}$ and $\Delta_\sigma^!$
by general properties of the Gysin pullback
(\cite[Theorem 6.2 a)]{Fulton}, \cite[Theorem 6.4]{Fulton}),
similar to Lemmas \ref{lem_alpha} and \ref{lem_beta}.
\end{proof}

We state below the splitting formula in the final and most explicit form.

\begin{theorem} \label{thm_splitting_3}
For every $(X,D)$-valued stable graph $\Gamma$ and every edge $e$ of $\Gamma$, 
\begin{multline*} [\shP_\Gamma(X,D)]^\virt = \Delta^! [\shP_{\Gamma \setminus e}(X,D)]^\virt\\   -  
\sum_{\sigma=(\gamma_1,\gamma_2)\in \partial \overline{\Omega}_{\Gamma \setminus e}}
\frac{m(\sigma)}{|\Aut(\sigma)|}
\Psi_{\sigma, *} \Delta_\sigma^!([\shP_{\gamma_1}(X,D)]^\virt
\times \Delta_D^! [\shR_{\gamma_2}(D,N_{D|X})]^\virt)\,.
\end{multline*}
\end{theorem}

\begin{proof}
By Theorem \ref{thm_splitting_1}, we have 
$[\shP_\Gamma(X,D)]^\virt=\Delta_{\rel}^![\shP_{\Gamma \setminus e}(X,D)]^\virt$. By \eqref{eq_excess}, we have 
$$\Delta_{\rel}^! = \Delta^! -\Delta_R^!\, .$$
Finally, 
$\Delta_R^! [\shP_{\Gamma \setminus e}(X,D)]^\virt$
is given by Theorem \ref{thm_splitting_2}.
\end{proof}

\begin{example}
We illustrate Theorem \ref{thm_splitting_3} with a basic example.
Let $X=\PP^1$ and $D=\{0\} \subset \PP^1$. Let $\Gamma$ be the graph with two vertices $v_1$ and $v_2$
(of genus   $\g(v_1)=\g(v_2)=0$, class $\beta(v_1)=\beta(v_2)=1 \in \Z_{\geq 0} =H_2^+(\PP^1)$,
and each with a relative leg of relative multiplicity $1$)
connected by an edge $e$ as illustrated in Figure \ref{Fig: nodetozero}.

\begin{figure}
\resizebox{0.9\linewidth}{!}{\input{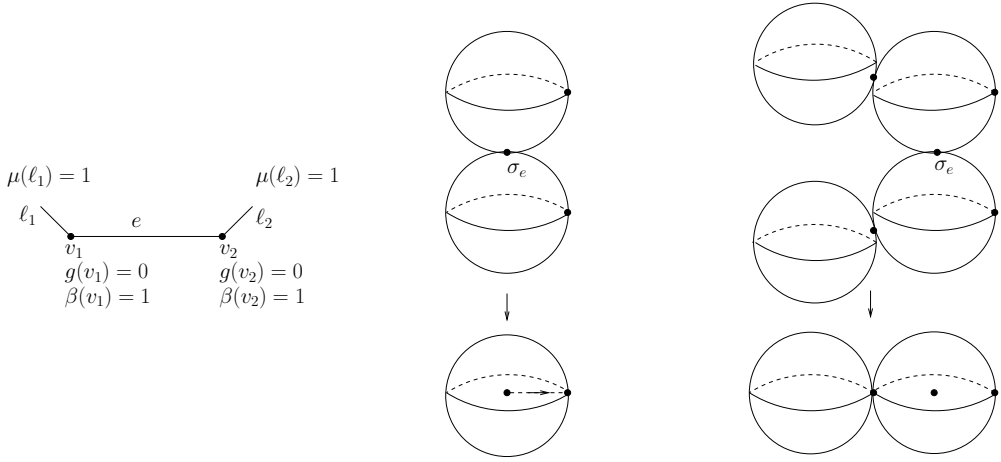}}
\caption{The $(\PP^1,0)$-valued stable graph $\Gamma$ and a $\Gamma$-marked relative stable map bubbling as the image of the node $\sigma_e$ tends to $D=\{0\}$.}
\label{Fig: nodetozero}
\end{figure}

\noindent A general element of $\shP_{\Gamma}(\PP^1,D)$ has a domain curve consisting of two $\PP^1$ components
connected by a node 
$\sigma_e$, and each of these $\PP^1$-components is mapped isomorphically to the target $X=\PP^1$. Such relative stable map is uniquely determined by the position of the image of $\sigma_e$ in $\PP^1$. When the image of $\sigma_e$ tends to $D=\{0\}$, the target expands, and the limit stable map is unique. Hence, we have an isomorphism 
$$\shP_{\Gamma}(\PP^1,0) \simeq \PP^1$$
given by the image of $\sigma_e$.

On the other hand, $\shP_{\Gamma \setminus e}(\PP^1,0)$ is a moduli space of disconnected relative stable maps. 
A general element of $\shP_{\Gamma \setminus e}(\PP^1,0)$ has a domain curve consisting of two disjoint $\PP^1$-components, each  mapping isomorphically to $\PP^1$ and each with an extra marked point. These two extra marked points $p_1$ and $p_2$ correspond to the interior legs of 
$\Gamma \setminus e$ created by splitting the edge $e$ of $\Gamma$. When the image of $p_1$ or $p_2$ tends to $0$, the target expands. In other words, taking the image of 
$(p_1,p_2)$ defines an isomorphism $\shP_{\Gamma \setminus e}(\PP^1,0) \simeq (\PP^1,0)^2$ with the moduli space $(\PP^1,0)^2$ of pairs of ordered points in 
$(\PP^1,0)$. Via the isomorphisms $$\shP_\Gamma(\PP^1,0)=\PP^1 \ \  \ \text{and} \ \ \ \shP_{\Gamma \setminus e}(\PP^1,0)=(\PP^1,0)^2\,,$$ 
we can view $\shP_\Gamma(\PP^1,0)$ as the strict transform 
$\Delta_{\rel}$ of the diagonal $$\Delta=\PP^1 \subset \PP^1 \times \PP^1\,.$$
The splitting formula of Theorem \ref{thm_splitting_3} reduces to the expression \eqref{eq_R} of
$\Delta_{\rel}$ as the difference between the total transform of $\Delta$ and the 
correction term $R$, which here is the exceptional divisor of 
the blow-up $(\PP^1,0)^2 \rightarrow \PP^1 \times \PP^1$ at the point $(0,0)$,
as in Figure \ref{Fig: mod}. Concretely, 
$R$ is a moduli space of rubber maps from the disjoint union of two $1$-marked degree $1$ copies of $\PP^1$ to a bubble. The relative position of these two marked points up to the 
scaling action of $\mathbb{G}_m$ on the bubble induces an isomorphism $R \simeq \PP^1$.
\end{example}

\subsection{Nodal rubber calculus}
\label{section_nodal_rubber}

A $(\bP, D_0 \cup D_\infty)$-valued stable graph 
$\gamma'$ is of {\em multiple fiber type} if $\gamma'$ consists of disconnected vertices $v$
satisfying $$g(v)=0\, , \ \  n(v)=2, \ \ \rho_*\beta(v)=0\, ,$$
where  $\rho \colon \bP \rightarrow D$ is the
projection.

\begin{definition}
For every $(\bP, D_0 \cup D_\infty)$-valued stable graph 
$\gamma'$ {\em not of multiple fiber type}, we define a $(D, N_{D|X})$-valued stable rubber graph  $\rho_* \gamma'$ by the following procedure.
For every vertex $v$ 
of $\gamma'$, we replace  the class $\beta(v) \in H_2^+(\bP)$ by the class $\rho_{*}\beta(v) \in H_2^+(D)$.

Furthermore, for every $(D, N_{D|X})$-valued stable rubber graph $\gamma$, let $\Xi_\gamma$  be the set of  
$(\bP, D_0 \cup D_\infty)$-valued stable graphs 
$\gamma'$ such that $\rho_* \gamma'=\gamma$.
\hspace*{\fill} $\Diamond$
\end{definition}
\vspace{8pt}

Let $\gamma$ be a $(D, N_{D|X})$-valued stable rubber graph.
For every $\gamma' \in \Xi_\gamma$, we have a natural forgetful morphism 
\[ \tau_{\gamma'} \colon \shP_{\gamma'}(\bP, D_0 \cup D_\infty) \rightarrow \shR_{\gamma}(D,N_{D|X}) \,,\]
obtained by viewing a relative map as a rubber map and stabilizing if necessary. The morphism $\tau_{\gamma'}$
is well-defined because components of the domain curves contracted by the stabilization are $\PP^1$-covers of $\PP^1$-fibers of 
$\rho$ fully ramified over $D_0$ and $D_\infty$, and in particular do not contain and are not adjacent to a $\gamma'$-marked node.

\begin{lemma} \label{lem_rigidification}
Let $\gamma$ be a $(D, N_{D|X})$-valued stable rubber graph and $\ell$ an interior leg of $\gamma$.
For every $\gamma' \in \Xi_\gamma$, let 
$$\ev_{\ell} \colon \shP_{\gamma'}(\bP, D_0 \cup D_\infty) \rightarrow \bP$$ be the evaluation morphism at the marked point marked by $\ell$.
Then,
\begin{align*} [\shR_\gamma(D,N_{D|X})]^\virt 
&= \sum_{\gamma' \in \Xi_\gamma} \tau_{\gamma',*}(\ev_{\ell}^*([D_0]) \cap [\shP_{\gamma'}(\bP, D_0 \cup D_\infty)]^\virt)\\
&= \sum_{\gamma' \in \Xi_\gamma} \tau_{\gamma',*}(\ev_{\ell}^*([D_\infty]) \cap [\shP_{\gamma'}(\bP, D_0 \cup D_\infty)]^\virt)
\,.\end{align*}
\end{lemma}

\begin{proof}
We prove the result for $\ev_\ell^*([D_0])$. The result for 
$\ev_\ell^*([D_\infty])$ will follow by symmetry.
To simplify the notation, we write 
$\shR_\gamma$ for $\shR_\gamma(D,N_{D|X})$ and
$\shP_{\gamma'}$ for $\shP_{\gamma'}(\bP,D_0 \cup D_\infty)$.

Let $\overline{\gamma}$ be the $(D, N_{D|X})$-valued stable rubber graph without edges obtained from $\gamma$ by contracting all edges, as in Definition 
\ref{def_contraction}. There exists a unique $\overline{\gamma}' \in \Xi_{\overline{\gamma}}$ such that the moduli stack $\shP_{\overline{\gamma}'}$
is not empty. Indeed, for every vertex $v$ of $\overline{\gamma}$ with class 
$\beta(v)\in H_2^+(D)$, there exists a unique class 
$\beta'(v) \in H_2^+(\mathbb{P})$ satisfying $$\rho_* \beta'(v)=\beta(v)\, ,$$ and 
such that the intersections numbers 
$\beta(v) \cdot D_0$ and $\beta(v) \cdot D_\infty$ are equal to the sums of relative multiplicities of legs of $\overline{\gamma}$ relative to $D_0$ and $D_\infty$ adjacent to $v$.

According to the rigidification result \cite[Lemma 2]{MP}, we have
\begin{align} \label{eq_lem_mp}
[\shR_{\overline{\gamma}}]^\virt 
=  \tau_{\overline{\gamma}',*}(\ev_{\ell}^*([D_0]) \cap [\shP_{\overline{\gamma}'}]^\virt)
\,.\end{align}
We apply $\iota_\gamma^!$ to both sides of \eqref{eq_lem_mp},
where 
$\iota_\gamma \colon \ooMM_\gamma  \rightarrow \ooMM_{\overline{\gamma}}$ defines the $\gamma$-marking.
By definition, the restriction of 
$\iota_{\gamma}^!  [\shR_{\overline{\gamma}}]^\virt$ to $\shR_{\gamma}$ is $ [\shR_{\gamma}]^\virt $.
It remains to compute the restriction to $\shR_{\gamma}$ of 
\[\iota_\gamma^! \tau_{\overline{\gamma}',*}(\ev_{\ell}^*([D_0]) \cap [\shP_{\overline{\gamma}'}]^\virt)\,.\]

The diagram
\[
\begin{tikzcd}[column sep=1em]
\bigsqcup_{\gamma'\in \Xi_\gamma} \shP_{\gamma'}  \arrow[r] \arrow[d, "\tau_{\gamma'}"'] &
\shP_{\overline{\gamma}'}
\arrow[d,"\tau_{\overline{\gamma}'}"] &
 \\
\shR_\gamma
\arrow[r] \arrow[d,"\epsilon_\gamma"'] &
\shR_{\overline{\gamma}}
\arrow[d,"\epsilon_{\overline{\gamma}}"] \\
\ooMM_\gamma  \arrow[r,"\iota_{\gamma}"]& \ooMM_{\overline{\gamma}}
\end{tikzcd}\]
is the restriction of a fiber diagram to the locus where the 
$\gamma$-marked nodes are not mapped to the singular locus of the expanded target. Indeed, a 
$\gamma$-marking of a rubber stable map
obtained by the forgetful morphism from a relative stable map naturally defines a $\gamma'$-marking of this relative stable map: for every vertex $v$, define 
$\beta(v) \in H_2^+(\bP)$ as the class of the component of the relative stable map marked by the vertex $v$.

By compatibility of the Gysin pull-back with proper push-forward
\cite[Theorem 6.2 a)]{Fulton}, we deduce that, after restriction to $\shR_\gamma$,
\[ 
\iota_\gamma^! 
\tau_{\overline{\gamma}',*} (\ev_{\ell}^*([D_0]) \cap [\shP_{\overline{\gamma}'}]^\virt)
= \sum_{\gamma' \in \Xi_\gamma}
\tau_{\gamma',*}(\iota_\gamma^!(\ev_{\ell}^*([D_0]) \cap [\shP_{\overline{\gamma}'}]^\virt))|_{\shP_{\gamma'}}\,.\]

On the other hand, for every 
$\gamma' \in \Xi_\gamma$, we have in restriction to 
$\shP_{\gamma'}$
\[ \tau_{\gamma',*}(\ev_\ell^*([D_0]) \cap [\shP_{\gamma'}]^\virt)
=\tau_{\gamma',*}\iota_{\gamma'}^!(\ev_\ell^{*}([D_0]) 
\cap [\shP_{\overline{\gamma}'}]^\virt)\,,
\]
where $\iota_{\gamma'}^!$ is defined using the diagram
\[
\begin{tikzcd}[column sep=1em]
\shP_{\gamma'}  \arrow[r] \arrow[d,"\epsilon_{\gamma'}"'] &
\shP_{\overline{\gamma}'}
\arrow[d,"\epsilon_{\overline{\gamma}'}"] &
 \\
\ooMM_\gamma
\arrow[r,"\iota_{\gamma'}"] &
\ooMM_{\overline{\gamma}'}\,.
\end{tikzcd}\]

Therefore, it is enough to show for every 
$\gamma'\in \Xi_\gamma$ that, after restriction to 
$\shP_{\gamma'}$, 
\[ (\iota_\gamma^!(\ev_{\ell}^*([D_0]) \cap [\shP_{\overline{\gamma}'}]^\virt))|_{\shP_{\gamma'}} 
= \iota_{\gamma'}^! (\ev_{\ell}^*([D_0]) \cap [\shP_{\overline{\gamma}'}]^\virt)\,.\]
The claim follows from the existence, implied by 
\eqref{eq_normal}
and the fact that the components of the domains curves contracted by the stabilization do not contain and are not adjacent to a $\gamma'$-marked node,
of a canonical isomorphism between the pullbacks of normal bundles
\[ (\tau_{\gamma'}\circ \epsilon_\gamma)^* N_{\iota_\gamma} 
\simeq \epsilon_{\gamma'}^{*}N_{\iota_{\gamma'}} \]
on $\shP_{\gamma'}$.
\end{proof}

\subsection{Nodal relative in terms of absolute}
\label{section_nodal_relative_absolute}

We are now in position to prove the main reconstruction result for nodal relative Gromov-Witten theory.
\begin{theorem}[\textbf{Theorem \ref{thm_intro_splitting}}]
\label{thm_splitting}
Let $X$ be a smooth projective variety over $\C$ and $D \subset X$ a smooth divisor.
Then, the nodal relative Gromov--Witten invariants of $(X,D)$ can be effectively reconstructed from the Gromov--Witten invariants of $X$, the Gromov--Witten invariants of $D$, and the restriction map $H^\star(X) \rightarrow H^\star(D)$. 
\end{theorem}
\begin{proof}
We prove by induction on $k$ the following claim:

\vspace{5pt}
\noindent {\em For every $k \in \Z_{\geq 0}$, for every pair 
$(X,D)$, and for every $(X,D)$-valued stable graph $\Gamma$
with $|E_\Gamma|=k$ edges, the nodal relative Gromov--Witten invariants of $(X,D)$
of type $\Gamma$ can be reconstructed from the Gromov--Witten invariants of $X$, the Gromov--Witten invariants of $D$, and the restriction map $H^\star(X) \rightarrow H^\star(D)$.}
\vspace{5pt}

\noindent In the base case $k=0$, the claim holds by 
\cite[Theorem 2]{MP}. 

If $k>0$, let $e$ be an edge of $\Gamma$, obtained by gluing half-edges $h$ and $h'$. By the splitting formula of Theorem 
\ref{thm_splitting_3} and the Künneth decompositions of the classes of the diagonals
$$\Delta \colon X \rightarrow X\times X\, , \ \ \ \Delta_\sigma \colon D^{\ell(\sigma)} \rightarrow (D \times D)^{\ell(\sigma)}\,, \ \ \
\Delta_D \colon D \rightarrow D \times D\, ,$$
the nodal relative Gromov--Witten invariants of $(X,D)$ of type $\Gamma$ can be computed in terms of nodal relative Gromov--Witten invariants of $(X,D)$
of type $\Gamma \setminus e$ and $\gamma_1$,
and the nodal rubber Gromov--Witten invariants of 
$(D,N_{D|X})$ of type $\gamma_2$, 
where $\gamma_1$ and $\gamma_2$ run over all
boundary splittings $(\gamma_1,\gamma_2)\in \partial 
\overline{\Omega}_{\Gamma \setminus e}$.

A $(D, N_{D|X})$-valued stable rubber graph of the form $\gamma_2$ contains the interior legs  $h$ and $h'$ coming from splitting the edge $e$ of $\Gamma$. 
Hence, we can apply Lemma \ref{lem_rigidification} to 
$\gamma_2$ to deduce that the nodal rubber Gromov--Witten invariants of $(D, N_{D|X})$ of type $\gamma_2$
can be computed in terms of the nodal relative Gromov--Witten invariants of $(\bP,D_0 \cup D_\infty)$ of types $\gamma_2' \in \Xi_{\gamma_2}$. Here, we are using the fact that the 
forgetful morphism $\tau_{\gamma_2'}$ in Lemma \ref{lem_rigidification} is compatible with cotangent lines (no marked point lies on a component contracted by stabilization from relative to rubber).

Hence, the nodal Gromov--Witten invariants of $(X,D)$ of type 
$\Gamma$ can be computed in terms of the nodal relative Gromov--Witten invariants of $(X,D)$ of type $\Gamma \setminus e$ and
$\gamma_1$, and the nodal relative Gromov--Witten invariants of
$(\bP,D_0 \cup D_\infty)$ of type $\gamma_2'$. As $|E_{\Gamma \setminus e}|$, $|E_{\gamma_1}|<k$, by the induction hypothesis, the nodal relative Gromov--Witten invariants of $(X,D)$ of type $\Gamma \setminus e$ and 
$\gamma_1$ can be effectively reconstructed from the Gromov--Witten invariants of $X$, the Gromov--Witten invariants of $D$, and the restriction map 
$H^\star(X) \rightarrow H^\star(D)$. 
As
$|E_{\gamma_2'}|<k$, by the induction hypothesis,
the nodal relative Gromov--Witten invariants of $(\bP,D_0 \cup D_\infty)$ of types $\gamma_2'$ can be reconstructed from the Gromov--Witten invariants of $\bP$, the Gromov--Witten invariants of $D$ and the class $c_1(N_{D|X}) \in H^2(D)$.
Finally, the Gromov--Witten invariants of $\bP$ can be reconstructed from the Gromov--Witten invariants of $D$ and 
$c_1(N_{D|X}) \in H^2(D)$ by \cite[Theorem 1]{MP}.
\end{proof}

\section{Cohomology of complete intersections}
\label{sec:gw_ci1}

We review Deligne's results on the monodromy of complete intersections in \S \ref{sec_monodromy} and present  basic 
aspects of the invariant theory of the
orthogonal and symplectic groups in 
 \S \ref{section_invariant_theory} - \ref{sec_matrix_diagonal}. These results are applied in the next section \S \ref{sec:gw_ci2} to the study of the Gromov--Witten theory of complete intersections.

\subsection{Monodromy of complete intersections}
\label{sec_monodromy}
Let $m$ be a positive integer and $(d_1,\dots,d_r)$ a tuple of
$r \geq 1$
positive integers. Let 
\[ U \subset \prod_{i=1}^r \PP(H^0(\PP^{m+r},\cO(d_i)))\] 
be the open subset 
parameterizing smooth $m$-dimensional complete intersections of degrees 
$(d_1,\dots,d_r)$ in $\PP^{m+r}$. 
 We fix a point $u \in U$,
and we denote by $X$ the corresponding smooth complete intersection in 
$\PP^{m+r}$.

By the Lefschetz hyperplane theorem, the restriction map in cohomology 
\[ H^i(\PP^{m+r}) \rightarrow H^i(X)\] 
is an isomorphism for $i \neq m$ and $0 \leq i \leq 2m$. We have a decomposition 
\[ H^m(X)=H^m(\PP^{m+r}) \oplus H^m(X)_{\prim} \]
where $H^m(X)_{\prim}$ is the primitive cohomology, the subspace of $H^m(X)$
annihilated by the hyperplane class.
A class $\gamma \in H^{*}(X)$ is \emph{simple} if $\gamma$ lies in the image of 
$H^{*}(\PP^{m+r})$ and \emph{primitive} if $\gamma \in H^m(X)_{\prim}$.

The fundamental group $\pi_1(U,u)$ of $U$ based at $u$ acts on the primitive cohomology 
$H^\star(X)_{\prim}$. The (algebraic) \emph{monodromy group} $G$ is,
by definition, the Zariski closure of the image of $\pi_1(U,u)$ in 
the algebraic group $\mathrm{Aut} (H^\star(X)_{\prim}\otimes \C)$ of linear automorphisms of $H^\star(X)_{\prim} \otimes \C$.

We review below the known explicit description of the monodromy group $G$ depending on the parity of the dimension $m$ of $X$.
If the dimension $m$ of $X$ is odd, then the intersection pairing 
defines a skew-symmetric non-degenerate bilinear form $\omega$ on 
$H^m(X)$. 
Moreover, as the odd cohomology of $\PP^{m+r}$ is zero, the entire cohomology 
$H^m(X)$ is primitive. 
Let $\mathrm{Sp}(V)$ be the complex symplectic group of 
automorphisms of $$V := H^m(X)_{\prim}\otimes \C=H^m(X) \otimes \C$$ endowed with $\omega$.
As the monodromy preserves $\omega$, we necessarily have $G \subset \mathrm{Sp}(V)$.

\begin{proposition} \label{prop_big_monodromy_odd}
If $\mathrm{dim}(X)=m$ is odd, then the monodromy group $G$ acting on the primitive cohomology $V$ is as large as possible:  $G=\mathrm{Sp}(V)$.
\end{proposition}

\begin{proof}
The result follows from \cite[Theorem 4.4.1]{weil2} applied to a Lefschetz pencil 
of sections of a complete intersection of dimension $m+1$ and multidegree 
$(d_1,\dots, d_{r-1})$ by degree $d_r$ hypersurfaces (such Lefschetz pencils exists by \cite[XVII, Theorem 2.5.2]{deligne1973groupes}).
\end{proof}

On the other hand, if the dimension $m$ of $X$ is even, then the intersection pairing 
defines a symmetric non-degenerate bilinear form $\alpha$ on 
$H^m(X)$. The primitive cohomology 
$H^m(X)_{\prim}$ is the orthogonal with respect to $\alpha$ of 
the image of the restriction map 
$H^m(\PP^{m+r}) \rightarrow H^m(X)$. 
We also denote by $\alpha$ the symmetric non-degenerate bilinear form 
obtained by restricting $\alpha$ to $H^m(X)_{\prim}$.
Let $O(V)$ be the orthogonal group of 
automorphisms of 
$$V := H^m(X)_{\prim} \otimes \C$$ endowed with $\alpha$. 
As the monodromy group preserves $\alpha$, we necessarily have $G \subset O(V)$.

\begin{proposition} \label{prop_big_monodromy_even}
If $\mathrm{dim}(X)=m$ is even, then the monodromy group $G$
acting on the primitive cohomology $V$ is as large as possible, $G=O(V)$,
except if $X$ is a cubic surface or a complete intersection of two quadrics.
Furthermore, 
\begin{itemize}
\item[(i)] if $X$ is a cubic surface, then $G=W(E_6)$,
\item[(ii)] if $X$ is a complete intersection of two quadrics, then $G=W(D_{m+3})$,
    \end{itemize}
where $W(R)$ denotes the Weyl group of the root system $R$.
\end{proposition}

\begin{proof}
We consider a Lefschetz pencil of sections of a complete intersection of dimension $m+1$
and multidegree $(d_1,\dots,d_{r-1})$ by degree $d_r$ hypersurfaces
(such Lefschetz pencil exists by \cite[XVII, Theorem 2.5.2]{deligne1973groupes}). We denote by 
$G' \subset O(V)$ the Zariski closure of the image of the corresponding 
monodromy representation. By \cite[Theorem 4.41]{weil2}, we have either 
$G'=O(V)$ or $G'$ is finite. By  
\cite[XIX, Proposition 3.4]{deligne1973groupes}, $G'$ is finite if and only if 
the primitive cohomology $V$ is entirely of Hodge type $(m,m)$, and so, using 
\cite[XI, \S 2.9]{deligne1973groupes}, if and only if $X$ is a quadric, a complete intersection of two quadrics, or a cubic surface.
If $X$ is not of these cases, we therefore have $G=O(V)$.

If $X$ is a quadric, then by 
 \cite[XIX, \S 5.2]{deligne1973groupes}, the primitive cohomology $V$ is of dimension $1$, and $G=W(A_1)=\{ \pm \mathrm{Id}\}$, which is also equal to $O(V)$.
 
If $X$ is a a cubic surface or a complete intersection of two quadrics, 
the result 
is contained in \cite[XIX, \S 5.2-5.3]{deligne1973groupes}.
\end{proof}

\subsection{Pairings and the representation
theory of the symmetric group}
\label{section_invariant_theory}
As will be described in more detail in \S \ref{sec_trading}, the 
Gromov--Witten invariants of a complete intersection $X$ can naturally be viewed as multilinear forms on the primitive cohomology 
$V$ of $X$.  
Deformation invariance implies that these multilinear forms are invariant under the action of the monodromy group 
$G$ on $V$. By Propositions \ref{prop_big_monodromy_odd}\,-\,\ref{prop_big_monodromy_even}, $G$ is in most cases a symplectic or orthogonal group. 
After a discussion of the representation of the
symmetric group, 
we will review  the invariant theory of orthogonal and sympletic groups 
in \S \ref{orthh} and \S \ref{sympll}
as preparation for \S \ref{sec_trading}. For a more comprehensive review, see  \cite{fulton1991representation,vershik2005new,zhao2008young}.


\begin{definition}
A \emph{partition} of a positive integer $n$ is a sequence of positive integers
\begin{equation}
    \label{Eq: partition}
\lambda=(\lambda_1, \dots, \lambda_{\ell})    
\end{equation}
 satisfying $\lambda_1 \geq \cdots \geq \lambda_{\ell}$ and $n=\lambda_1+\cdots +\lambda_{\ell}$. We write $\lambda \vdash n$ to indicate that $\lambda$ is a partition of $n$.
 Let $\ell(\lambda) = \ell$ be the {\em length}
 of $\lambda$.
\hspace*{\fill} $\Diamond$
\end{definition}

\begin{definition}
\label{Def: Young diagram}
A \emph{Young diagram} is a finite collection of boxes arranged in left-justified rows,
with the row sizes weakly decreasing\footnote{Apart from the \emph{English convention} which we use here, in recent references one can also find the \emph{French convention}, which is the upside-down form of the English convention.}.
The Young diagram associated to the partition 
$\lambda = (\lambda_1,\lambda_2,\ldots ,\lambda_\ell)$ has $\ell$ rows  and $\lambda_i$ boxes on the $i$-th row, see Figure \ref{Fig: young4}.
\hspace*{\fill} $\Diamond$
\end{definition}
\vspace{8pt}

\begin{figure}
\resizebox{.7\linewidth}{!}{\input{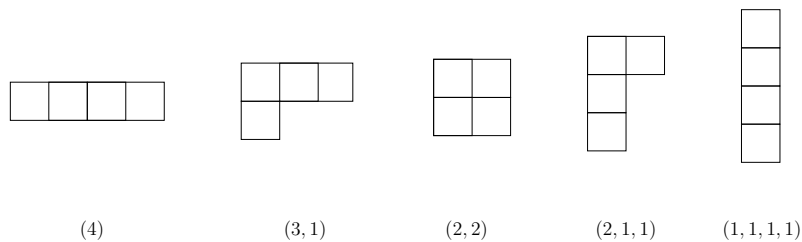}}
\caption{Partitions of $4$ and the corresponding Young diagrams}
\label{Fig: young4}
\end{figure}


The conjugacy classes of irreducible representations of the symmetric group $S_n$ are in a natural (up to a sign) one-to-one correspondence with partitions of~$n$ (or with the Young diagrams with~$n$ boxes). 
We denote by $\rho_\lambda$ the irreducible representation of $S_n$ corresponding to the 
partition $\lambda$. We use the convention that the trivial representation corresponds to the partition $(n)$, or to the Young diagram with a single row, whereas the 1-dimensional sign representation
$\epsilon$, which sends every element of $S_n$ to its sign,
corresponds to the partition $(1,\ldots ,1)$, or to the Young diagram with a single column. 

\begin{definition}
\label{Def: pairing}
An \emph{$n$-pairing} is a 
fixed-point free involution on the set $\{1,...,2n\}$. 
We specify an $n$-pairing by listing the transpositions formed by an element and its image under the involution, so $$(1,2)(3,4)\ldots (2n-1,2n)$$ is an $n$-pairing. We can also represent an $n$-pairing by the associated arc diagram, see Figure~\ref{Fig:ArcDiagram}.
We will often call an $n$-pairing  just a \emph{pairing}. 
\hspace*{\fill} $\Diamond$
\end{definition}
\vspace{8pt}

\begin{figure}
\resizebox{.5\linewidth}{!}{\input{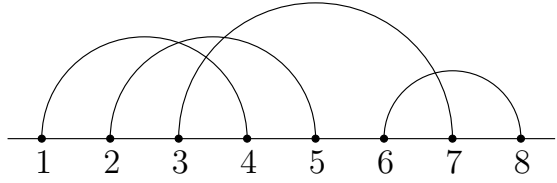}}
    \caption{The arc diagram of the pairing $(14)(25)(37)(68)$}
    \label{Fig:ArcDiagram}
\end{figure}

We denote by $\mathcal{P}_n$ the set of $n$-pairings. Let $\CC\mathcal{P}_n$ be the vector space 
\begin{equation}
    \label{Eq: vector space pn}
   \CC\mathcal{P}_n := \bigoplus_{P \in \mathcal{P}_n} \CC P\,.
\end{equation}
The number of
$n$-pairings is  $$(2n-1)!!=(2n-1)\times (2n-3) \times \dots \times 3 \times 1\, ,$$ 
so $\CC \mathcal{P}_n$ is a vector space of dimension 
$(2n-1)!!$.
There is a natural permutation action of the symmetric group $S_{2n}$
on $\mathcal{P}_n$  inducing a representation of 
$S_{2n}$ on $\CC\mathcal{P}_n$, which we denote  by $\rho_n$. By \cite[Theorem (6.2.1)]{MR2265844}, see also \cite[A.2.9]{MR1676282} or \cite[VII.2, (2.4)]{MR553598}),  the decomposition of $\rho_n$ into irreducible representations of $S_{2n}$ is given by 
\begin{equation} \rho_n = \bigoplus_{\substack{\lambda \vdash 2n \\ \lambda \,\text{has even rows}}}
\rho_\lambda\,.
\label{eq:decomp}
\end{equation}
Each representation appears in the decomposition with multiplicity~1. We denote by
\begin{equation} \CC\mathcal{P}_n = \bigoplus_{\substack{\lambda \vdash 2n \\ \lambda \,\text{has even rows}}}
M_\lambda
\label{eq:decomp_1}
\end{equation}
the underlying canonical decomposition of vector spaces.

\begin{example}
\label{ex_2_pairings}
For $n=2$, we have 
$$
\CC\mathcal{P}_2 = 
\Bigl\{ x \cdot [(12)(34)] + y \cdot [(13)(24)] + z \cdot [(14)(23)] \Bigr\}\, .
$$
which is the direct sum of the trivial representation $\rho_{(4)}$ of $S_4$ on the line $x = y = z$ and the representation $\rho_{(2,2)}$ on the plane $x+y+z=0$. In general, $\rho_n$ always contains a copy of the trivial representation $\rho_{(2n)}$ of 
$S_{2n}$, generated by $\sum_{P \in \mathcal{P}_n} P$, and the corresponding Young diagram has a single row with $2n$ boxes.
\end{example}

\begin{definition} \label{def_crossing_loop}
We associate the following
numbers to $n$-pairings:
\begin{enumerate}
     \item[$\bullet$]
Two pairs in an $n$-pairing form a {\em crossing} if they can be written as $(i, k)$ and $(j, \ell)$ with $i<j<k<\ell$. We denote by $c(P)$ the total number of crossings in an $n$-pairing $P$.
\item[$\bullet$]
The {\em loop number} $L(P_1,P_2)$ associated to 
two pairings $P_1$ and $P_2$ is
the number of loops in the union of their arc diagrams.
The product of permutations $P_1 P_2$ has $2L$ cycles. \hspace*{\fill} $\Diamond$
\end{enumerate}
\end{definition}
\vspace{8pt}

\begin{example}
Consider the pairing $P_1=(14)(25)(37)(68)$ from the previous example and the pairing $P_2 = (12)(34)(57)(68)$. Then $c(P_1) = 4$, $c(P_2) = 1$, and $L(P_1, P_2 ) = 2$, see Figure~\ref{Fig:LoopNumber}. 
\end{example}

\begin{figure}
\resizebox{.5\linewidth}{!}{\input{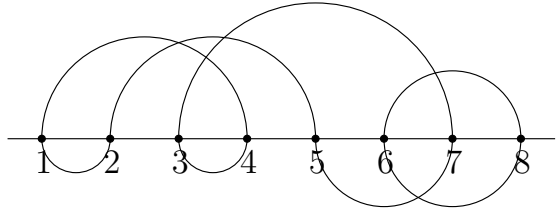}}
    \caption{The loop number of pairings $(14)(25)(37)(68)$ and $(12)(34)(57)(68)$ equals 2.}
    \label{Fig:LoopNumber}
\end{figure}


\subsection{Invariant theory for the orthogonal group} \label{orthh}
Let $V$ be a $k$-dimensional complex vector space endowed with a non-degenerate 
symmetric bilinear form $\alpha$, and let $$O(V) \subset GL(V)$$
be the corresponding orthogonal group.

For every $1 \leq i, j \leq 2n$, define
\begin{align}
\nonumber
   \alpha_{ij} \colon V^{\oplus 2n} & \longrightarrow \CC \\
\nonumber
   (v_1, \ldots, v_{2n}) & \longmapsto \alpha(v_i,v_j)\,.
\end{align}
For a pair $p=(i,j)$, we write $\alpha_p := \alpha_{ij}$. For an $n$-pairing $P \in \mathcal{P}_n$, we define 
\begin{equation} \label{eq_alpha}
\alpha_P := \prod_{p \in P} \alpha_p \in (V^*)^{\otimes 2n} \,
\end{equation}
by multilinearity.
As $\alpha_P$ lies in the invariant subspace 
$((V^*)^{\otimes 2n})^{O(V)} \subset (V^*)^{\otimes 2n}$ under the action of $O(V)$, 
we obtain a linear map 
\begin{align}
    \label{Eq: varphi}
   \varphi \colon \CC\mathcal{P}_n & \longrightarrow ((V^*)^{\otimes 2n})^{O(V)} \\
   \nonumber
   P & \longmapsto \alpha_P \, , 
\end{align}
which is a morphism of  
representations of the symmetric group  $S_{2n}$ since $\alpha$ is symmetric. The action of $S_{2n}$ on $\CC\mathcal{P}_n$ is defined by the representation $\rho_n$ in \eqref{eq:decomp}, 
and the $S_{2n}$-action  on $((V^*)^{\otimes 2n})^{O(V)}$ defined by the permutation of factors.

Basic results in  the invariant theory
of $O(V)$ characterize the invariant subspace $((V^*)^{\otimes 2n})^{\mathrm{O}(V)}\subset
V^{\otimes 2n}$.

\begin{theorem}
\label{Thm: fund inv o}
For every $n \in \Z_{\geq 0}$, the map $$ \varphi \colon \CC\mathcal{P}_n \longrightarrow ((V^*)^{\otimes 2n})^{\mathrm{O}(V)}$$ of \eqref{Eq: varphi} is surjective with kernel 
\begin{equation} \label{ker_orthogonal}
\mathrm{Ker}\, \varphi = \bigoplus_{\substack{\lambda \vdash 2n \\ \lambda \,\text{has even rows}\\ \ell(\lambda)
\geq k+1}}
M_\lambda\ \
\subset  \ \
\CC\mathcal{P}_n = \bigoplus_{\substack{\lambda \vdash 2n \\ \lambda \,\text{has even rows}}}
M_\lambda \,.
\end{equation}
In other words, the map $\varphi$ induces an isomorphism 
\begin{equation} \label{eq_inv_orthogonal}
    ((V^*)^{\otimes 2n})^{\mathrm{O}(V)} \simeq  \bigoplus_{\substack{\lambda \vdash 2n \\ \lambda \,\text{has even rows}\\ \ell(\lambda)
\leq k}}
M_\lambda \ \subset\ \CC \mathcal{P}_n\,.
\end{equation}
Moreover, for every $n \in \Z_{\geq 0}$, 
the space $((V^*)^{\otimes 2n+1})^{\mathrm{O}(V)}$ equals~0.
\end{theorem}

\begin{proof}
The surjectivity of $\rho$ follows from the first fundamental theorem of invariant theory for the orthogonal group, see \cite[\S 11.2.1, Thm, p390]{MR2265844}. 
By the second fundamental theorem of invariant theory for the orthogonal group
in the form of
\cite[\S 11.6.3.1, Theorem 1, p413]{MR2265844}, the kernel of $\varphi$ is given by \eqref{ker_orthogonal}.

The vanishing of the invariant subspace of $(V^*)^{\otimes 2n+1}$ follows from the fact that $$-{\rm Id} \in \mathrm{O}(V)$$ 
acts as $- {\rm Id}$ on $(V^*)^{\otimes 2n+1}$, so the only invariant vector is~0.
\end{proof}

\begin{example}
According to Theorem~\ref{Thm: fund inv o}, every $\mathrm{O}(V)$-invariant element of $(V^*)^{\otimes 4}$ has the form
$$
x \cdot \alpha_{12} \alpha_{34} + y \cdot \alpha_{13} \alpha_{24}+  z \cdot \alpha_{14}\alpha_{23}\,.
$$
If $k \geq 2$, then $x,y,z$ are simply the coordinates on the invariant space. Indeed, the two partitions of 4 into even elements are $(4)$ and $(2,2)$, and none of them has $\geq k+1$ elements.

The case $k=1$ is special, since then
$2 \geq k+1$. By Theorem~\ref{Thm: fund inv o},
the space $M_{(2,2)}$ must lie in the kernel of~$\varphi$. 
To check this claim directly, we denote by $u_i$ the coordinate in the $i$th copy of~$V$. Then, $\alpha_{ij} = u_i u_j$. Thus 
$$
 \alpha_{12} \alpha_{34} = \alpha_{13} \alpha_{24}= \alpha_{14}\alpha_{23} = u_1u_2u_3u_4\,.
$$
The relations span the kernel $x+y+z=0$ of $\varphi$, which is precisely $M_{(2,2)}$. 
\end{example}

\subsection{Invariant theory for the symplectic group} \label{sympll}
Let $V$ be a $2k$-dimensional complex vector space endowed with a non-degenerate 
skew-symmetric bilinear form $\omega$. 
Let $$\mathrm{Sp}(V) \subset GL(V)$$ be the corresponding symplectic group.
The invariant theory for the symplectic group is very similar to the
orthogonal case, but care must be taken with the signs.

For every $1 \leq i, j \leq 2n$, define
\begin{align}
   \nonumber
   \omega_{ij} \colon V^{\oplus 2n} & \longrightarrow \CC \\
   \nonumber
   (v_1, \ldots, v_{2n}) & \longmapsto \omega(v_i,v_j)\,. 
\end{align}
For a pair $p=(i,j)$ with $i<j$, we write 
$\omega_p := \omega_{ij}$.
For an $n$-pairing $P \in \mathcal{P}_n$, we 
define 
\begin{equation}
    \label{Eq: omegap}
    \omega_P := (-1)^{c(P)}\prod_{p \in P} \omega_p \in (V^*)^{\otimes 2n},
\end{equation}
where $c(P)$ is the number of crossings in~$P$
(see Definition \ref{def_crossing_loop}). Clearly, $\omega_P$ lies in the invariant subspace 
$((V^*)^{\otimes 2n})^{\mathrm{Sp}(V)}$, so we obtain a linear map 
\begin{align}
       \label{Eq:psi}
 \psi \colon \CC\mathcal{P}_n & \longrightarrow ((V^*)^{\otimes 2n})^{\mathrm{Sp}(V)} \\
 \nonumber
P & \longmapsto \omega_P \,.
\end{align}

The sign convention involving the number of crossings in the definition of $\omega_P$ in~\eqref{Eq: omegap}
is necessary for the following result to hold.

\begin{lemma} \label{lem_sign}
The map $\psi \colon \CC\mathcal{P}_n \to ((V^*)^{\otimes 2n})^{\mathrm{Sp}(V)}$ of \eqref{Eq:psi} is a morphism of representations of $S_{2n}$, where 
\begin{enumerate}
    \item [(i)] the action of  $S_{2n}$ on
    $\CC\mathcal{P}_n$ is given 
    by $\rho_n \otimes \epsilon$, where $\rho_n$ is as in \eqref{eq:decomp} and $\epsilon$ is the sign representation,
    \item[(ii)] the action of  $S_{2n}$ on $((V^*)^{\otimes 2n})^{\mathrm{Sp}(V)}$ is given by permutations of the $2n$ factors.
\end{enumerate}
\end{lemma}

\begin{proof}
It is enough to show, for every elementary transposition $\tau = (i,i+1) \in S_{2n}$, we have 
\begin{equation}\label{eq_relation}
\omega_{\tau \cdot P}
= - \tau \cdot \omega_P\,.\end{equation}

\noindent $\bullet$ If $(i,i+1)$ is a pair of the pairing~$P$, then $\tau \cdot P=P$, so $\omega_{\tau\cdot P} = \omega_P$. On the other hand,
$\tau \cdot \omega_{P}=-\omega_P$, because $\omega$ is skew-symmetric, and  \eqref{eq_relation} holds.

\noindent $\bullet$ If $i$ and $i+1$ belong to two different pairs, then let $\tau \cdot P = P'$. We have $\omega_{\tau\cdot P} = \omega_{P'}$. The pairing $P'$ has exactly one crossing more or less than~$P$. If we sort the elements of each pair of~$P$ in increasing order and then permute $i$ and $i+1$, the elements of every pair of $P'$ are still sorted in  increasing order. Thus we have $\tau \cdot \omega_P = - \omega_{P'}$, so \eqref{eq_relation} holds again.
\end{proof}

Basic results in  the invariant theory
of $\mathrm{Sp}(V)$ characterize the invariant subspace
$((V^*)^{\otimes 2n})^{\mathrm{Sp}(V)} \subset (V^*)^{\otimes 2n}$.

\begin{theorem}
\label{Thm: fund inv sp}
For every $n \in \Z_{\geq 0}$, the map $$\psi \colon \CC\mathcal{P}_n \longrightarrow ((V^*)^{\otimes 2n})^{\mathrm{Sp}(V)}$$ of \eqref{Eq:psi} is surjective with kernel 
\begin{equation}
\label{Eq: ker psi}
    \mathrm{Ker}\, \psi =  \bigoplus_{\substack{\lambda \vdash 2n \\ \lambda \,\text{has even rows}\\ 
    \lambda_1 \geq 2k+2}} M_\lambda
\ \subset\  
\CC\mathcal{P}_n \,.
\end{equation}
In other words, the map $\psi$ induces an isomorphism 
\begin{equation} \label{eq_inv_symplectic}
    ((V^*)^{\otimes 2n})^{\mathrm{Sp}(V)} \simeq  \bigoplus_{\substack{\lambda \vdash 2n \\ \lambda \,\text{has even rows}\\ 
    \lambda_1 \leq 2k}} M_\lambda    \ \subset\ \CC \mathcal{P}_n\,.
\end{equation}
Moreover, for every $n \in \Z_{\geq 0}$, 
the space $((V^*)^{\otimes 2n+1})^{\mathrm{Sp}(V)}$ equals~0.
\end{theorem}

\begin{proof}
The surjectivity of $\psi$ follows from the first fundamental theorem of invariant theory for the symplectic group, see \cite[\S 11.2.1, Thm, p390]{MR2265844}. 
By Lemma \ref{lem_sign}, we can apply the second fundamental theorem of invariant theory for the symplectic group
in the form of
\cite[\S 11.6.3.2, Theorem 2, p414]{MR2265844}.
The kernel of $\psi$ is then given as a subrepresentation of $\rho_n \otimes \epsilon$
by 
\begin{equation} 
 \bigoplus_{\substack{\lambda \vdash 2n \\ \lambda \,\text{has even columns}\\ \ell(\lambda)
\geq 2k+2}}
\rho_\lambda \,.
\end{equation}
As $\rho_{\lambda} \otimes \epsilon = \rho_{\lambda^t}$,
where $\lambda^t$ is the transposed Young diagram, the kernel of 
$\psi$, viewed simply as a linear subspace of $\CC\mathcal{P}_n$, is given by \eqref{Eq: ker psi}.

The vanishing of the invariant subspace of $(V^*)^{\otimes 2n+1}$ follows from the fact that $$-{\rm Id} \in \mathrm{Sp}(V)$$ acts as $- {\rm Id}$ on $(V^*)^{\otimes 2n+1}$, so the only invariant vector is~0.
\end{proof}

\begin{example}
According to Theorem~\ref{Thm: fund inv sp}, every $\mathrm{Sp}(V)$-invariant element of $(V^*)^{\otimes 4}$ has the form
$$
x \cdot \omega_{12} \omega_{34}- y \cdot \omega_{13} \omega_{24}+  z \cdot \omega_{14}\omega_{23}\,.
$$
If $k \geq 2$, then $x,y,z$ are simply the coordinates on the invariant space. Indeed, the two partitions of 4 into even elements are $(4)$ and $(2,2)$, and none of them has elements $\geq 2k+2$. 

The case $k=1$ is special, since $4 \geq 2k+2$. 
By Theorem \ref{Thm: fund inv sp},
the space $M_{(4)}$ must lie in the kernel of~$\psi$. To check this claim directly,
we denote by $u_i, v_i$ the coordinates in the $i$th copy of~$V$. Then, $\omega_{ij} = u_i v_j - v_j u_i$. These polynomials satisfy the well-known Pl\"ucker relation
$$
\omega_{12} \omega_{34} - \omega_{13} \omega_{24} + \omega_{14}\omega_{23} = 0\,.
$$
This relation spans the kernel $x=y=z$ of $\psi$, which is precisely $M_{(4)}$. 
\end{example}

\subsection{The loop matrix}
\label{sec_matrix}
Building on the representation theory  reviewed in 
\S \ref{section_invariant_theory} - \S \ref{sympll}, we describe here a family 
of matrices 
which will be used in \S \ref{sec_trading}
to show that primitive insertions in Gromov--Witten theory of complete intersections can be traded against nodes. Such matrices have also
appeared in various places in the literature, particularly in
\cite[\S3]{HW}, \cite[\S VII.2]{MR553598}, \cite[\S 3]{MR2586868},
and, in the study of  the tautological ring, in \cite{Tavakol}.

\begin{definition} \label{def_matrix}
For every $n \in \Z_{\geq 0}$,  the \emph{loop matrix} $\mathsf{M}(n,x)$ is
the $(2n-1)!! \times (2n-1)!!$ matrix with entries 
indexed by pairs of $n$-pairings $P$, $P'$ and given by 
\[ \mathsf{M}(n,x)_{P,P'}:=x^{L(P,P')}\, ,\]
where $L(P,P')$ is the loop number of the pairings $P$ and $P'$ (see
Definition \ref{def_crossing_loop}), and $x$ is a formal variable (which will be eventually specialized to integer values).

The matrix  $\mathsf{M}(n,x)$ canonically acts as an endomorphism of the vector space $\CC\mathcal{P}_n =\bigoplus_{P \in \CC \mathcal{P}_n} \CC P$.
\hspace*{\fill} $\Diamond$
\end{definition}

\begin{proposition} \label{prop_eigenvalue}
The direct sum decomposition 
\[ \CC\mathcal{P}_n =\bigoplus_{\substack{\lambda \vdash 2n \\ \lambda \,\text{has even rows}}}
M_\lambda\,.\]
given in \eqref{eq:decomp_1} is an eigenspace decomposition for the loop matrix $\mathsf{M}(n,x)$.
More precisely, for every partition $\lambda \vdash 2n$ with even rows, the subspace $M_\lambda \subset \CC\mathcal{P}_n$
is an eigenspace of $\mathsf{M}(n,x)$ for the eigenvalue 
\begin{equation}\label{eigfor}
 \prod_{(i,j) \in \frac{1}{2}\lambda} (x-i+2j-1) \,,
 \end{equation}
where $\frac{1}{2} \lambda$ is the partition whose parts are half of those of $\lambda$. The product \eqref{eigfor}
is over the boxes of the Young diagram of $\frac{1}{2}\lambda$, where $i$ is the row index, starting at $1$ and increasing when going downwards, and $j$ is the column index, starting at $1$ and increasing when going from left to right.
\end{proposition}

\begin{proof}
There are several existing proofs in the literature. Following Zinn-Justin \cite[\S 3]{MR2586868}, there exists a lift of $\mathsf{M}(n,x)$ to an element of the group algebra of $S_{2n}$ expressed in terms of Jucys-Murphy elements \cite[Lemma 1]{MR2586868}. The final result can be found in the paragraph above \cite[Proposition 5]{MR2586868}. An alternative proof by  Hancon-Wales
can be found in \cite[Theorem 3.1]{HW}. A more expository 
presentation with the same approach is contained in \S VII.2 of Macdonald's book \cite{MR553598}, where a more general result follows from the theory of zonal symmetric polynomials. 

We summarize here 
the approach taken in Macdonald's book.
Let $P$ and $P'$ be two $n$-pairings. Each loop in the graph obtained by gluing together the arc diagrams of $P$ and $P'$ has even length. The half-lengths of these loops define a partition $\lambda(P,P')$ of $n$. In particular, the loop number 
    $L(P,P')$ is exactly the length of this partition: $L(P,P')=\ell(\lambda(P,P'))$.
    
    Let $\mathsf{M}(n,x_1,\dots,x_N)$ be the 
    $(2n-1)!! \times (2n-1)!!$ matrix whose entries indexed by pairs $P$ and $P'$ of 
    $n$-pairings are given by the power-sum symmetric polynomials $p_{\lambda(P,P')}(x_1,\cdots, x_N)$ in some number $N$ of variables. 
    According to 
    \cite[VII.2, Example 5]{MR553598}, for every $\lambda \vdash 2n$ with even rows,
    the subspace $M_\lambda \subset \CC\mathcal{P}_n$ is an eigenspace for
    $\mathsf{M}(n, x_1,\dots, x_N)$, and the corresponding eigenvalue is the zonal symmetric 
    polynomial $Z_{\frac{1}{2}\lambda}(x_1,\cdots, x_N)$. We now consider the 
    specialization $x_1=\dots=x_N=1$. 
    As 
    \[ p_{\lambda(P,P')}(1,\dots, 1)=N^{\ell(\lambda(P,P'))}=N^{L(P,P')}\,,\] 
    we have 
    $\mathsf{M}(n, 1,\dots,1)=\mathsf{M}(n, N)$. On the other hand, we have 
    $$Z_{\frac{1}{2}\lambda}(1,\dots,1)=\prod_{(i,j) \in\frac{1}{2}\lambda} (N-i+2j-1)$$ by 
    \cite[VII.2, Eq (2.24)-(2.25)]{MR553598}, which proves Proposition 
    \ref{prop_eigenvalue} for $x$ specialized to a positive integer $N$. 
    As the dependence of $\mathsf{M}(n,x)$ on $x$ is polynomial, the argument actually proves 
    Proposition 
    \ref{prop_eigenvalue} in general for any $x$.
\end{proof}

\begin{corollary} \label{cor_1}
For every $k \in \Z_{\geq 0}$, the endomorphism
$\mathsf{M}(n, k)$ preserves the direct sum decomposition  
\[\CC\mathcal{P}_n =\bigoplus_{\substack{\lambda \vdash 2n \\ \lambda \,\text{has even rows}}}
M_\lambda\,.\] 
Moreover, $\mathsf{M}(n,k)$ is zero when restricted to the subspace 
\[\bigoplus_{\substack{\lambda \vdash 2n \\ \lambda \,\text{has even rows}\\ \ell(\lambda)
\geq k+1}}
M_\lambda
\subset 
\CC\mathcal{P}_n\]
and invertible when restricted to the subspace \[\bigoplus_{\substack{\lambda \vdash 2n \\ \lambda \,\text{has even rows}\\ \ell(\lambda)
\leq  k}}
M_\lambda
\subset 
\CC\mathcal{P}_n\,.\]
In particular, for every $k$-dimensional complex vector space $V$ endowed with a non-degenerate symmetric bilinear form $\alpha$, the endomorphism 
$\mathsf{M}(n,k)$ preserves the subspace of invariants $((V^*)^{\otimes 2n})^{O(V)}$, viewed as a subspace of 
$\CC \mathcal{P}_n$ via \eqref{eq_inv_orthogonal}, 
and the restriction of $\mathsf{M}(n,k)$ to the subspace 
of invariants is invertible.
\end{corollary}

\begin{proof}
By Proposition \ref{prop_eigenvalue}, 
$M_\lambda \subset \CC\mathcal{P}_n$ is an eigenspace of $\mathsf{M}(n,k)$ for the eigenvalue
\[ c_\lambda(k) :=\prod_{(i,j) \in \frac{1}{2}\lambda} (k-i+2j-1)\,.\]
We have
$\ell(\frac{1}{2} \lambda)=\ell(\lambda)$.
If $\ell(\lambda) \geq k+1$, there exists $(i,j) \in \frac{1}{2}\lambda$ with 
$i=k+1$ and $j=1$, and so $c_\lambda(k)=0$.
If $\ell(\lambda) \leq k$, then, for every $(i,j) \in \frac{1}{2}\lambda$, we have 
$i \leq k$. As $j \geq 1$, we obtain
$k-i+2j-1 \geq 1$, and so  $c_\lambda(k) \neq 0$.
\end{proof}

\begin{corollary} \label{cor_2}
For every $k \in \Z_{\geq 0}$, the endomorphism
$\mathsf{M}(n,-2k)$ 
 preserves the direct sum decomposition  
\[\CC\mathcal{P}_n =\bigoplus_{\substack{\lambda \vdash 2n \\ \lambda \,\text{has even rows}}}
M_\lambda\,.\] 
Moreover, $\mathsf{M}(n,-2k)$ is zero when restricted to the subspace 
\[\bigoplus_{\substack{\lambda \vdash 2n \\ \lambda \,\text{has even rows}\\ \lambda_1
\geq 2k+2}}
M_\lambda
\subset 
\CC\mathcal{P}_n\]
and invertible when restricted to the subspace \[\bigoplus_{\substack{\lambda \vdash 2n \\ \lambda \,\text{has even rows}\\ \lambda_1
\leq  2k}}
M_\lambda
\subset 
\CC\mathcal{P}_n\,.\]
In particular, for every $2k$-dimensional complex vector space $V$ endowed with a non-degenerate skew-symmetric bilinear form $\omega$, the endomorphism 
$\mathsf{M}(n,-2k)$ preserves the subspace of invariants $((V^*)^{\otimes 2n})^{\mathrm{Sp}(V)}$, viewed as a subspace of $\CC \mathcal{P}_n$ using 
\eqref{eq_inv_symplectic}, and the restriction of $\mathsf{M}(n,-2k)$ to the subspace of invariants is invertible.
\end{corollary}

\begin{proof}
By Proposition \ref{prop_eigenvalue}, the eigenvalue of $\mathsf{M}(n,-2k)$ restricted to 
$M_\lambda \subset \CC\mathcal{P}_n$ is  
\[ d_\lambda(k) :=\prod_{(i,j) \in \frac{1}{2}\lambda} (-2k-i+2j-1)\,.\]
If $\lambda_1 \geq 2k+2$, then $\frac{\lambda_1}{2} \geq k+1$. There exists $(i,j) \in \frac{1}{2}\lambda$ with 
$i=1$ and $j=k+1$, and so $d_\lambda(k)=0$.
If $\lambda_1 \leq 2k$, then $\frac{\lambda_1}{2} \leq k$, and
for every $(i,j) \in \frac{1}{2}\lambda$, we have 
$j \leq k$, and as $i \geq 1$, we obtain 
$-2k-i+2j-1 \leq -2$, and so  $d_\lambda(k) \neq 0$.
\end{proof}

\begin{example}
For $n=2$, we have $3$ pairings: $P_1=(12)(34)$, $P_2=(13)(24)$, and $P_3=(14)(23)$, as in Example \ref{ex_2_pairings}. Computing directly the loop numbers, we find
\[ \mathsf{M}(2,x)=
\begin{pmatrix}
x^2 & x & x\\
x & x^2 & x \\
x& x& x^2
\end{pmatrix} \,,\]
which has a simple eigenvalue $x(x+2)$, corresponding to 
$\lambda=(4)$ in Proposition \ref{prop_eigenvalue}
(the Young diagram of $\frac{1}{2}\lambda$ is a row of two boxes, with 
$(i,j)=(1,1)$ and $(1,2)$), and a double eigenvalue 
$x(x-1)$, corresponding to $\lambda=(2,2)$ in Proposition \ref{prop_eigenvalue} 
(the Young diagram of $\frac{1}{2}\lambda$ is a column of two boxes, with 
$(i,j)=(1,1)$ and $(2,1)$).
The eigenspace for the eigenvalue 
$x(x+2)$ is the line $x=y=z$, and the eigenspace for the eigenvalue $x(x-1)$
is the plane $x+y+z=0$, as predicted by Proposition \ref{prop_eigenvalue} and Example \ref{ex_2_pairings}.

We leave as an exercise for the reader to write explicitly the $15 \times 15$ matrix $\mathsf{M}(3,x)$ and to check directly that the eigenvalues are $x(x+2)(x+4)$, $x(x+2)(x-1)$, and $x(x-1)(x-2)$ with multiplicities $1$, $9$, and $5$ respectively.
\end{example}

\subsection{The matrix of diagonal insertions}
\label{sec_matrix_diagonal}

In \S \ref{sec_trading}, we will invert a system of relations obtained by applying the splitting formula 
\eqref{eq_splitting}
to nodal invariants.
To write these relations explicitly, we must first effectively compute the effect of the insertions of the class of the diagonal in the splitting formula. 

Let $X$ be an $m$-dimensional complete intersection in projective space, and let 
$V=H^m(X)_{\prim} \otimes \C$ be the primitive cohomology. As in \S\ref{sec_monodromy}, let $\alpha$ be the non-degenerate symmetric intersection form on $V$ if $m$ is even, and let
$\omega$ be the non-degenerate skew-symmetric intersection form on $V$ if $m$ is odd.

For every $n \in \Z_{\geq 0}$ and every $n$-pairing $P \in \mathcal{P}_n$, we define a {\em test multivector} 
\[ \Delta_P \in V^{\otimes 2n}\]
as follows.

\vspace{8pt}
\noindent$\bullet$ \textbf{Test multivectors, symmetric case.} If $m$ is even, the intersection form~$\alpha$ on $V$ is symmetric. Denote by $\alpha^{-1}$ the inverse symmetric bi-vector in $V \otimes V$. For a pair $p = \{ i, j \}$, denote by $\alpha^{-1}_p$ the same bi-vector lying in the tensor product of the $i$th and $j$th copies of~$V$. For an $n$-pairing~$P$, we define  
\begin{equation}\label{eq_delta_P_alpha}
\Delta_P := \bigotimes_{p \in P} \alpha^{-1}_p\,.
\end{equation}

\vspace{8pt}
\noindent $\bullet$ \textbf{Test multivectors, skew-symmetric case.} If $m$ is odd, the intersection form~$\omega$ on $V$ is skew-symmetric. Denote by $\omega^{-1}$, the inverse skew-symmetric bi-vector in $V \otimes V$. For a pair $p = \{ i, j \}$, $i<j$ denote by $\omega^{-1}_p$ the same bi-vector lying in the tensor product of the $j$th and $i$th copies of~$V$ in that order (in  decreasing order of elements of the pair). For an $n$-pairing~$P$,
we define 
\begin{equation}\label{eq_delta_P_omega}
\Delta_P = (-1)^{c(P)} \bigotimes_{p \in P} \omega^{-1}_p\,,
\end{equation}
where $c(P)$ is the number of crossings
(defined in Definition \ref{def_crossing_loop}).

\begin{definition} \label{def_matrix_diagonal}
The \emph{matrix of diagonal insertions} $\mathsf{M}^\Delta(n,V)$ is the $(2n-1)!! \times (2n-1)!!$ matrix with entries indexed by pairs of $n$-pairings $P$, $P'$ and given by
$$
\mathsf{M}^\Delta(n,V)_{P, P'} = \alpha_P(\Delta_{P'}) 
$$
if $m$ is even, where $\alpha_P \in (V^{*})^{\otimes 2n}$ 
is as in \eqref{eq_alpha}, and 
$$
\mathsf{M}^\Delta(n,V)_{P, P'}= \omega_P(\Delta_{P'})
$$
if $m$ is odd, where $\omega_P \in (V^{\star})^{\otimes 2n}$ 
is as in \eqref{Eq: omegap}.
\hspace*{\fill} $\Diamond$
\end{definition}
\vspace{8pt}

The following result explicitly computes the matrix 
$\mathsf{M}^\Delta(n,V)$ in terms of the loop matrix
of Definition \ref{def_matrix}.

\begin{theorem}\label{thm_matrix_diagonal}
There are two cases:
\begin{itemize}
\item[(i)] If $m$ is even, we have 
$\mathsf{M}^\Delta(n,V)=\mathsf{M}(n,k)$,
where $V$ is of dimension $k$.

\item[(ii)] If $m$ is odd, we have
$\mathsf{M}^\Delta(n,V)=\mathsf{M}(n,-2k)$,
where $V$ is of dimension $2k$.
\end{itemize}
\end{theorem}


\vspace{4pt}
\noindent \textbf{Proof of Theorem \ref{thm_matrix_diagonal} for $m$ even.}
Let $(e_\mu)_{1\leq \mu\leq k}$ be an orthonormal basis of 
$(V,\alpha)$. We denote by 
$x^{i\mu}$ the corresponding linear coordinates on the $i$-th copy of $V$ in 
$V^{\otimes 2n}$. For every $1 \leq i, j \leq 2n$, we define
\begin{equation}\label{eq_Delta_ij}
\Delta_{ij} := \sum_{\mu=1}^k \frac{\partial}{\partial x^{i\mu}}
\otimes \frac{\partial}{\partial x^{j\mu}} \,,\end{equation}
\begin{equation} \label{eq_alpha_ij}
\alpha_{ij} := \sum_{\mu=1}^k dx^{i\mu}
\otimes dx^{j\mu} \,.\end{equation}
In particular, for every pair $p=(i,j)$, we have 
$\alpha_p^{-1} =\Delta_{ij}$, and $\alpha_p=\alpha_{ij}$.

\begin{lemma} \label{lem_diag_ortho}  We have
\begin{enumerate}
\item[(i)] for every $1 \leq i,j \leq 2n$, $\alpha_{ij}
(\Delta_{ij})=k$,
\item[(ii)] for every $1 \leq i,j,p,q \leq 2n$, 
$(\alpha_{ip} \otimes \alpha_{qj})(\Delta_{pq})=\alpha_{ij}$.
\end{enumerate}
\end{lemma}
\begin{proof}
The claim follows from \eqref{eq_Delta_ij} and \eqref{eq_alpha_ij} by a direct calculation.
\end{proof}

To prove Theorem \ref{thm_matrix_diagonal} for $m$ even, we must show
\begin{equation}\label{msloop}
\alpha_P(\Delta_{P'})=k^{L(P,P')}\, 
\end{equation}
for all $n$-pairings $P$ and $P'$.
Equation \eqref{msloop} follows from
Lemma \ref{lem_diag_ortho} and by viewing graphically the contractions of tensors in terms of the arc diagrams of $P$ and $P'$. \hspace*{\fill} $\blacklozenge$

\vspace{8pt}
\noindent \textbf{Proof of Theorem \ref{thm_matrix_diagonal} for $m$ odd.}
 The argument is parallel to the $m$ even case, but
 some care with signs is required.

Let $(e_\mu,f_\mu)_{1 \leq \mu \leq k}$ a symplectic basis of $(V,\omega)$: 
\begin{eqnarray*}
\forall \mu,\nu, \ \ \ & \omega(e_\mu,e_\nu)=\omega(f_\mu,f_\nu)=0\, , \\
\forall \mu\neq \nu, &
\omega(e_\mu,f_\nu)=0\, ,\\
\forall \mu, \ \ \ \ \ \, &
\omega(e_\mu,f_\mu)=-\omega(f_\mu,e_\mu)=1\, .
\end{eqnarray*}
We denote by $x^{i\mu}$ and $y^{i\mu}$ the corresponding linear coordinates on the $i$-th copy of $V$ in $V^{\otimes 2n}$. 
For every $1 \leq i,j\leq 2n$, we define
\begin{equation} \label{eq_Delta_ij_symp}
\Delta_{ij}=\sum_{\mu=1}^k
\left( \frac{\partial}{\partial x^{i\mu}} \otimes \frac{\partial}{\partial y^{j\mu}} - 
\frac{\partial}{\partial x^{j\mu}} \otimes \frac{\partial}{\partial y^{i\mu}} \right)\,,\end{equation}
\begin{equation} \label{eq_alpha_symp}
    \omega_{ij}=\sum_{\mu=1}^k (dx^{i\mu} \otimes dy^{j\mu} 
    - dx^{j\mu} \otimes dy^{i\mu})\,.
\end{equation}
In particular, for every pair $p=(i,j)$ with $i<j$, we have $\omega_p^{-1}=\Delta_{ji}$ and $\omega_p=\omega_{ij}$.
\begin{lemma} \ \label{lem:cycle_symp} We have
\begin{itemize}
\item[(i)] For every $1 \leq i,j \leq 2n$, $\omega_{ij}(\Delta_{ji})=-2k$.
\item[(ii)]For every $1\leq i,j,p,q \leq 2n$, 
$(\omega_{ip} \otimes \omega_{qj})(\Delta_{pq})
=\omega_{ji}$.
\end{itemize}
\end{lemma}

\begin{proof} The claim follows from
\eqref{eq_Delta_ij_symp} and \eqref{eq_alpha_symp}
by a direct calculation. \end{proof}

\begin{remark} The sign rule in the last equality is to choose a cyclic order on the set $\{i,p,q,j\}$ and follow it for the 2-forms $\omega$ and the bi-vector $\Delta$:
$\omega_{ip}$, then $\Delta_{pq}$, then $\omega_{qj}$, and we get $\omega_{ji}$. 
\end{remark}

\begin{lemma} \label{Prop:cycle}
Let $i_1, \dots, i_{2h}$ be pairwise different indices among $1, \dots, 2n$. Then, 
\[(\omega_{i_1i_2}\otimes \omega_{i_3i_4}\otimes \dots \otimes \omega_{i_{2h-1}i_{2h}})(\Delta_{i_2i_3}\otimes \Delta_{i_4i_5} \otimes\dots \Delta_{i_{2m}i_1})=(-1)^h (2k)\,.\]
\end{lemma}

\begin{proof} The result follows from Lemma \ref{lem:cycle_symp} 
by induction. \end{proof}

\begin{lemma} \label{lem_tree}
Consider a simple closed oriented curve crossing a horizontal line at $2n$ points. Then, the number of arcs oriented from left to right 
and the number of arcs oriented from right to left
are both odd.
\end{lemma}

\begin{proof}
Assume for simplicity that the curve is oriented clockwise.
The line cuts the region surrounded by the curve into $n+1$ parts of which $a$ lie above the line and $b$ below the line. Consider the bipartite tree whose vertices correspond to these parts: two vertices are joined by an edge if the corresponding parts are adjacent, see the example in  Figure \ref{fig:tree}.

\begin{figure}
\resizebox{.4\linewidth}{!}{\input{meander.pspdftex}}
   \caption{}
\label{fig:tree}
\end{figure}

Each region above the line contains exactly one left-to-right arc. A region below the line contains $d-1$ left-to-right arcs, where $d$ is the degree of the corresponding vertex in the tree. After summing  over all vertices, we obtain
$a + n-b$ left-to-right arcs.  Modulo 2, we have
$$a+n-b = a+n+b= n+(n+1) = 2n+1\, ,$$
which is always odd.
\end{proof}

\begin{lemma} \label{Lem:clr}
Consider $L$ closed oriented curves with $\ell$ simple intersections or self-intersections
crossing a horizontal line at $2n$ points. Let $s$ be the number of arcs joining one of these $2n$ points to another and oriented from left to right. Then $L+\ell+s$ is even.
\end{lemma}

\begin{proof}
Any intersection or self-intersection can be uniquely resolved so that the orientations of the arcs are preserved. The resolution decreases $\ell$ by 1, changes $L$ by~1, and leaves $s$ invariant. The parity of $L+\ell+s$ does not change. When there are no intersections left, we have $L$ independent closed curves, so the assertion follows from Lemma \ref{lem_tree}.
\end{proof}

To prove Theorem \ref{thm_matrix_diagonal} for $m$ odd, we must show 
$$\omega_P(\Delta_{P'})=(-2k)^{L(P,P')}$$ 
for all $n$-pairings $P$ and $P'$.
We represent $P$ and $P'$ by arc diagrams with arcs respectively above and below a horizontal line with $2n$ dots. The arcs above the line are oriented from left to right as the pairs in $\omega_{P}$, while the arcs below the line are oriented from right to left as the pairs of $\Delta_{P'}$. We obtain a collection of $L(P,P')$ intersecting loops. The total number $\ell$ of intersections equals the total number of crossings $c(P)+c(P')$ in both pairings. By definition of $\omega_{P}$ and $\Delta_{P'}$
in \eqref{Eq: omegap} and \eqref{eq_delta_P_omega}, we have a factor $(-1)^{c(P)}(-1)^{c(P')}=(-1)^\ell$ before the product of $\omega_p$'s and $\omega_p^{-1}$'s.

Now we change the orientation of some of the arcs so as to obtain a union of {\em oriented} loops. We change the signs of $\omega_{ij}$'s and $\Delta_{ij}$'s corresponding to these arcs accordingly.
Let $s$ be the number of arcs in the new collection of oriented curves oriented from left to right.
Then the change of orientation changes the sign by 
$(-1)^{n-s}$, which by Lemma~\ref{Lem:clr} is equal to $(-1)^{n+\ell+L(P,P')}$. Now, by Lemma~\ref{Prop:cycle} every closed loop of length~$2h$ contributes a factor of $(-1)^h (2k)$ to the contraction. The product of signs $(-1)^h$ over all loops gives $(-1)^n$. 
Finally, we obtain
$$\ \ \ \ \ \omega_P(\Delta_{P'})=
(-1)^\ell \cdot (-1)^{n+\ell+L(P,P')} (-1)^n (2k)^{L(P,P')} = (-2k)^{L(P,P')}\,.
\ \ \ \ \ \ \ \ \ \ \ \ \ \ \ \
\hspace*{\fill} \blacklozenge$$

\section{Gromov--Witten theory of complete intersections}
\label{sec:gw_ci2}

In this section we  present the main result of the paper: an algorithm computing Gromov–Witten invariants with arbitrary insertions of all smooth complete intersections in projective  space.
The main idea is to trade primitive insertions against nodes. The precise formulation of this idea is presented in  \S\ref{sec_trading}.
Once Gromov--Witten invariants with primitive insertions are turned into nodal Gromov--Witten invariants without primitive insertions, we  use the general nodal Gromov--Witten theory developed in \S \ref{NGWTh}-\ref{section_splitting} to determine an algorithm in \S \ref{section_algorithm}-\S \ref{algor}. As a main consequence, we prove that the Gromov--Witten classes of complete intersections are tautological in \S \ref{sec_tautological}. 

\begin{example}\label{ex:elliptic}
To illustrate this idea of trading primitive insertions against nodes, consider the example of an elliptic curve{\footnote{$E$ may be viewed here as a cubic in $\PP^2$.}} $E$ with cycles $a, b \in H^1(E)$. Suppose we want to compute Gromov--Witten invariants in genus~$g$,  degree $d$, with simple insertions and two primitive insertions $u_1a+v_1b$ and $u_2a + v_2 b$. Such an invariant is equal to 
\begin{equation}
    \label{exxx}
u_1 u_2 \langle a,a \rangle^E + u_1 v_2 \langle a,b \rangle^E + v_1 u_2 \langle b,a \rangle^E + v_1 v_2 \langle b,b \rangle^E \,,
\end{equation}
where we have suppressed all the simple
insertions (and almost everything else)
in the notation.
A priori, we have four constants 
in \eqref{exxx} to determine. However, using the deformation invariance of Gromov--Witten invariants, we can reduce four constants to just one. Indeed, by a deformation of $E$ we can transform the cycles $$(a,b) \mapsto (a,a+b)\, .$$
It follows that $\langle a,a \rangle^E = 0$ and similarly $\langle b,b \rangle^E = 0$ and $\langle a,b \rangle^E =- \langle b,a \rangle^E$. 
Hence, the Gromov--Witten invariant \eqref{exxx} is equal to 
$$
(u_1 v_2 - u_2 v_1) \langle a,b \rangle^E.
$$

To determine $\langle a,b \rangle^E$, we use nodal Gromov--Witten invariants of genus $g+1$. By the splitting formula, the node can be replaced by two marked points with insertions 
$$1 \otimes p + a \otimes b - b \otimes a + p \otimes 1\, ,$$
where $p \in H^2(E)$ is the class of a point. Thus, the nodal Gromov--Witten invariant is equal to a sum of Gromov--Witten invariants with simple insertions (1 and $p$) and $2 \, \langle a,b \rangle^E$. We have expressed $\langle a,b \rangle^E$ in terms of a nodal Gromov--Witten invariant and Gromov--Witten invariants with simple insertions. 
\end{example}

We will generalize the above discussion of $E$ to an arbitrary number of primitive insertions and to all complete intersections. While the monodromy on
$H^1(E)$ played a central role in \cite{Janda,OP},
the addition input there was a set of 
{\em elliptic vanishing relations}. While
the method of elliptic vanishing does {\em not}
easily extend to higher dimensions, trading 
against nodes does.

\subsection{Trading primitive insertions against nodes}
\label{sec_trading}
Using the algebraic results of \S \ref{sec:gw_ci1},
we trade primitive insertions in Gromov--Witten theory of complete intersections for nodes, as illustrated for the elliptic curve in Example \ref{ex:elliptic}.

\begin{theorem}[\textbf{Theorem \ref{thm_intro_1}}]
\label{thm_nodes}
Let $X$ be a complete intersection in projective space which is not a cubic surface or an even dimensional complete intersection of two quadrics.
Then, the Gromov--Witten invariants of $X$ can be effectively reconstructed from the simple nodal Gromov--Witten invariants of $X$.
\end{theorem}

\begin{proof}

We prove Theorem \ref{thm_nodes} by induction on the number of primitive insertions. 
Let $V \subset H^*(X)\otimes \C$ the primitive cohomology. For $a,b \in \Z_{\geq 0}$, let $\Gamma$ be an $X$-valued stable graph with a single vertex, no edges, and $a+b$ legs
divided in two groups labeled from $1$ to $a$ and from $1$ to $b$. We consider the Gromov--Witten theory of $X$ of type $\Gamma$ with $a$ primitive and $b$ simple insertions. More precisely, we fix integers $k_1,\dots,k_{a}\in \Z_{\geq 0}$, 
$\ell_1,\dots,\ell_b \in \Z_{\geq 0}$,
and simple classes $\beta_{1},\dots,\beta_{b}$, and we define the map 
\begin{align} \Omega \colon V^{\otimes a}& \longrightarrow H_{\star}(\oM_\Gamma(X)) \\
 \alpha_1 \otimes \dots \otimes \alpha_a &\longmapsto \prod_{i=1}^{a} \psi_i^{k_i} \ev_i^{*}(\alpha_i) \prod_{j=1}^b \psi_j^{\ell_j}\ev_j^{*}(\beta_j) \cap [\oM_\Gamma(X)]^\virt\,,
\end{align}
obtained by considering all possible primitive insertions 
$\alpha_1,\dots,\alpha_a$. Our goal is to compute~$\Omega$. 

By deformation invariance of the virtual class in Gromov--Witten theory, the class $\Omega$ is invariant under the action of the monodromy group $G$ on $V$. 
By Propositions \ref{prop_big_monodromy_odd}-\ref{prop_big_monodromy_even},
the monodromy group is 
$G=O(V)$ if 
$\dim X$ is even  and $G=\mathrm{Sp}(V)$ if $\dim X$ is odd 
(since $X$ is not a cubic surface or an even dimensional complete intersection of two quadrics). 
In particular, $\Omega=0$ if $a$ is odd by Theorems~\ref{Thm: fund inv o} and ~\ref{Thm: fund inv sp}. From now on, we assume that $a$ is even and we write $a=2n$ with $n \in \Z_{\geq 0}$. Then, $\Omega$ is an element of 
$((V^{*})^{\otimes 2n})^G \otimes H_{*}(\oM_\Gamma(X))$, which can be viewed as a subspace of 
$\C \shP_n \otimes H_{*}(\oM_\Gamma(X))$ by \eqref{eq_inv_orthogonal}\,-\,\eqref{eq_inv_symplectic}.

For every $n$-pairing $P \in \mathcal{P}_n$, let $\Gamma_P$ be the $X$-valued graph obtained from $\Gamma$ by gluing together the legs $i$ and $j$ for every pair $(i,j)$ of $P$. The graph $\Gamma_P$ contains one vertex, $n$ loops, and $b$ legs, see Figure~\ref{Fig:nloops}.

\begin{figure}
\resizebox{.9\linewidth}{!}{\input{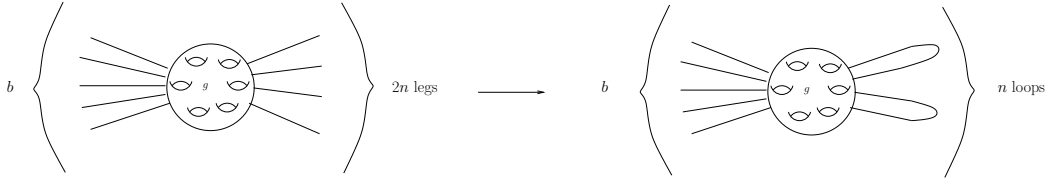}}
 \caption{Trading $2n$ primitive insertions for $n$ nodes}
\label{Fig:nloops}
\end{figure}

Using the moduli stack $\oM_{\Gamma_P}(X)$ of $\Gamma_P$-marked stable maps to $X$, we define the class 
\[\Omega_P := \prod_{i=1}^{2n} \psi_i^{k_i} \prod_{j=1}^b \psi_j^{\ell_j} \ev_j^{*}(\beta_j) 
\cap [\oM_{\Gamma_P}(X)]^{\virt} \in H_{\star}(\oM_{\Gamma_P}(X))\,.\]
Splitting the $n$ loops of $\Gamma_P$ defines a morphism 
\[ s_P \colon \oM_{\Gamma_P}(X) \rightarrow \oM_\Gamma(X) \,.\]
By the splitting formula \eqref{eq_splitting}, we have 
\begin{equation} \label{eq_sP}
s_{P,*} \Omega_P = \pm \Omega(\Delta_P) + \Omega',
\end{equation}
where the insertion of the multivector $\Delta_P$ defined in \S \ref{sec_matrix_diagonal} results from the terms in the Künneth decompositions of the diagonals containing only primitive classes, and $\Omega'$ results from the other terms in the Künneth decompositions of the diagonals and so contains at least one simple class insertion.
The class $s_* \Omega_P$ is defined by simple nodal Gromov--Witten theory of $X$, and the class $\Omega'$ is defined by Gromov--Witten theory of $X$ with 
fewer than $2n$ primitive insertions. It is therefore sufficient to show that 
$\Omega$ can be recovered from the data of the classes 
$\Omega(\Delta_P)$ for $P \in \mathcal{P}_n$.

Let $\mathsf{M}^\Delta(n,V)$ be the endomorphism of $\C \mathcal{P}_n$ defined by the matrix of diagonal insertions of Definition \ref{def_matrix_diagonal}. It follows from Definition \ref{def_matrix_diagonal} that 
\begin{equation} \label{eq_MDelta} 
\mathsf{M}^\Delta(n,V)\,  \Omega = (\Omega(\Delta_P))_{P \in \mathcal{P}_n} 
\in \C \mathcal{P}_n \otimes H_{\star}(\oM_\Gamma(X))\,.
\end{equation}
By Theorem \ref{thm_matrix_diagonal} expressing $\mathsf{M}^\Delta(n,V)$ in terms of the loop matrix, 
and Corollaries \ref{cor_1}-\ref{cor_2} on the spectral properties of the loop matrix, 
$\mathsf{M}^\Delta(n,V)$ preserves the subspace of invariants 
$((V^{*})^{\otimes 2n})^G$ and is invertible in restriction after restriction to this subspace. It follows that $\Omega$ can be recovered from $(\Omega(\Delta_P))_{P \in \mathcal{P}_n}$.

As the relations \eqref{eq_sP} and \eqref{eq_MDelta} are linear, the statement of Theorem \ref{thm_nodes} for Gromov--Witten invariants follows by taking the degrees of the classes $\Omega$
and $\Omega(\Delta_P)$.
\end{proof}


\subsection{Gromov--Witten invariants of complete intersections}
\label{section_algorithm}

We are finally ready to state the main goal of the paper:
\begin{theorem}[\textbf{Theorem \ref{thm_intro_main}}]
\label{thm_main}
Let $X$ be an $m$-dimensional smooth complete intersection 
in $\PP^{m+r}$ of degrees $(d_1,\dots, d_r)$. 
Then, for every decomposition 
$$d_r=d_{r,1}+d_{r,2}\ \ \ \text{with} \ \ \ d_{r,1}, d_{r,2} \in \Z_{\geq 1}\, ,$$
the Gromov--Witten invariants of $X$ can be effectively 
reconstructed from:
\begin{itemize}
\item[(i)] the Gromov--Witten invariants of an $m$-dimensional smooth complete intersection $X_1 \subset \PP^{m+r}$ of degrees $(d_1,\dots, d_{r-1},d_{r,1})$,
\item[(ii)] the Gromov--Witten invariants of an $m$-dimensional smooth complete intersection $X_2 \subset \PP^{m+r}$ of degrees $(d_1,\dots,d_{r-1}, d_{r,2})$,
\item[(iii)] the Gromov--Witten invariants of an 
$(m-1)$-dimensional smooth complete intersection $D \subset \PP^{m+r}$ of degrees $(d_1,\dots,d_{r-1}, d_{r,1}, d_{r,2})$,
\item[(iv)] the Gromov--Witten invariants of an
$(m-2)$-dimensional smooth complete intersection 
$Z \subset \PP^{m+r}$ of degrees $(d_1,\dots, d_{r-1},d_r,d_{r,1}, d_{r,2})$.
\end{itemize}
\end{theorem}

\begin{proof}
To prove Theorem \ref{thm_intro_main}, we consider, following 
\cite[\S 0.5.4]{MP}, 
the degeneration $X \rightsquigarrow 
X_1 \cup_D \widetilde{X}_2$
obtained by factoring the degree $d_r$ defining equation of $X$ 
into factors of degrees $d_{r,1}$ and $d_{r,2}$.
Here, 
$\widetilde{X}_2$ is the blow-up of $X_2$ along $Z$, and 
$X_1 \cup_D \widetilde{X}_2$ denotes $X_1$ and $X_2$ transversally glued along a copy of $D$.

To construct such a degeneration, we start with the equations 
\[ f_1=\dots=f_r=0 \]
of $X$, with $f_i$ of degree $d_i$, and we write a product
$g_r=f_{r,1} f_{r,2}$ of generic polynomials $f_{r,1}$ and
$f_{r,2}$ of degree $d_{r,1}$ and $d_{r,2}$ respectively.
Then, the equations 
\[ f_1=\dots=f_{r-1}=f_{r,1}= tf_r-g_r=0 \]
define a 1-parameter flat family 
$\mathcal{X}\rightarrow \AA^1$.

The fibers $\mathcal{X}_t$ for general $t \neq 0$ are
smooth complete intersections of degrees $(d_1,\dots,d_r)$ and so deformation equivalent to $X$. 
The fiber $0\in \AA^1$ is given by 
$$\mathcal{X}_0=X_1 \cup_D X_2\, ,$$ where $X_1$ is the $m$-dimensional smooth complete intersection
 with equations 
\[ f_1=\dots=f_{r-1}=f_{r,1}=0\, \]
of degrees $(d_1,\dots,d_{r-1},d_{r,1})$,
$X_2$ is the $m$-dimensional smooth complete intersection of 
with equations, 
\[ f_1=\dots=f_{r-1}=f_{r,2}=0\,\]
degrees 
$(d_1,\dots, d_{r_1}, d_{r,2})$,
and $X_1 \cup_D X_2$ is the union of $X_1$ and $X_2$ glued along their common 
intersection $D$, which is the $(m-1)$-dimensional smooth complete intersection 
of degrees $(d_1,\dots,d_{r-1}, d_{r,1}, d_{r,2})$
with equations 
\[ f_1=\dots=f_{r-1}=f_{r,1}=f_{r,2}=0\,.\]

The total space $\mathcal{X}$ of the family over $\AA^1$ is singular with singular locus $Z$, the 
$(m-2)$-dimensional smooth complete intersection of degrees $(d_1,\dots,d_{r-1}, d_r, d_{r,1}, d_{r,2})$
with equations 
\[ f_1=\dots=f_{r-1}=f_r=f_{r,1}=f_{r,2}=0\,.\]
Locally analytically and transversally to $Z$, the singularities of the total space are 
ordinary $3$-fold double points. Blowing-up $X_2$ in $\mathcal{X}$  resolves these singularities: we obtain a new degeneration over $\AA^1$ with smooth total space and special fiber over $0$ given by
$X_1 \cup_D \widetilde{X}_2$, where $\widetilde{X}_2$ is the blow-up of $X_2$ along $Z$.

\vspace{8pt}

\noindent $\bullet$ \textbf{Large monodromy:}
We first prove Theorem \ref{thm_main} when $X$ is not a cubic surface or an even dimensional complete intersection of two quadrics. 
The Gromov--Witten invariants of $X$ are then
determined by the simple nodal Gromov--Witten invariants of $X$ by Theorem \ref{thm_nodes}.
By the nodal degeneration formula of Theorem \ref{thm_degeneration_cycle}
applied to the degeneration of $X\rightsquigarrow X_1 \cup_D \widetilde{X}_2$, the simple nodal Gromov--Witten invariants of $X$ can be computed in terms of the nodal relative Gromov--Witten invariants of $(X_1,D)$ and $(\widetilde{X}_2,D)$. 

More precisely, for the latter claim, we must show that the finite sum on the left side of the degeneration formula \eqref{eq_deg_num} actually reduces to a single possibly non-zero term.
In general, the finite sum arises when the monodromy around the special fiber acts non-trivially on the curve classes of the general fiber. Then, different curve classes of the general fiber can specialize to the same curve class in the special fiber. If $\dim X=m \neq 2$, there is no such monodromy since every curve class $\beta \in H_2^+(X)$ is pulled back from the ambient $\PP^{m+r}$ by the Lefschetz hyperplane theorem. 
If $\dim X=m=2$, and $X$ is a quadric surface, one can check directly that the monodromy action on curve classes is trivial. If $\dim X=m=2$, and $X$ is not a quadric surface, then, as we are also assuming that $X$ is not a cubic surface, or a complete intersection of two quadrics, the Noether-Lefschetz theorem 
\cite{Lef}\cite[XIX, Theorem 1.2]{deligne1973groupes}
implies that the only effective curve classes on a very general deformation of $X$ are pull-back from the ambient $\PP^{m+r}$. By deformation invariance of Gromov--Witten invariants, we conclude that the only possibly non-zero Gromov--Witten invariants of $X$ are for curve classes pull-back from 
$\PP^{m+r}$.

By Theorem \ref{thm_splitting}, the nodal relative Gromov--Witten invariants of 
$(X_1,D)$ and $(\widetilde{X}_2,D)$ can be reconstructed from the Gromov--Witten invariants of $X_1$, $D$, and $\widetilde{X}_2$.
Finally, Gromov--Witten invariants of 
$\widetilde{X}_2$ are determined by the Gromov--Witten invariants of $X_2$ and $Z$
by the blow-up result of
\cite[Theorem B]{fan2017chern}.

\vspace{8pt}
\noindent $\bullet$ \textbf{Small monodromy:}
It remains to prove Theorem \ref{thm_main} when $X$ is a cubic surface or an even dimensional complete intersection of two quadrics. We claim that, in these cases, all cohomology classes of $X$ extend across the special fiber $X_1 \cup_D \widetilde{X}_2$ of the degeneration: the monodromy is trivial
and there are no vanishing cycles.
Assuming the extension claim, the result follows by the degeneration formula, the reconstruction of relative Gromov--Witten invariants of $(X_1,D)$ and $(\widetilde{X}_2,D)$ in terms of Gromov--Witten invariants of $X_1$, $D$, and $\widetilde{X}_2$ \cite[Theorem 2]{MP}, and the reconstruction of the Gromov--Witten invariants of the blow-up $\widetilde{X}_2$ in terms of the invariants of $X_2$ and $Z$ \cite[Theorem B]{fan2017chern}.

To prove the extension claim, we first note that, for $X$ a cubic surface or an even dimensional complete intersection of two quadrics, the monodromy around the special fiber of the degeneration is finite by Proposition \ref{prop_big_monodromy_even}. On the other hand, the degeneration of 
$X$ to $X_1 \cup_D \widetilde{X}_2$ is semi-stable and so has unipotent monodromy by the local monodromy theorem \cite[Theorem 1']{Landman}\cite[Corollary 3.4]{SGA7-1}. Hence, the monodromy, being both finite and unipotent, is trivial. Finally, triviality of the monodromy implies the absence of vanishing cycles by the local cycle invariant theorem \cite[\S 3]{Clemens}.
 
We give below an alternative proof of the absence of vanishing cycles by an explicit study of the topology of the degeneration $X \rightsquigarrow X_1 \cup_D \widetilde{X}_2$. The explicit study can be used to concretely extend cohomology classes across the special fiber for effective computations of Gromov--Witten invariants.

We denote by $S_{X_1}D$ and $S_{X_2}D$ the $S^1$-bundles in the normal line bundles $N_{D|X_1}$ and $N_{D|X_2}$ respectively.
The space $X$ is obtained topologically by gluing the complements of tubular neighborhoods of $D$ in $X_1$ and $\widetilde{X}_2$ by an 
orientation-reversing diffeomorphism of their boundaries $S_{X_1}D$ and $S_{X_2}D$.
Hence, we have a natural commutative diagram of Mayer-Vietoris cohomology sequences 
(as in the proof of \cite[Lemma 4.11]{TZ}):
{\small{
\[\begin{tikzcd}
H^{m-1}(D) 
\arrow[r]
\arrow[d]
& H^m(X_1 \cup_D \widetilde{X}_2) 
\arrow[r]
\arrow[d]
&
H^m(X_1) \oplus H^m(\widetilde{X}_2) 
\arrow[r]
\arrow[d]
&
H^m(D)
\arrow[d]
\\
H^{m-1}(SD) \arrow[r] 
& H^m(X) 
\arrow[r]
&
H^m(X_1 \setminus D) \oplus H^m(\widetilde{X}_2 \setminus D) 
\arrow[r]
&
H^m(SD)
\,,
\end{tikzcd}\]}}

\noindent where we denote\footnote{$S_{X_1}D \simeq S_{X_2}D$ is an isomorphism of $S^1$-fiber bundles but not of oriented (or principal) 
$S^1$-bundles as it reverses the orientation.} by
$SD$ the space $S_{X_1}D \simeq S_{X_2}D$. On the other hand, Mayer-Vietoris for the decompositions $X_1=(X_1 \setminus D)\cup D$ and $\widetilde{X}_2=(\widetilde{X}_2 \setminus D) \cup D$ gives exact sequences
\[\begin{tikzcd}
H^{m-1}(SD) 
\arrow[r]
&
H^{m}(X_1) 
\arrow[r]
& H^m(X_1 \setminus D) \oplus H^m(D) 
\arrow[r]
&
H^m(SD) 
\,,
\end{tikzcd}\]
\[\begin{tikzcd}
H^{m-1}(SD) 
\arrow[r]
&
H^{m}(\widetilde{X}_2) 
\arrow[r]
& H^m(\widetilde{X}_2 \setminus D) \oplus H^m(D) 
\arrow[r]
&
H^m(SD) 
\,.
\end{tikzcd}\]
Hence, to prove that 
$H^m(X_1 \cup_D \widetilde{X}_2) \rightarrow H^m(X)$ is surjective, it suffices to show that $H^m(SD)=0$ 

The Gysin exact sequence for $S^1$-bundles yields
\[\begin{tikzcd}
H^{m-2}(D) 
\arrow[r]
& H^m(D) 
\arrow[r]
&
H^m(SD) 
\arrow[r]
&
H^{m-1}(D) 
\,,
\end{tikzcd}\]
where $H^{m-2}(D) \rightarrow H^m(D)$ is given by the cup-product with the Euler class of $SD$. The key point is that for $X$ a cubic surface or a complete intersection of two quadrics, $D$ is a 
$(m-1)$-dimensional quadric, so for $m$ even, $H^{m-1}(D)=0$ and $H^{m-2}(D)=H^m(D)=\Q$
\cite[XI, 2.6]{deligne1973groupes}. On the other hand, as $D$ is an ample divisor of $X_1$, we also have $c_1(N_{D|X_1}) \neq 0$, so the map $H^{m-2}(SD) \rightarrow H^m(D)$ is an isomorphism. We conclude $H^m(SD)=0$, so the natural map $H^m(X_1 \cup_D \widetilde{X}_2) \rightarrow H^m(X)$ is surjective.
\end{proof}

\subsection{The algorithm} \label{algor}
We now explain how to use Theorem \ref{thm_main} to recursively compute all Gromov--Witten invariants of a smooth complete intersection $X$ in projective space.
The computation is done by induction on the dimension and the degrees. The initial cases of the induction are projective spaces, $X\simeq \PP^m$, whose 
Gromov--Witten invariants can be determined by the localization formula \cite{GP} and the calculation of Hodge integrals on the moduli spaces of curves \cite{FP}.

For the induction step, let $X$ be an $m$-dimensional smooth complete intersection in projective space of degrees $(d_1,\dots,d_r)$ such that $X \neq \PP^m$.
One can assume without loss of generality that $d_1 \geq d_2 \geq \dots \geq d_r \geq 2$.
By Theorem \ref{thm_main} applied to the decomposition $d_r=(d_r-1,1)$, the Gromov--Witten invariants of $X$ can be reconstructed from the Gromov--Witten invariants of $X_1$, $X_2$, $D$, and $Z$, which are complete intersections of dimension $m$, $m$, $m-1$, $m-2$, and degrees $(d_1,\dots,d_{r-1},d_r-1)$,
$(d_1,\dots, d_{r-1},1)$, $(d_1, \dots, d_{r-1}, d_r-1,1)$, 
$(d_1, \dots, d_{r-1},d_r, d_r-1,1)$ respectively.
The complete intersections
$X_1$ and $X_2$ have degrees strictly smaller than $X$.
The complete intersections
$D$ and $Z$ are of dimension strictly smaller than $X$.
Therefore we can apply the induction hypothesis to $X_1$, $X_2$, $D$, and $Z$.




\subsection{Tautological classes}
\label{sec_tautological}

We gave in \S \ref{sec_intro_tautological}  a brief introduction to the general question of whether Gromov--Witten classes \eqref{eq_gw_class} are tautological classes in 
$H^{\star}(\oM_{g,n})$.

\begin{theorem}[\textbf{Theorem \ref{thm_intro_tautological}}]
\label{thm_tautological}
Let $X$ be a smooth complete intersection in projective space.
Then, the Gromov--Witten classes of $X$ are tautological.
\end{theorem}

\begin{proof}
It is enough to show that the algorithm given in \S \ref{section_algorithm} - \S \ref{algor}
for Gromov--Witten invariants can be lifted to the level of Gromov--Witten classes and implies the property of being tautological.
Nodal Gromov--Witten classes and nodal relative Gromov--Witten classes are defined as in \eqref{eq_gw_class} by push-forward to the moduli space of possibly disconnected stable curves. 

We go briefly through the various ingredients used in the algorithm to discuss the lift:
\begin{itemize}
    \item[(i)] The Gromov--Witten classes of projective spaces 
are tautological by the localization formula \cite{GP} and the fact that $\lambda$ classes are tautological \cite{Mumford}.
    \item[(ii)] The proof of Theorem \ref{thm_nodes} is formulated at the level of virtual classes on the moduli spaces of stable maps and so implies a version of Theorem \ref{thm_nodes} for Gromov--Witten classes: under the assumptions of Theorem \ref{thm_nodes}, the Gromov--Witten classes of $X$ can be effectively reconstructed from the simple nodal Gromov--Witten classes of $X$, and are tautological if the simple nodal Gromov--Witten classes of $X$ are tautological.
    \item[(iii)] There is a version of Theorem \ref{thm_splitting} for Gromov--Witten classes: under the assumptions of Theorem \ref{thm_splitting}, the nodal relative Gromov--Witten classes of $(X,D)$ can be effectively reconstructed from the Gromov--Witten classes of $X$ and $D$, and are tautological if the Gromov--Witten classes of $X$ and $D$ are tautological. Indeed, the splitting formula of Theorem \ref{thm_splitting_2} and the rigidification result of Lemma \ref{lem_rigidification} are stated at the level of virtual classes on moduli spaces of stable maps and so can be applied to study Gromov--Witten classes. In addition, we need to know that the reconstruction results of \cite{MP} have a version for Gromov--Witten classes preserving the property of being tautological. This is the case because all the proofs in \cite{MP} are principal terms arguments following from the degeneration and localization formulas and so work at the level of Gromov--Witten classes. An explicit example of such principal terms arguments done at the level of Gromov--Witten classes can be found in \cite{FP1}, where it is shown that relative Gromov--Witten classes of $\PP^1$ are tautological.
    \item[(iv)] There is a version of Theorem \ref{thm_main} for Gromov--Witten classes: under the assumptions of Theorem \ref{thm_main}, the nodal Gromov--Witten classes of $X$ can be effectively reconstructed from the 
    Gromov--Witten classes of $X_1$, $X_2$, $D$, and $Z$, and are tautological if the Gromov--Witten classes of $X_1$, $X_2$, $D$, and $Z$ are tautological. Indeed, the degeneration formula of Theorem \ref{thm_degeneration_cycle} is stated at the level of virtual classes on moduli space of stable maps and so can be applied to study Gromov--Witten classes. In addition, the blow-up result of \cite{fan2017chern} also admits a version for Gromov--Witten classes preserving the property of being tautological because its proof consists of principal terms arguments based on degeneration and localization formulas, as the results of 
    \cite{MP} mentioned in (iii).
\end{itemize}

\noindent The proof using (i)-(iv), in fact, provides an
algorithm to calculate the Gromov-Witten classes of
complete intersections.
\end{proof}

\begin{remark} \label{rem_tautological_curves}
Using Theorem \ref{thm_tautological}, we can give a slightly different proof
of the result proved in \cite{Janda}
that the Gromov--Witten classes of curves are tautological. Indeed, by degeneration, it is enough to prove that the relative Gromov--Witten classes of a genus $1$ curve are tautological. By the results of \cite{MP} expressing relative Gromov--Witten theory in terms of absolute, and lifted to the level of Gromov--Witten classes
as discussed in (iii) of the proof of Theorem \ref{thm_tautological}, it is enough to show that the Gromov--Witten classes of a genus $1$ curve are tautological. But this follows from Theorem \ref{thm_tautological} because a genus $1$ curve is a cubic in 
$\PP^2$.
\end{remark}

\appendix
\section{Nodal cubics in the projective plane}
\label{sec_appendix}

We provide here an explicit example of the nodal degeneration formula and  the relative splitting formula to compute a nodal Gromov--Witten invariant.

Let $X=\PP^2$ be the complex projective plane, and let
$\Gamma$ be the $X$-valued graph consisting of a single vertex $v$ with a loop $e$ and $8$ legs, as illustrated in Figure \ref{Fig: gammae}. Furthermore, let $\g(v)=0$ and $\beta(v)=3 \in H_2^+(\PP^2)=\Z_{\geq 0}$. We first directly
compute the nodal Gromov--Witten invariant 
\begin{equation} \label{eq_gw_ex}
\left\langle \prod_{i=1}^8 \tau_0(p) \right\rangle^{\PP^2}_\Gamma\,,
\end{equation}
where $p \in H^4(\PP^2)$ is the class of a point. Afterwards, we apply the nodal degeneration and the relative splitting formula to find the same result. 

The Künneth decomposition of the class of the diagonal
 $\Delta \subset \PP^2 \times \PP^2$ is $$[\Delta]=1 \otimes p+H \otimes H+p\otimes 1\, ,$$ where $H \in H^2(\PP^2)$
is the class of a line. Hence, by the splitting formula 
\eqref{eq_splitting}, we obtain
\begin{equation*}
    \left\langle \prod_{i=1}^8 \tau_0(p) \right\rangle^{\PP^2}_\Gamma= 
    \left\langle
(\tau_0(1)\tau_0(p) 
+\tau_0(H) \tau_0(H)
+ \tau_0(p)\tau_0(1))\prod_{i=1}^8 \tau_0(p)
\right\rangle^{\PP^2}_{\Gamma\setminus e}\,.
\end{equation*}
As the graph $\Gamma \setminus e$ has no edges, the right side of the above splitting formula consists of $10$-pointed, genus $0$, degree $3$ Gromov--Witten invariants of $\PP^2$. As there are no rational cubic passing through $9$ points in general position in $\PP^2$, the contribution from the terms involving $\tau_0(1)\tau_0(p)$ is zero. After applying the divisor equation to remove the two insertions of $\tau_0(H)$ and using the general fact that there are 12 rational cubics passing through $8$ points in general position in $\PP^2$, we obtain 
\begin{equation*}
\left\langle \prod_{i=1}^8 \tau_0(p) \right\rangle^{\PP^2}_\Gamma
=  3 \times 3 \times 12 = 108  \,.\end{equation*}

\begin{figure}
\resizebox{.7\linewidth}{!}{\input{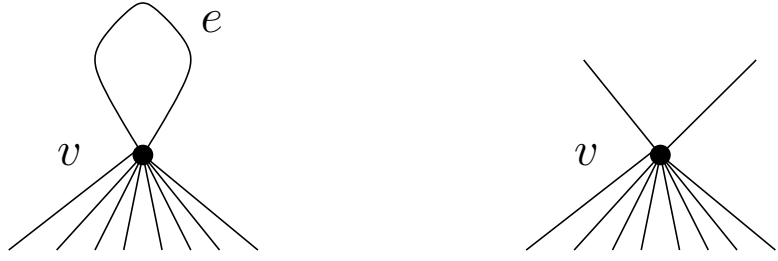}}
\caption{The graph $\Gamma$ on the left and $\Gamma \setminus e$ on the right.}
\label{Fig: gammae}
\end{figure}

\begin{figure}
\resizebox{.7\linewidth}{!}{\input{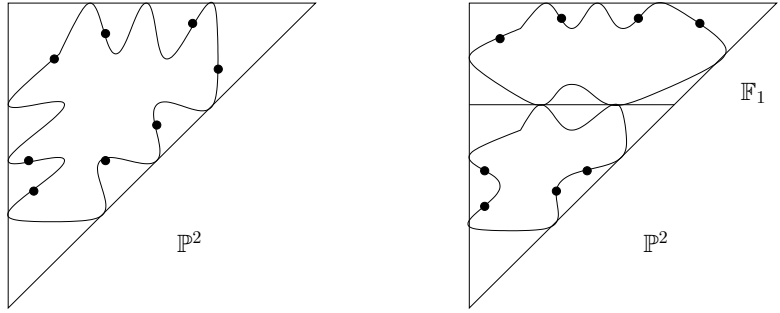}}
\caption{A cubic in $\PP^2$ degenerating to the union of a conic in $\PP^2$ and a curve of class $D_0+3F$ in 
$\mathbb{F}_1$.}
\label{Fig: p2f1}
\end{figure}

We consider next the degeneration of $\PP^2$ to the normal cone of a line $L$, so that the special fiber is given by the union of $\PP^2$ and the Hirzebruch surface $$\mathbb{F}_1:=\PP(\cO_{\PP^1}\oplus \cO_{\PP^1}(1))\, $$ formed by identifying $L\subset \PP^2$ with the section $D_0$ of $\mathbb{F}_1$ with self-intersection $-1$. 
We distribute $4$ of the $8$ point constraints on the general fiber to the $\PP^2$-component, while the other $4$ point constraints are sent to the $\mathbb{F}_1$-component of the special fiber.

In the chosen degeneration, the curves in 
the special fiber which contribute to the nodal degeneration formula can only decompose as a (possibly singular) conic in the $\PP^2$-component and a curve of class $D_0+3F$ in $\mathbb{F}_1$, as illustrated in Figure \ref{Fig: p2f1}. 
Indeed, on the $\PP^2$ component, we have $4$ point constraints, and there is no line through $4$ general points. Moreover, if we had a cubic in $\PP^2$, then 
on the $\FF_1$ component, we would have a curve whose class is $3F$. However, 
there is no curve of class $3F$ passing through $4$ general points in $\FF_1$. 
Hence, the curves in the $\PP^2$-component of the special fiber contributing to the nodal degeneration formula are necessarily conics. Then, the curves in the $\mathbb{F}_1$-component 
must be be of class $D_0+3F$, since $(D_0+3F) \cdot D_0= 2$ and $(D_0+3F) \cdot F= 1$. 
\begin{figure}
\resizebox{.9\linewidth}{!}{\input{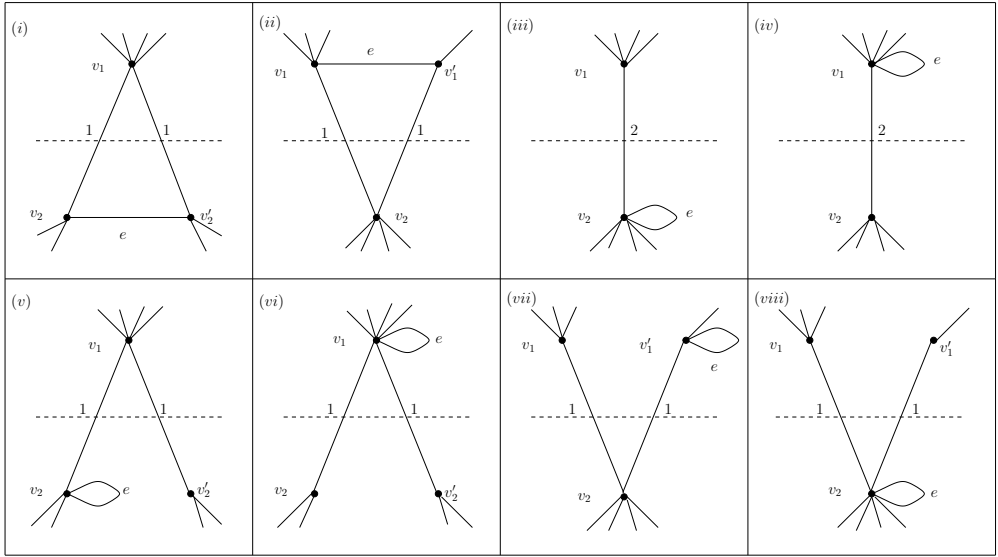}}
\caption{All splittings $\sigma=(\gamma_1,\gamma_2)$ contributing to the nodal degeneration formula. The graphs above the dotted lines are the $(\mathbb{F}_1,D_0)$-valued stable graphs denoted by $\gamma_1$, and below are the $(\PP^2,L)$-valued stable graphs denoted by $\gamma_2$.}
\label{Fig: graphs}
\end{figure}

We display in Figure \ref{Fig: graphs} all possible splittings $\sigma=(\gamma_1,\gamma_2)$ of $\Gamma$ contributing to the nodal degeneration formula. 
The contributions $N_\sigma$ associated to each $\sigma$ is 
computed below using \eqref{eq_deg_num} of Theorem
\ref{thm_degeneration_cycle}:

\begin{itemize}
    \item[(i)] Let $\sigma=(\gamma_1,\gamma_2)$ be the splitting of $\Gamma$ as illustrated in Figure \ref{Fig: graphs}(i) with $\beta(v_1)=D_0+3F$ and $\beta(v_2)=\beta(v_2')=1$. We have 
    \[ N_\sigma =1 \times 2 \times 3 \times 1=6 \,,\]
    where the first factor is $m(\sigma)=1$, the factor $2$ comes from the ordering of the half-edges of $e$, the factor $3$ is the number of (unordered) pairs of lines in the pencil of conics passing through $4$ general points in $\PP^2$, and the last factor is the number of curves of class $D_0+3F$ passing through $6$ general points (the $4$ points in the interior of $\mathbb{F}_1$ and the $2$ boundary points on $D_0$ which 
    are fixed by choosing a pair of lines on the side of $\PP^2$).
      
      \item[(ii)] Let $\sigma=(\gamma_1,\gamma_2)$ be the splitting of $\Gamma$ as illustrated in Figure \ref{Fig: graphs}(ii) with $\beta(v_1)=D_0+2F$,
      $\beta(v_1')=F$, and $\beta(v_2)=2$. We have 
       \[ N_\sigma= 1 \times 2 \times 5 \times 1=10\,,\]
         where the first factor is $m(\sigma)=1$, the factor $2$ comes from the ordering of the half-edges of $e$, the factor $5$ is the number of reducible curves with two components in the pencil of curves of class $D_0+3F$ passing through $5$ given points in $\mathbb{F}_1$, and the last factor $1$ is the number of conics passing through $5$ given points in $\PP^2$. The factor $5$ can be computed as follows:
         the total space of the pencil of curves of class $D_0+3F$ and passing through $5$ points in $\mathbb{F}_1$ is the blow-up $\mathrm{Bl}_5(\FF_1)$ of $\FF_1$ in $5$ points, which has Euler characteristic $9$. On the other hand, a $\PP^1$-fibration over $\PP^1$ with $k$ reducible fibers with two components has Euler characteristic $$(2-k)\times 2+k \times 3=k+4\, .$$ The equation $k+4=9$ implies $k=5$.
         
           \item[(iii)] Let $\sigma=(\gamma_1,\gamma_2)$ be the splitting of $\Gamma$ as illustrated in Figure \ref{Fig: graphs}(iii) with $\beta(v_1)=D_0+3F$ and $\beta(v_2)=2$. We have 
           \[ N_\sigma= 2  \times 4 \times 2\times 1=16 \,,\]
   where the first factor is $m(\sigma)=2$, the factor $4$ comes from the relative splitting formula of Theorem \ref{thm_splitting_3} to compute the nodal invariant of type $\gamma_2$ in $\PP^2$ (the factor $4$ comes from the divisor equation applied to remove the only contributing term $H \otimes H$ of the diagonal of $\PP^2$ in the relative splitting formula{\footnote{The correction term is zero as the contact point with $D_0$ is already fixed, whereas the insertion of the diagonal 
    $\Delta_{D_0}=1 \otimes p + p \otimes 1$ would also impose a constraint on the contact point.}}), the factor $2$ is the number of conics through $4$ points in $\PP^2$ tangent to a given line, and the last factor $1$ is the number of curves of class $D_0+3F$
    passing through $4$ points in $\mathbb{F}_1$ and tangent to $D_0$ at a fixed point. 
  
  \item[(iv)]  Let $\sigma=(\gamma_1,\gamma_2)$ be the splitting of $\Gamma$ as illustrated in Figure \ref{Fig: graphs}(iv) with $\beta(v_1)=D_0+3F$ and $\beta(v_2)=2$. We have   
     \[ N_\sigma = 2 \times 2  \times 5 \times 1 =20 \,,\]
     where the first factor is $m(\sigma)=2$, the second factor $2$ is the number of conics through $4$ points in $\PP^2$ tangent to a given line, the factor $5$ comes from the relative splitting formula of Theorem \ref{thm_splitting_3} to compute the nodal invariant of type $\gamma_1$ in $\mathbb{F}_1$ (the factor $5$ comes from the divisor equation{\footnote{We have $(D_0+3F)\cdot F=1$, 
     $(D_0+3F)\cdot (D_0+F)=3$, and $(D_0+3F) \cdot D_0=2$.}} 
     applied to remove the only contributing terms 
     $F \otimes (D_0+F) + D_0 \otimes F$ of the diagonal of 
     $\mathbb{F}_1$). The last factor $1$ is the number of curves of class $D_0+3F$ passing through $4$ general points and tangent to $D_0$ at a fixed point.
    
     \item[(v)]  Let $\sigma=(\gamma_1,\gamma_2)$ be the splitting of $\Gamma$ as illustrated in Figure \ref{Fig: graphs}(v) with $\beta(v_1)=D_0+3F$  and $\beta(v_2)=\beta(v_2')=1$, and let $\sigma'$ be the same splitting except that the loop is on the vertex $v_2'$. We have
     \[ N_\sigma+N_{\sigma'}= 1 \times 3 \times (1+1)  \times 1  = 6\,,\]
      where the first factor is $m(\sigma)=1$, the factor $3$ is the number of (unordered) pairs of lines in the pencil of conics passing through $4$ points in $\PP^2$, the factor $(1+1)$ comes from the relative splitting formula applied to the loop $e$ adjacent to either $v_2$ or $v_2'$ (the factor $1$ comes from the divisor equation applied to remove the only contributing term $H \otimes H$ of the diagonal of $\PP^2$), and the last factor $1$ is the number of curves of class $D_0+3F$ passing through $6$ given points.
    
         \item[(vi)]  Let $\sigma=(\gamma_1,\gamma_2)$ be the splitting of $\Gamma$ as illustrated in Figure \ref{Fig: graphs}(vi) with $\beta(v_1)=D_0+3F$, and $\beta(v_2)=\beta(v_2')=1$. We have  
           \[ N_\sigma = 1 \times 3  \times 5 \times 1 =15\, ,\] 
        obtained by proceeding analogously as in (v), except that now the loop is on $v_2$. So, the term $F \otimes (D_0+F) +D_0\otimes F$ of the diagonal of $\mathbb{F}_1$ contributes to the divisor equation (as in (iv)).
        
          \item[(vii)]  Let $\sigma=(\gamma_1,\gamma_2)$ be the splitting of $\Gamma$ as illustrated in Figure \ref{Fig: graphs}(vii) with $\beta(v_1)=D_0+2F$,
      $\beta(v_1')=F$, and $\beta(v_2)=2$ as in (ii), and let $\sigma'$ be the same splitting except that the loop is on the vertex $v_1$. We have 
        \[ N_\sigma +N_{\sigma'}= 1 \times 5  \times (3+0) \times 1= 15\,,\]
      where the first factor is $m(\sigma)=1$,  the factor $5$ is the number of reducible curves with two components in the pencil of curves of class $D_0+3F$ passing through $5$ given points in $\mathbb{F}_1$ (as in (ii)), the factor $(3+0)$ comes from the diagonal of $\mathbb{F}_1$ in the divisor equation, with either $3$ or $0$ depending on $e$ adjacent to $v_1$ or $v_1'$, and the last factor $1$ is the number of conics through $5$ points in $\PP^2$.
       
       \item[(viii)]  Let $\sigma=(\gamma_1,\gamma_2)$ be the splitting of $\Gamma$ as illustrated in Figure \ref{Fig: graphs}(vii) with $\beta(v_1)=D_0+2F$,
      $\beta(v_1')=F$ and $\beta(v_2)=2$. We have   
         \[ N_\sigma=1 \times 5  \times 4 \times 1 =20\,,\]
         computed as in (vii), except that the loop is on $v_2$, and so it is the diagonal of $\PP^2$ which contributes to the divisor equation.
    \end{itemize}
    
\noindent After summing the above $8$ contributions, we obtain
\[ 6+10+16+20+6+15+15+20=108\,,\]
as expected.

The above example clearly illustrates why it is essential for the nodal degeneration formula and the relative splitting formula to work with nodal relative invariants defined using the moduli stack $\shP_\gamma(X,D)$, where the imposed nodes are not permitted to lie in the singular locus of the expanded target, and not with the full moduli stack 
$\oM_\gamma(X,D)$. 

For instance, in case (iii), we have on the $\PP^2$-component the moduli stack of genus $1$ degree $2$ relative stable maps to $(\PP^2,L)$ tangent to $L$ and passing through $4$ given general points. A general element of this moduli space consists of one of the two conics $C$ in $\PP^2$ tangent to $L$ and passing through the $4$ given points, to which is attached a genus $1$ component contracted to a point $p \in C$.
This is a legitimate relative stable map as long as $p$ is away from the tangent point of $C$ with $L$, as illustrated on the left of Figure \ref{Fig: 3figs}. When $p$ goes to the tangent point, the target expands, and the genus $1$ curve contracts to
a degree $2$ cover of the $\PP^1$-fiber of the bubble over the tangent point, fully ramified along the two sections of the bubble.
Another possibility is that the genus $1$ curve itself maps as a degree 2 cover of
the $\PP^1$-fiber,
as illustrated in the middle of Figure \ref{Fig: 3figs}. There is a 1-parameter family of such degree 2 covers (parameterized by the relative position on 
$\PP^1$ of the two extra ramification points of the cover). When the ramification point goes to a section of the bubble, and a second bubble appears, creating the curve illustrated on the right of Figure \ref{Fig: 3figs}. In particular, this curve has 2 non-separating nodes mapping to the singular locus of the expanded target.

Hence, $\oM_{\gamma_2}(\PP^2,L)$ is the disjoint union of $\shP_{\gamma_2}(\PP^2,L)$ 
with a set of isolated maps corresponding to geometry of the rightmost configuration in  
Figure \ref{Fig: 3figs}. A similar issue arises in case (iv). 
Including these prohibited maps would produce an error of 16 in the nodal
degeneration formula.


It is interesting to note that if we interpret \eqref{eq_gw_ex} as a genus $1$
Gromov--Witten invariant with insertion of the class $2\delta_0=24 \lambda_1$
and apply the ordinary degeneration formula with insertion of $24 \lambda_1$, 
as done in 
\cite{Bou} more generally with insertion of $(-1)^g \lambda_g$, then these extra curves do contribute $16$. However, with the insertion of $\lambda_1$, the cases (i) and (ii) contributing $16$ in our discussion do not contribute because $\lambda_1$ vanishes on families of curves containing a cycle, and so we obtain the same total answer, as expected.

\begin{figure}
\resizebox{.9\linewidth}{!}{\input{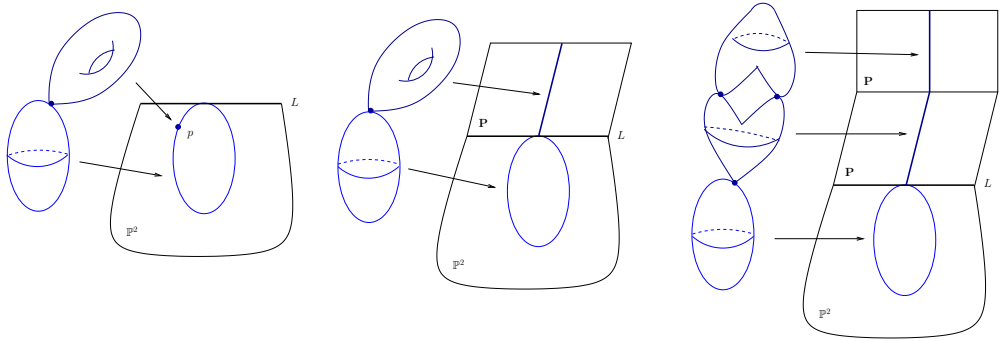}}
\caption{Degree $2$ genus $1$ relative stable maps to $(\PP^2,L)$}
\label{Fig: 3figs}
\end{figure}

\FloatBarrier

\bibliographystyle{plain}
\bibliography{bibliography}

\end{document}